\documentclass[12pt]{amsart}

\usepackage{amssymb}		
\usepackage{graphicx}		
\usepackage{hyperref}		
\usepackage{amsmath}
\usepackage{array}
\usepackage{parskip}
\usepackage{amsthm}

\usepackage{amsfonts}
\usepackage{tikz-cd}

\usepackage{mathrsfs}
\usepackage{breqn}
\usepackage{xcolor}
\usepackage[mathscr]{eucal}
\usepackage{enumitem}
\newtheorem{defn}{Definition}
\newtheorem{lem}{Lemma}
\newtheorem{prop}{Proposition}
\newtheorem{thm}{Theorem}
\newtheorem{cor}{Corollary}
\newtheorem{remark}{Remark}
\newtheorem{ex}{Example}
\newtheorem{Q}{Question}
\newtheorem{proc contact}{Procedure Contact}
\newtheorem{conj}{Conjecture}

\newcommand{\partiald}[2]{\displaystyle\frac{\partial#1}{\partial#2}}
\newcommand{\partialds}[3]{\displaystyle\frac{\partial^2#1}{\partial#2\partial#3}}

\newcommand{\ann}{\mathrm{ann}\,}
\newcommand{\mcal}[1]{\mathcal{#1}}
\newcommand{\wh}[1]{\widehat{#1}}
\newcommand{\vel}[1]{\mathrm{vel}(\mathcal{#1})}
\newcommand{\dec}[1]{\mathrm{deccel}(\mathcal{#1})}
\newcommand{\der}[2]{\mathcal{#1}^{(#2)}}
\newcommand{\coder}[2]{{#1}^{(#2)}}
\newcommand{\Chare}[1]{\mathrm{Char}\,{\mathcal{#1}}}
\newcommand{\Char}[2]{\mathrm{Char}\,{\mathcal{#1}}^{(#2)}}
\newcommand{\hChar}[2]{\mathrm{Char}\,{\widehat{\mathcal{#1}}}^{(#2)}}
\newcommand{\inChar}[2]{\mathrm{Char}\,{\mathcal{#1}}^{(#2)}_{#2-1}}
\newcommand{\inCharOne}[1]{\mathrm{Char}\,{\mathcal{#1}}^{(1)}_{0}}
\newcommand{\hinCharOne}[1]{\mathrm{Char}\,{\widehat{\mathcal{#1}}}^{(1)}_{0}}

\newcommand{\rank}{\mathrm{rank}\,}
\newcommand{\tab}{\hspace*{2em}}

\newcommand{\sD}{\mathscr{D}} 
\newcommand{\sE}{\mathscr{E}}
\newcommand{\sL}{\mathscr{L}}

\newcommand{\bA}{\textbf{A}} 
\newcommand{\bB}{\textbf{B}}

\newcommand{\bc}{\textbf{c}}
\newcommand{\bff}{\textbf{f}}
\newcommand{\bg}{\textbf{g}}
\newcommand{\bh}{\textbf{h}}
\newcommand{\bU}{\textbf{U}}
\newcommand{\bu}{\textbf{u}}

\newcommand{\bv}{\textbf{v}}
\newcommand{\bw}{\textbf{w}}
\newcommand{\bX}{\textbf{X}}
\newcommand{\bx}{\textbf{x}}
\newcommand{\bxd}{\textbf{\.x}}

\newcommand{\by}{\textbf{y}}
\newcommand{\byd}{\textbf{\.y}}

\newcommand{\bz}{\textbf{z}}
\newcommand{\bzd}{\textbf{\.z}}

\everymath{\displaystyle}

\makeatletter
\renewcommand{\tocsection}[3]{%
  \indentlabel{\@ifnotempty{#2}{\bfseries\ignorespaces#1 #2\quad}}\bfseries#3}
\renewcommand{\tocsubsection}[3]{%
  \indentlabel{\@ifnotempty{#2}{\ignorespaces#1 #2\quad}}#3}

\newcommand\@dotsep{4.5}
\def\@tocline#1#2#3#4#5#6#7{\relax
  \ifnum #1>\c@tocdepth 
  \else
    \par \addpenalty\@secpenalty\addvspace{#2}%
    \begingroup \hyphenpenalty\@M
    \@ifempty{#4}{%
      \@tempdima\csname r@tocindent\number#1\endcsname\relax
    }{%
      \@tempdima#4\relax
    }%
    \parindent\z@ \leftskip#3\relax \advance\leftskip\@tempdima\relax
    \rightskip\@pnumwidth plus1em \parfillskip-\@pnumwidth
    #5\leavevmode\hskip-\@tempdima{#6}\nobreak
    \leaders\hbox{$\m@th\mkern \@dotsep mu\hbox{.}\mkern \@dotsep mu$}\hfill
    \nobreak
    \hbox to\@pnumwidth{\@tocpagenum{\ifnum#1=1\bfseries\fi#7}}\par
    \nobreak
    \endgroup
  \fi}
\AtBeginDocument{%
\expandafter\renewcommand\csname r@tocindent0\endcsname{0pt}
}
\def\l@subsection{\@tocline{2}{0pt}{2.5pc}{5pc}{}}
\makeatother

\begin{document}

\title{Geometry of Cascade Feedback Linearizable Control Systems}

\author{Taylor J. Klotz}
\address{Department of Mathematics, 395 UCB, University of
Colorado, Boulder, CO 80309-0395}
\email{Taylor.Klotz@colorado.edu}


\begin{abstract}
In this thesis, we provide new insights into the theory of cascade feedback linearization of control systems. In particular, we present a new explicit class of cascade feedback linearizable control systems, as well as a new obstruction to the existence of a cascade feedback linearization for a given invariant control system. These theorems are presented in Chapter 4, where truncated versions of operators from the calculus of variations are introduced and explored to prove these new results. This connection reveals new geometry behind cascade feedback linearization and establishes a foundation for future exciting work on the subject with important consequences for dynamic feedback linearization. 
\end{abstract} 

\maketitle

\tableofcontents

\section{Brief Control Theory Background\label{chap:one}}
\subsection{Introduction}
Control theory, simply defined, is an area of research in mathematics and engineering that is fundamentally driven by the question: how can one control the result of some mechanical, electrical, biological, or other process in an acceptable way? For example, autonmous driving and parallel parking, optimal power consumption for super computers, blood sugar levels in humans, confined plasma devices etc. It is an incredibly vague and daunting question that is almost always dependent on the particular process being considered. Nevertheless, there have been many important mathematical and engineering insights that have led to general techniques, as well as the recognition of types of control systems where said techniques may apply. In the course of researching this question, more specific lines of inquiry begin to arise, such as: can one optimally control the process subject to some constraint? How sensitive is the process to the choice of controls? Can the process be controlled even in the presence of random noise? and many more. This thesis is essentially concerned with a particular variant of the question: what types of seemingly nonlinear processes are actually linear in some--potentially expanded--sense? When this is the case, the full array of well developed techniques in linear control theory are suddenly applicable. Even this question has already accumulated a deep literature, some of which we will introduce in the remainder of Chapter 1. For a reasonably detailed overview of the big ideas and history of control theory, the reader is encouraged to peruse \cite{ControlOverviewHistory}. 

This thesis is in the area of geometric control theory, specifically from the viewpoint of differential geometry. There are many papers from this perspective in control theory, and we will be particularly concerned with results that answer questions of equivalence between control systems; that is, questions related to classifying control systems. Chapter 2 will address the key concepts from differential geometry, as they apply to geometric control theory, that are relevant to this thesis. These concepts include topics such as: distributions, exterior differential systems, derived systems, jet bundles, generalized Goursat bundles, etc. 
Chapter 3 is specifically concerned with the role of symmetry from the perspective of exterior differential systems, and provides a key property needed to define cascade feedback linearizable systems. Chapter 4 opens with the definition of a \textit{cascade feeback linearizable control system}, and then presents the primary results of this thesis, including a new necessary condition for cascade feedback linearizable control systems, as well as the presentation of a new explicit class of such systems.  

For more on geometric control theory, good places to begin are \cite{BulloLewisBook},\cite{JurdjevicBook}, and \cite{GeometricControlBoulder}. 

\subsection{Control Systems}
In this section we will define explicitly what we mean by a \textit{control system} for the purposes of this thesis, as well as introduce several examples that will appear throughout.
\begin{defn}\label{control system def 1}
Let $M$ be a manifold such that $M\cong_{\mathrm{loc}}\mathbb{R}\times\mathbb{R}^n\times\mathbb{R}^m$, with coordinates $(t,\bx,\bu)$, where $\bx=(x^1,\ldots,x^n)$ and $\bu=(u^1,\ldots,u^m)$. A \textbf{control system} on $M$ is an underdetermined system of ordinary differential equations,
\begin{equation}\label{control system}
\frac{d\bx}{dt}=\bff(t,\bx,\bu),
\end{equation}
where $\bff(t,\bx,\bu)=(f^1(t,\bx,\bu),\ldots,f^n(t,\bx,\bu))$. The coordinate $t$ will denote time, and the variables $\bx$ and $\bu$ are the \textbf{state variables} and \textbf{control variables} respectively. Additionally, denote $\bX(M)\cong\mathbb{R}^n$ to be the \textbf{state space} of $M$ with the states $\bx$ as local coordinates on $\bX(M)$. 
\end{defn}
\begin{defn}
A \textbf{solution} or \textbf{trajectory} of a control system is any curve $(t, \bx(t),\bu(t))$ in $M\cong_{\mathrm{loc}}\mathbb{R}\times\mathbb{R}^n\times\mathbb{R}^m$ that satisfies equation (\ref{control system}). 
\end{defn}
As alluded to in the introduction, we mention that, in general, a \textit{control system} could refer to many different types of differential equations or processes, e.g. PDEs, SDEs, discrete DEs, general stochastic processes, etc. The author is curious to know the extent to which the ideas in this thesis may be applied to other types of control systems, and will likely investigate this to some degree in the future. For now, however, we will be content with our above definition of a control system. 

One important property of control systems is the question of controllability. 
\begin{defn}
A control system is \textbf{controllable} if, for any two points $p$ and $q$ in $\bX(M)$, there exists a solution to (\ref{control system}) such that $\bx(t_0)=p$ and $\bx(t_1)=q$. 
\end{defn}
Studying controllability of control systems is of central importance in the overall field of control theory, and there are many different types of controllability and related notions. We will not explore this topic any further in this thesis, except briefly in Section 1.3 and in the following example. One can refer to \cite{BulloLewisBook} for more on controllability. We now list several examples of control systems that appear throughout this thesis. 
\begin{ex}\label{lin 1}
\begin{equation}
\begin{aligned}
\,&\,&\dot{x}^1&=x^2,&\dot{x}^2&=u^2,\,&\,&\,\\
\,&\,&\dot{x}^3&=u^1,&\dot{x}^4&=u^2.\,&\,&\,
\end{aligned}
\end{equation}
This control system has 4 states and 2 controls. 
\end{ex}
Solutions for this control system passing through the point $(\bx_0=\bx(t_0))\in \bX(M)\cong \mathbb{R}^4$ are easily seen to be given by 
\begin{equation}
\begin{aligned}
\,&\,& x^1&= f_2(t),&x^2&=\dot{f}_2(t),\,&\,&\,\\
\,&\,&x^3&=f_1(t),&x^4&=\dot{f}_2(t) +C,\,&\,&\,
\end{aligned}
\end{equation}
where $C=x^4(t_0)-x^2(t_0)$. We then notice the algebraic constraint $x^4(t)=x^2(t)+C$ for any choice of $f_2(t)$ and for all $t$. Thus, this control system is not controllable. All remaining examples introduced in this section are controllable. 
\begin{ex}\label{so(5)}\cite{Cascade2}
The following is a control system of 3 states and 2 controls.
\begin{equation}
\begin{aligned}
\dot{x}_1&=\frac{1}{2}(x_2+2x_3x_5),& \dot{x}_2&=2(x_3+x_1x_5),\\
\dot{x}_3&=\frac{2(u_1-x_1u_2)}{1+x_1},&\dot{x}_4&=x_5\\
\dot{x}_5&=\frac{2(u_1+u_2)}{1+x_1}.&\,
\end{aligned}
\end{equation}
\end{ex}
This example first appears in \cite{Cascade2}, and will have importance as an illustration of the main results in Chapter 4. It has the property of being cascade feedback linearizable, as shown in \cite{Cascade2}, and in particular, it serves as an example of Theorem \ref{sufficiency} in Chapter 4. 
\begin{ex}\label{HSM chap1}\cite{HuntSuMeyerLin}\cite{GSalgorithmExample}
\begin{equation}
\begin{aligned}
\frac{dx^1}{dt}&=\sin(x^2),\,&\frac{dx^2}{dt}&=\sin(x^3),\,&\frac{dx^3}{dt}&=(x^4)^3+u^1,&\,\\
\frac{dx^4}{dt}&=x^5+(x^4)^3-(x^1)^{10},\,&\frac{dx^5}{dt}&=u^2.\,&\,&\,&\,&\,
\end{aligned}
\end{equation}
\end{ex}
Example \ref{HSM chap1} above is an example of a control system that appears to be a nonlinear system of underdetermined ODE. However, the control system is equivalent to a \textit{linear} system in a precise way to be defined in Sections 1.2 and 3.7.
\begin{ex}\label{BC chap1}
\begin{equation}
\begin{aligned}
\,&\frac{dx^1}{dt}=u^1,\,&\frac{dx^2}{dt}=x^1,\,&\,\phantom{==}&\frac{dx^3}{dt}=(x^2+x^6+x^2u^1),&\,\\
\,&\frac{dx^4}{dt}=(u^2+x^1u^3),\,&\frac{dx^5}{dt}=x^4,\,&\,\phantom{=}&\frac{dx^6}{dt}=(x^5+x^2x^4),&\,\\
\,&\frac{dx^7}{dt}=u^3.\,&\,&\,&\,&\,
\end{aligned}
\end{equation}
\end{ex}
The above system in Example \ref{BC chap1} will be referred to as the BC system, for Battilotti and Califano, who introduced the system in \cite{BC3control}. This system is also cascade feedback linearizable, as will be demonstrated in Chapter 4.  

\begin{ex}\label{reduction example}
\begin{equation}
\begin{aligned}
\frac{dz^1_0}{dt}&=z^1_1,&\frac{dz^2_0}{dt}&=z^2_1,\,&\,\\
\frac{dz^1_1}{dt}&=z^1_2,&\frac{dz^2_1}{dt}&=z^2_2,\,&\,\\
\frac{d\epsilon}{dt}&=e^{z^1_1z^2_0}.&\,&\,&\,&\,
\end{aligned}
\end{equation}
\end{ex}
This relatively simple looking control system will be used only once, and will not reappear until Chapter 4. The control parameters are $z^1_2$ and $z^2_2$. It provides a nice demonstration of the necessary condition found in Theorem \ref{necessity}, which is one of the main results of this thesis. 

\begin{ex}\label{affine sym}
\begin{equation}
\begin{aligned}
\frac{dx^1}{dt}&=((x^2)^2+x^1f(t,x^3,x^4,x^5,u^2)),\\
\frac{dx^2}{dt}&=x^2f(t,x^3,x^4,x^5,u^2),\\
\frac{dx^i}{dt}&=g^i(t,x^3,x^4,x^5,u^2)\left(x^2e^{-u^1}\right)^{a_i}\,\,\mathrm{ for }\,\,3\leq i\leq 5
\end{aligned}
\end{equation}
where $f,g$ are arbitrary functions and the $a_i$ are constants, not all zero. 
\end{ex}
The above family of control systems possesses a familiar set of symmetries--the affine transformations of the real plane. Specific choices of the functions $g^i$ and the constants $a_i$ will be used to demonstrate various theorems in Chapter 3 and Chapter 4. 

\begin{ex}\label{PVTOL}\cite{VTOL}
\begin{equation}
\begin{aligned}
\ddot{x}&=-u_1\,\sin(\theta)+h\,u_2\,\cos(\theta), \\ 
\ddot{z}&=u_1\,\cos(\theta)+h\,u_2\,\sin(\theta)-g, \\
\ddot{\theta}&=\lambda u_2.
\end{aligned} 
\end{equation}
\end{ex}
One important example to mention, which will not be explored in this thesis, is that of the planar vertical take-off and landing vehicle, (PVTOL) control system, listed above. An in-depth analysis of the system regarding cascade feedback linearization will appear in a later work. 

\begin{ex}\label{Sluis}
\begin{equation}\label{Sluis system}
\begin{aligned}
\dot{x}^1&=c_1x^1+c_3x^3+u^1(a_0+a_1x^1+a_3x^3+a_4x^4),\\
\dot{x}^2&=e_1x^1+e_3x^3+u^2(b_0+b_3x^3+b_4x^4),\\
\dot{x}^3&=u^1,\\
\dot{x}^4&=u^2.
\end{aligned}
\end{equation}
\end{ex}
Finally, Example \ref{Sluis} is an 11-parameter family of control systems. Any choice of the parameters leads to a system that is \textit{not} linearizable (see Definition \ref{ESFL def 1}). However, this system does have the property of being linearized when additional differential equations are imposed. This concept will be made precise in Section 1.3. 

All of the examples presented in this section have very different properties, and in terms of classification, are all inequivalent to one another. 
\subsection{Linear Control Systems}
The most fundamental and well studied class of control systems are those that are linear.
\begin{defn}
A control system (\ref{control system}) in $n$ states and $m$ controls is \textbf{linear} if it has the form 
\begin{equation}\label{linear control system}
\bxd(t)=A\bx+B\bu,
\end{equation}
where $A$ and $B$ are $n\times n$ and $n\times m$ constant matrices respectively. 
\end{defn}
In particular, Example \ref{lin 1} is a linear control system. We demonstrated that it was not controllable, which naturally makes one wonder about when a linear control system is controllable.
\begin{thm}(Kalman Condition)\cite{KalmanCondition}
A linear control system (\ref{linear control system}) is controllable if and only if the $n\times nm$ matrix 
\begin{equation}
[B\,AB\,A^2B\,\cdots\,A^{n-1}B],
\end{equation}
has rank $n$.
\end{thm}
We will restrict ourselves to controllable systems from here on. Given a control system, it may be possible to change the state and control variables in such a way that the new system is a linear control system. In particular, a theorem of Brunovsk\'y \cite{Brunovsky} says that \text{all} controllable linear control systems may be put into the following form by a specific type of transformation.  
\begin{defn}\cite{Brunovsky}
The \textbf{Brunovsk\'y normal form}, is a linear control system (\ref{linear control system}) such that matrix $A$ consists of $\sigma_i\times \sigma_i$ block matrices $A_i$, $1\leq i\leq m$ down the diagonal, with the form
\begin{equation}
A_i=
\begin{bmatrix}
0& 1 &0&\cdots& 0\\
0& 0 & 1 &\cdots& 0\\
\vdots&\vdots&\vdots&\ddots&\vdots\\
0& 0 & 0 &\cdots& 1\\
0& 0 & 0 &\cdots& 0
\end{bmatrix}, 
\end{equation}
and the matrix $B$ has entries $B^l_k=1$ if $l=k=\sum_{i=1}^s\sigma_i$ for all $1\leq s\leq m$, and $B^l_k=0$ otherwise. That is, the matrix $B$ has a  one in a diagonal position on row $k$ if the $k$th row of $A$ is all zeroes, and the rest of the entries for $B$ are zero.  
\end{defn}
Sometimes it is possible to transform a seemingly nonlinear control system to Brunovsk\'y normal form via a change of coordinates. Indeed, in \cite{HuntSuMeyerLin} and \cite{GSalgorithmExample} Example (\ref{HSM chap1}) was shown to be equivalent to a Brunovsk\'y normal form via a change of coordinates where the new state variables are of the form $\bz=\bz(\bx)$ and the new control variables have the form $\bv=\bv(\bx,\bu)$. 
\begin{defn}\label{ESFL def 1}
A control system (\ref{control system}) is called \textbf{static feedback linearizable} (SFL) if there is an invertible map $(t,\bz,\bv)=(t,\varphi(\bx),\psi(\bx,\bu))$ such that (\ref{control system}) transforms to a Brunovsk\'y normal form $\bzd=A\bz+B\bv$. The control system is called \textbf{extended static feedback linearizable} (ESFL) if the map has the form $(t,\bz,\bv)=(t,\varphi(t,\bx),\psi(t,\bx,\bu))$.   
\end{defn}
Notice that the forms of the maps only take state variables to state variables, while the new controls are allowed to depend on the old controls \textit{and} old state variables. This is what is meant by ``feedback". For example, when one is driving a car, the current position of the car is used to determine how to change the steering wheel or acceleration to stay on the road. However, a driver has no way of controlling the shape or orientation of the road in order to keep the car on the road. This property is important for having meaningful solutions for a control system.  

The first results concerning when a given \textit{nonlinear} control system is SFL were given by Krener in \cite{KrenerLin}, as well as by Brockett in \cite{BrockettLin} and then Jakubczyk and Respondek in \cite{RespondekLin}. Constructing explicit maps for SFL systems is harder, and that work was started by Hunt, Su, and Meyer in \cite{HuntSuMeyerLin}, and then a more geometric approach based on symmetry was developed in \cite{BrunovskySymmetry} by Gardner, Shadwick, and Wilkens, and finally work of Gardner and Shadwick \cite{GSalgorithm},\cite{GSFeedback}, and \cite{GSalgorithmExample} provided what is now known as the GS algorithm for static feedback linearization. The work of Vassiliou in \cite{VassiliouGoursat} and \cite{VassiliouGoursatEfficient} provides a way to construct the required maps for ESFL systems, as well as identifying when a given control system is ESFL. It is Vassiliou's work that will be central to this thesis. The main results of the two previously mentioned papers will appear in Chapter 2. 
\subsection{Dynamic Feedback Linearizable Control Systems} 
A particularly desirable property for a control system is when solutions can be written purely in terms of arbitrary function and their derivatives. Much like the case of Example \ref{lin 1}, determining a solution curve requires no integration, and involves only algebraic expressions of arbitrary functions and their derivatives. Solutions to Brunovsk\'y linear control systems always have this property. However, there are nonlinear systems that may also have this property and are \textit{not} SFL. 
\begin{defn}
A controllable control system is called \textbf{explicitly integrable} (EI) if generic solutions may be written as
\begin{equation}
\bx(t)=\bA(t,z^i_0(t),z^i_1(t),\ldots,z^i_{s_i}(t)),\,\,\bu(t)=\bB(t,z^i_0(t),z^i_1(t),\ldots,z^i_{r_i}(t)),
\end{equation}
for $1\leq i\leq m$ for $m$ the number of controls and $z^i_{l_i}(t)=\frac{d^{l_i}z^i_0}{dt^{l_i}}$ for some arbitrary smooth functions $z^i_0(t)$. Additionally, we may add the distinction of \textbf{autonomous} to an EI system if $A$ and $B$ have time dependence only through the functions $z^i_0(t)$ and their derivatives. That is, $A$ and $B$ have the form
\begin{equation}
\bx(t)=\bA(z^i_0(t),z^i_1(t),\ldots,z^i_{s_i}(t)),\,\,\bu(t)=\bB(z^i_0(t),z^i_1(t),\ldots,z^i_{r_i}(t)).
\end{equation}
\end{defn}
It turns out that EI systems are related to another type of linearization called \textit{dynamic feedback linearization}. 
\begin{defn}
A control system (\ref{control system}) is \textbf{dynamic feedback linearizable} (DFL) if there exists an augmented system of the form
\begin{equation}
\begin{aligned}
\bxd&=\bff\,(\bx,\bu),\phantom{\bw==}\bx\in\mathbb{R}^n,\phantom{==} \bu\in\mathbb{R}^m,\\
\byd&=\bg\,(\bx,\by,\bw),\phantom{==}\by\in\mathbb{R}^k, \phantom{==}\bw\in\mathbb{R}^q,\\
\bu&=\bh\,(\bx,\by,\bw),
\end{aligned}
\end{equation}
such that the control system
\begin{equation}
\begin{aligned}
\bxd&=\bff\,(\bx,\bh(\bx,\by,\bw)),\\
\byd&=\bg\,(\bx,\by,\bw),
\end{aligned}
\end{equation}
is SFL. 
\end{defn}
There has been considerable effort to understand DFL systems, more than we can exhaustively list here. The concept first appears in \cite{SinghDFL}, and was subsequently studied in \cite{IsidoriMoogLucaDFL}, \cite{CharletLevineMarino1}, and \cite{CharletLevineMarino2}. A geometric necessary condition based on ruled submanifolds was presented in \cite{sluis1994absolute}. A method for producing a DFL, if it exists, was the subject of work by Battilotti and Califano in \cite{BCDFL1},\cite{BCDFL2}, and \cite{BC3control}. However, a complete classification of DFL systems has yet to be achieved. 

The 11-parameter family from Example \ref{Sluis} was shown to be DFL in \cite{sluis1994absolute} by differentiating twice along the control $u^1$. That is, if we augment (\ref{Sluis system}) by 
\begin{equation}
\begin{aligned}
\dot{y}^1&=y^2,\\
\dot{y}^2&=v^1,\\
u^1&=y^1,
\end{aligned}
\end{equation}
then the augmented system is SFL. 
We also have the following nonautonomous version of a dynamic feedback linearizable system. 
\begin{defn}
A control system (\ref{control system}) is \textbf{extended dynamic feedback linearizable} (EDFL) if there exists an augmented system of the form
\begin{equation}
\begin{aligned}
\bxd&=\bff\,(t,\bx,\bu), \bx\in\mathbb{R}^n, \bu\in\mathbb{R}^m,\\
\byd&=\bg\,(t,\bx,\by,\bw), \by\in\mathbb{R}^k, \bw\in\mathbb{R}^q,\\
\bu&=\bh\,(t,\bx,\by,\bw),
\end{aligned}
\end{equation}
such that the control system
\begin{equation}
\begin{aligned}
\bxd&=\bff\,(t,\bx,\bh(t,\bx,\by,\bw)),\\
\byd&=\bg\,(t,\bx,\by,\bw),
\end{aligned}
\end{equation}
is ESFL. 
\end{defn}
In forthcoming work \cite{Cascade2}, it is shown that 
\begin{prop}\label{EI iff EDFL}
A control system is EDFL if and only if it EI. Additionally, a control system is DFL if and only if it is an autonomous EI system.  
\end{prop}
Example \ref{BC chap1} is DFL, and in fact, in Chapter 4 of this thesis, we will show that Example \ref{BC chap1} is EI. In order to prove this, the theory of \textit{cascade feedback linearization} (CFL) is introduced and applied in Chapter 4. The key idea is the existence of two particular kinds of ESFL systems whose trajectories may be ``composed" in order to compute trajectories of Example \ref{BC chap1}. CFL systems are introduced in \cite{VassiliouCascadeI} where it is shown that such systems are EI. In light of this result and Proposition \ref{EI iff EDFL} from \cite{Cascade2}, we can say that any CFL system is EDFL. In particular, as is shown in \cite{Cascade2}, carrying out the CFL process tells one how to construct a simple augmented system that is ESFL, therefore demonstrating directly that a CFL system is EDFL. 

We would also like to remark briefly on Example \ref{PVTOL} mentioned in Section 2. The control system is shown to be DFL by Martin, Devasia, and Paden in \cite{DynamicVTOL}, and in \cite{VassiliouCascadeI} it is shown to be CFL. However, the sizes of the augmented systems differ from the two constructions, namely the CFL construction in \cite{VassiliouCascadeI} requires a larger augmented system. In another forthcoming work by the author, Example \ref{PVTOL} will be explored through the lens of CFL theory more closely, demonstrating, in particular, that there is an augmented system of the same size as that presented in \cite{DynamicVTOL}. 


\section{Geometry of Feedback Transformations and Linearization\label{chap:two}}

\subsection{Exterior Differential Systems, Distributions, and Derived Systems}
In this section we provide some background on exterior differential systems (EDS) and derived systems. For a comprehensive account of EDS, refer to \cite{BCGGG} and \cite{CartanBeginners}. Throughout this thesis, we will assume that the ranks of all bundles that appear are constant on sufficiently small open sets unless otherwise specified. 
\begin{defn} An \textbf{Exterior Differential System (EDS)} is an ideal $\mathcal{I}$ in the exterior algebra of differential forms on a manifold $M$ that satisfies the condition $d\mathcal{I}\subset\mathcal{I}$, where $d$ is the exterior derivative. 
\end{defn}
We will always consider the case that an EDS is finitely generated as an ideal. We have two ways of generating an EDS: algebraically or differentially. That is, 
\begin{align}
\mathcal{I}&=\langle \theta^a, d\theta^a \rangle_{\mathrm{alg}},\quad 1\leq a\leq k,\\
\mathcal{I}&=\langle \theta^a \rangle_{\mathrm{diff}}
\end{align}
where $k$ is positive integer, each $\theta^a$ is a differential form on $M$, and $\langle \theta^a\rangle_{\mathrm{diff}}:=\langle \theta^a, d\theta^a\rangle_{\mathrm{alg}}$. For shorthand, we will often drop the ``diff" subscript so that $\langle \theta^a\rangle=\langle \theta^a\rangle_{\mathrm{diff}}$. 
An important question about a given EDS is whether or not it admits integral manifolds. 
\begin{defn}
Let $f:N\to M$ be an injective immersion of a manifold $N$ into $M$. Then $f(N)$ is an integral manifold of the EDS $\mathcal{I}$ if $f^*\phi=0$ for all $\phi\in\mathcal{I}$. 
\end{defn}
A straightforward example is the case of integral curves of a nowhere vanishing vector field $X$ on a manifold $M^n$. Let $\{\theta^a\}_{a=1}^{n-1}$ span the space of all 1-forms $\psi$ on $M$ such that $\psi(X)=0$. Then the set of integral manifolds of the EDS $\mathcal{I}=\langle \theta^a\rangle$ contains the integral curves of $X$. If one considers the space $L(X)=\text{Span}_{C^\infty(M)}\{X\}$, then the set of all integral curves of vectors in $L(X)$ are in 1-1 correspondence with integral manifolds of $\langle \theta^a\rangle$.

Sometimes it is desirable to find an $m$-dimensional integral manifold $f:N\to M$ of $\mathcal{I}$ such that $f^*\Omega\neq0$ for a given $m$-form $\Omega=dx^1\wedge\cdots\wedge dx^m$ on $M$, where $\{x^i\}_{i=1}^m$ form part of a local coordinate system given by $(x^1,\ldots,x^m,y^{m+1},\ldots,y^n)$. When this is the case, $\Omega$ is called an \textit{independence condition}, and it plays the role of establishing independent variables for integral manifolds of the EDS. The requirement that $f^*\Omega\neq0$ is equivalent to claiming that $(x^1,\ldots,x^m)$ may be chosen as local coordinates for $N$. Then the integral manifold $f:N\to M$ may be thought of as a graph given by $(\bx,\by(\bx))$ where $\bx=(x^1,\ldots,x^m)$ and $\by=(y^{m+1},\ldots,y^n)$. Returning to the example of a vector field $X$, it may be possible to pick a 1-form $\Omega=dt$ such that integral curves to $X$ (and therefore its associated EDS) may be written locally as graphs $(t,x^1(t),\ldots,x^{n-1}(t))$, where $(t,x^1,\ldots,x^{n-1})$ form coordinates for $M$.
\begin{defn}
Let $\{\theta^a\}_{a=1}^r$ and $\{\theta^a,\omega^i\}_{a,i=1}^{r,m}$ be bases for sections of subbundles $I,J\subset T^*M$ respectively. An EDS $\mathcal{I}=\langle \theta^1,\ldots,\theta^r \rangle$, is called a \textbf{Pfaffian system}. We say that $I$ generates $\mcal{I}$, and write $\mcal{I}=\langle I\rangle$. If in addition, $\mcal{I}$ is given an independence condition $\Omega=\omega^1\wedge\ldots\wedge\omega^m$ such that $d\theta^a\equiv 0\,\mathrm{mod}\,J$ for all $1\leq a\leq r$, then $(\mcal{I},\Omega)$ is called a \textbf{linear Pfaffian system}. 
\end{defn}
\begin{defn}
Let $\mcal{I}$ and $\tilde{\mcal{I}}$ be two Pfaffian systems generated by the subbundles $I$ and $\tilde{I}$ of $T^*M$, respectively. Then the sum of two Pfaffian systems is defined to be
\begin{equation}
\mcal{I}+\tilde{\mcal{I}}:=\langle I + \tilde{I}\rangle.
\end{equation} 
Additionally, if $I\cap \tilde{I}$ is trivial, then the direct sum of two Pfaffian systems is the Pfaffian system
 \begin{equation}
\mcal{I}\oplus\tilde{\mcal{I}}:=\langle I\oplus \tilde{I}\rangle.
\end{equation}
\end{defn}
The EDS in this thesis will be either Pfaffian or linear Pfaffian systems. Since our systems will be Pfaffian, we will often formulate results using the dual notion of \textit{distributions}.
\begin{defn}
A \textbf{distribution} $\mcal{V}$ on a manifold $M$ is a subbundle of the tangent bundle $TM$. An \textbf{integral manifold} of a distribution $\mcal{V}$ is any submanifold $N$ of $M$ such that $TN$ is a subbundle of $\mcal{V}$.  
\end{defn}
We will denote distributions by a set of sections of $TM$ that generate the distribution by $C^\infty$ linear combinations. That is, if $\mcal{V}$ is a distribution of rank $s$, then
\begin{equation}
\mcal{V}=\{X_1,\ldots,X_s\},
\end{equation}
where the $X_i$, for $1\leq i\leq s$, are linearly independent sections of $\mcal{V}\subset TM$.
In the case that we have a Pfaffian system, there is a natural distribution whose integral manifolds are the same as those of the Pfaffian system. \begin{defn} Let $I\subset T^*M$ be a subbundle. Then the \textbf{annihilator} of $I$ is the subbundle of $TM$ given by
\begin{equation}
\ann I=\bigcup_{p\in M}\{X_p\in T_pM\colon\theta_p\left(X_p\right)=0, \forall\,\theta_p\in I_p\}. 
\end{equation}
Conversely, given a distribution $\mcal{V}$ on manifold M,
\begin{equation}
\ann \mcal{V}=\bigcup_{p\in M}\{\theta_p\in T^*M\colon \theta_p(X_p)=0, \forall\,X_p\in\mcal{V}_p\},
\end{equation}
is a subbundle of the cotangent bundle of $M$. Moreover, $\ann\left(\ann B\right)=B$ for any subbundle $B$ of $T^*M$ or $TM$. 
\end{defn}
We now discuss two equivalent versions of an important theorem in the study of EDS and distributions.  
\begin{thm}\label{Frobenius}The following are equivalent statements of the Frobenius theorem. 
\begin{enumerate}
\item Let $\mcal{I}=\langle \theta^1,\ldots,\theta^{n-r}\rangle$ be a rank $n-r$ Pfaffian system on manifold $M^n$ with $r< n$. If
\begin{equation}\label{fro condition}
d\theta^a=\alpha^a_b\wedge \theta^b,
\end{equation}
where $\alpha^a_b\in \Omega^1(M)$ for all $1\leq a,b\leq n-r$, then through any point $p\in M$ there exists an $r$-dimensional integral manifold of $\mcal{I}$ containing $p$. Furthermore, on a sufficiently small open neighborhood of $p\in M$, there exists a coordinate system $(y^1,\ldots,y^{n-r}, x^{n-r+1},\ldots,x^n)$ such that 
\begin{equation}
\mathcal{I}=\langle dy^1,\ldots,dy^{n-r}\rangle
\end{equation}
and integral manifolds are determined by the equations $y^1=c^1,\ldots,y^{n-r}=c^{n-1},$ where $c^a,\quad1\leq a\leq n-r$ are constants. 
\item Let $\mcal{V}$ be a distribution of rank $r$ on a manifold $M^n$. If $[X,Y]\in \Gamma(\mcal{V})$ for all $X,Y\in \Gamma(\mcal{V})$, then through any point $p\in M$ there exists an $r$-dimensional integral manifold of $\mcal{V}$.
\end{enumerate}
\end{thm}
The condition (\ref{fro condition}) may also be stated as
\begin{equation}
d\theta^a \equiv\, 0\,\text{mod}\,I,
\end{equation}
for all $1\leq a\leq n-r$. Condition (\ref{fro condition}) is also equivalent to saying that $\mcal{I}$ is algebraically generated by 1-forms, i.e.
\begin{equation}
\mcal{I}=\langle \theta^1,\ldots,\theta^{n-r}\rangle_{\mathrm{alg}}.  
\end{equation}
\begin{defn}
Pfaffian systems and distributions that satisfy the hypotheses of the Frobenius theorem are called \textbf{Frobenius} or \textbf{completely integrable}. 
\end{defn}
Certainly, not all Pfaffian systems are Frobenius, and indeed, one might be interested in measuring how far a system deviates from being completely integrable. One can do this by removing all the forms in $\mathcal{I}$ that obstruct the EDS from being Frobenius. This is the idea of the derived system. 
\begin{defn}\label{derived definition}
Let $\mathcal{I}$ be a Pfaffian system. Then the Pfaffian system generated by 
\begin{equation}
I^{(1)}=\{\theta\in\mathcal{I}\cap\Omega^1(M):d\theta\equiv 0\, \mathrm{mod}\, \,\mathcal{I}\cap\Omega^1(M)\}
\end{equation}
is called the \textbf{first derived system} or \textbf{derivation} of $\mathcal{I}$. If one starts with a distribution $\mathcal{V}$, then the first derived system is defined as
\begin{equation}
\mathcal{V}^{(1)}=\{Z\in \Gamma(TM)\colon Z=\sum_{i}a^i[X_i,Y_i],\,\, X_i,Y_i\in\Gamma(\mcal{V})\,,a^i\in C^\infty(M)\, \mathrm{or}\,Z\in \Gamma(\mcal{V})\}.
\end{equation}
Informally, we will denote the derived system as
\begin{equation}
\der{V}{1}=\mcal{V}+[\mcal{V},\mcal{V}]. 
\end{equation}
\end{defn}
Notice that if $\coder{I}{1}=I$ then $\mcal{I}$ generates a Frobenius Pfaffian system. Similarly, if $\der{V}{1}=\mcal{V}$ then $\mcal{V}$ is Frobenius. Furthermore, the derived system of a Pfaffian system is a diffeomorphism invariant since pullback commutes with exterior differentiation and the wedge product. Additionally, we mention that if a distribution and Pfaffian system are related by $\mcal{V}=\ann I$, then $\der{V}{1}=\ann \coder{I}{1}$. This fact follows from the identity
\begin{equation}
d\theta(X,Y)=X(\theta(Y))-Y(\theta(X))-\theta([X,Y])
\end{equation}
for any $X,Y\in \Gamma(TM),\,\theta\in \Omega^1(M)$.
 As there is a first derived system, one can repeat the constructions in Definition \ref{derived definition} to generate a \textit{second} derived system, and so on. 
\begin{defn}
Let $\mathcal{I}=\langle I\rangle$ for $I$ a subbundle of $T^*M$ and let $\mathcal{V}$ be a distribution. Then the $l$\textbf{th derived system} of $\mcal{I}$ is 
\begin{equation}
\coder{I}{l}=\{\theta\in\coder{I}{l-1}\colon d\theta\equiv 0\,\mathrm{mod}\,\coder{I}{l-1}\}.
\end{equation}
Similarly for a distribution $\mcal{V}$, 
\begin{equation}
\der{V}{l}=\der{V}{l-1}+[\der{V}{l-1},\der{V}{l-1}].
\end{equation}
The \textbf{derived flag} of a Pfaffian system is given by 
\begin{equation}
I^{(k)}\subset  I^{(k-1)}\subset I^{(k-2)}\subset\ldots\subset I^{(1)}\subset I\subseteq T^*M, 
\end{equation}
where $k$ is the smallest integer such that $I^{(k+1)}=I^{(k)}$. For a distribution $\mcal{V}$, the derived flag is
\begin{equation}
\mathcal{V}\subset \mathcal{V}^{(1)}\subset\ldots\subset\mathcal{V}^{(k-1)}\subset \mathcal{V}^{(k)}\subseteq TM,
\end{equation}
where $k$ is the smallest integer such that $\der{V}{k+1}=\der{V}{k}$. The integer $k$ is called the \textbf{derived length} of the EDS/distribution. 
\end{defn}
Additionally, we have the following integer invariants of a derived flag. 
\begin{defn}\label{accel def}\cite{VassiliouGoursat}
Let $\mathcal{V}$ be a distribution with derived length $k>1$. Let $m_i$ denote the rank of $\der{V}{i}$ for $0\leq i\leq k$. Then one can define the following lists of integers: 
\begin{enumerate}
\item The \textbf{velocity} of $\mathcal{V}$: given by $\vel{V}=\langle \Delta_1,\ldots,\Delta_k\rangle$ where $\Delta_i=m_i-m_{i-1}$ for $1\leq i\leq k$.
\item The \textbf{acceleration} of $\mathcal{V}$: given by $\mathrm{accel}(\mathcal{V})=\langle \Delta^2_2,\ldots,\Delta^2_k, \Delta_k\rangle $ where $\Delta^2_i=\Delta_i-\Delta_{i-1}$ for $2\leq i\leq k$.
\item The \textbf{deceleration} of $\mathcal{V}$: given by $\dec{V}=\langle -\Delta^2_2,\ldots,-\Delta^2_k,\Delta_k\rangle$.  
\end{enumerate}
\end{defn}
The last bundle in the derived flag is always Frobenius. We see that, for distributions, if $\der{V}{k+1}=\der{V}{k}$ then $[\der{V}{k},\der{V}{k}]\subseteq \der{V}{k}$, which is exactly the condition needed to apply the Frobenius theorem. In the case of a Pfaffian system, $\coder{I}{k+1}=\coder{I}{k}$ means that $d\theta\equiv\,0\mod \coder{I}{k}$ \text{for all} $\theta\in \coder{I}{k}$. Hence the Frobenius condition is satisfied. 
\begin{defn}
Given a Pfaffian system $\mathcal{I}=\langle I\rangle$, where $I$ is a subbundle of $T^*M$, the \textbf{first integrals} or \textbf{invariant functions} of $\mathcal{I}$ are all non-constant functions $f:M\to\mathbb{R}$ such that $df\in \Gamma\left(I\right)$. Given a distribution $\mcal{V}$, the first integrals are given by all non-constant functions $f$ such that $X(f)=0$ for all $X\in \mcal{V}$. 
\end{defn}
Consider a completely integrable Pfaffian system $\mcal{I}$ of rank $n-r$ on a manifold $M^n$. Then the Frobenius theorem says there is a coordinate system $(y^1,\ldots,y^{n-r},x^1,\ldots,x^r)$ where the coordinate functions $\{y^1,\ldots,y^{n-r}\}$ can be chosen so that $\mcal{I}$ is generated by $I=\{dy^1,\ldots,dy^{n-r}\}\subseteq T^*M$. Hence the coordinate functions $y^1,\ldots,y^{n-r}$ are first integrals of $\mcal{I}$. Furthermore, in this coordinate system, any other first integral $F$ of $\mcal{I}$ must be of the form $F(y^1,\ldots,y^{n-r})$. Indeed, if $F(\bx,\by)$ is a first integral, then $dF\wedge \Omega_y=0$, where $\Omega_y=dy^1\wedge\cdots\wedge dy^{n-r}\neq0$. However, this means that $\partiald{F}{x^i}dx^i\wedge\Omega_y=0$ for all $1\leq i\leq r$, and thus $F$ has no dependence on $x^i$ for all $1\leq i\leq r$.
\begin{defn} 
If the derived flag of a system terminates in the zero ideal for an EDS (or is the entire tangent bundle for a distribution), then there are no first integrals of the system. We call such systems \textbf{completely non-integrable}.
\end{defn}
A classic example of a completely non-integrable system is the EDS generated by a contact form on $\mathbb{R}^3$. Indeed, if
\begin{equation}
\mathcal{I}=\langle dy-z\,dx\rangle, 
\end{equation}
then it is clear that $d(dy-z\,dx)=-dz\wedge dx$. This 2-form is not zero modulo $dy-z\,dx$. Thus the first derived system is the zero ideal, and therefore this EDS has no first integrals. 
 When we consider control systems as linear Pfaffian systems, we will also assume that such systems are completely non-integrable. This is a necessary condition for a control system to be controllable. In fact, the contact system on $\mathbb{R}^3$ is a simple example of a control system with 1 control and 1 state, where $x$ is the independent variable. The contact system has only curves as integral manifolds, and such curves are given in coordinates by $(x,f(x),f'(x))$ where $f:\mathbb{R}\to\mathbb{R}$.

Another important bundle that is associated to a given EDS/distribution is the Cauchy bundle. 
\begin{defn}
The \textbf{Cauchy bundle} of an EDS $\mathcal{I}=\langle I\rangle$ is 
\begin{equation}
\mathrm{Char}\,(\mathcal{I})=\{X\in TM: \iota_X\psi\in\mathcal{I}\,\,\mathrm{for}\,\,\mathrm{all}\,\,\psi\in\mathcal{I}\},
\end{equation}
or in the language of distributions with $\mathcal{V}=\mathrm{ann}\,I$, 
\begin{equation}
\mathrm{Char}\,\mathcal{V}=\{X\in \mathcal{V}: [X,\mathcal{V}]\subseteq\mathcal{V}\}.
\end{equation}
We call sections of $\Chare{V}$ \textbf{Cauchy characteristics}. 
\end{defn}
Note that $\Chare{\mcal{V}}$ is integrable. This follows directly from the Jacobi identity on Lie brackets.

Let $X$ be any Cauchy characteristic for $\mcal{I}$ and let $N$ be any $m$-dimensional integral manifold of $\mathcal{I}$ on $M^n$ that is transverse to $X$. Consider the family of submanifolds $N_s=\varphi_s(N)$, where $\varphi_s$ is the 1-parameter family of diffeomorphisms generated by the flow of $X$. Each submanifold in the family $N_s$ is a an integral manifold of $\mcal{I}$, and moreover, the manifold $\bigcup_{s}N_s$ is an $m+1$ dimensional integral manifold of $\mcal{I}$. Hence, knowledge of the Cauchy bundle of an EDS can be used to construct more integral manifolds to the EDS. Additionally, we will always assume that the rank of any Cauchy bundle that appears in this thesis is constant. 

\begin{defn}\label{EDS symmetry}
Let $\mathcal{I}$ be an EDS. Then a vector field $X$ is an \textbf{infinitesimal symmetry of} $\mathcal{I}$ if $\mathcal{L}_X\psi\in\mathcal{I}$ for all $\psi\in\mathcal{I}$, where $\mcal{L}_X$ is the Lie derivative in the Lie derivative in the direction of $X$. 
\end{defn}
Cauchy characteristic vector fields turn out to be a special type of infinitesimal symmetry of an EDS. In general, the flow generated by an infinitesimal symmetry of an EDS will take integral manifolds to integral manifolds. However, we may not be able to construct higher dimensional integral manifolds as in the case of transverse Cauchy characteristics. In this thesis, symmetry plays a particularly important role and will be the main subject in Chapter 2. 

We will frequently use a diffeomorphism invariant from \cite{VassiliouGoursat} to identify particular types of distributions. 
\begin{defn}\label{refined derived type}\cite{VassiliouGoursatEfficient}
Let $\mcal{V}$ be a distribution with derived length $k>1$. Let
\begin{equation}
\inChar{V}{i}=\der{V}{i-1}\cap\Char{V}{i},\, 1\leq i\leq k,
\end{equation}
$m_i=\dim \der{V}{i}$, $\chi^i=\dim \Char{V}{i}$, and $\chi^i_{i-1}=\dim \inChar{V}{i}$. Then the \textbf{refined derived type} of $\mcal{V}$ is 
\begin{equation}
\mathfrak{d}_r(\mcal{V})=[[m_0,\chi^0],[m_1,\chi^1_0,\chi^1],\ldots,[m_{k-1},\chi^{k-1}_{k-2},\chi^{k-1}],[m_k,\chi^k]].
\end{equation}
\end{defn}

\subsection{The Resolvent Bundle}
A particularly special structure that appears in the study of generalized Goursat bundles (see section 2.6 below) is that of the resolvent bundle. In this section we present the definition of a resolvent bundle as well as important theorems from \cite{VassiliouGoursat}. 
Let $\mathcal{V}$ be a subbundle of $TM$ and consider the map
\begin{align}
\sigma:\mathcal{V}\to\mathrm{Hom}(\mathcal{V},TM/\mathcal{V}),\\
\sigma(X)(Y)=[X,Y]\mod\mcal{V}. 
\end{align}
The kernel of this map is exactly the Cauchy bundle of $\mathcal{V}$. 
\begin{defn}\label{singular polar}\cite{VassiliouGoursat}
For each $x\in M$, let 
\begin{equation}
\mathcal{S}_x=\{v\in \Gamma(\mathcal{V})\big|_x\backslash0\mid\,\sigma(v)\text{ has less than generic rank for all } y \text{ in a neighborhood of }x\}.
\end{equation}
Then the \textbf{singular variety} of $\mathcal{V}$ is the bundle
\begin{equation}
\mathrm{Sing}(\mathcal{V})=\coprod_{x\in M}\mathcal{S}_x.
\end{equation}
Additionally, for $X\in \mathcal{V}$, any matrix representation of the homomorphism $\sigma(X)$ is called a \textbf{polar matrix} of $[X]\in \mathbb{P}\mathcal{V}$. 
\end{defn}
Given some $[X]\in\mathbb{P}\mathcal{V}$, the map $\text{deg}_{\mathcal{V}}:\mathbb{P}\mathcal{V}\to\mathbb{N}$ is called the \textit{degree} of $[X]$ and is defined by 
\begin{equation}
\mathrm{deg}_{\mathcal{V}}\left([X]\right)=\mathrm{rank}\,\sigma(X),\text{ for }[X]\in\mathbb{P}\mathcal{V}.
\end{equation}
Note that for $X\in\mathrm{Char}\mathcal{V}$, $\text{deg}_{\mathcal{V}}\left([X]\right)=0$. For this reason, we consider the quotient $\pi:TM\to TM/\mathrm{Char}\mathcal{V}$ and denote all quotient objects by an overbar, so that $\overline{TM}=TM/\mathrm{Char}\mathcal{V}$. 
\begin{defn}\cite{VassiliouGoursat}
Let $\mathcal{V}$ be a subbundle of $TM$ of rank $c+q+1,\,q\geq2,\,c\geq0$ and $\dim{M}=c+2q+1$. Assume $\mathcal{V}$ has the properties:
\begin{enumerate}
\item $\dim\mathrm{Char}\mathcal{V}=c,\,\mathcal{V}^{(1)}=TM$,
\item $\bar{\Sigma}:=\mathrm{Sing}(\bar{\mathcal{V}})=\mathbb{P}\bar{\mathcal{B}}$, where $\bar{\mcal{B}}$ is some rank $q$ subbundle of $\bar{\mcal{V}}$. 
\end{enumerate}
Then we call $(\mathcal{V},\bar{\Sigma})$ a \textbf{Weber structure} on $M$. Furthermore, for a Weber structure $(\mathcal{V},\bar{\Sigma})$, let $\mathcal{R}_{\bar{\Sigma}}(\mathcal{V})$ denote the rank $q+c$ subbundle of $\mathcal{V}$ such that 
\begin{equation}
\pi(\mathcal{R}_{\bar{\Sigma}}(\mathcal{V}))=\bar{\mathcal{B}},
\end{equation}
where $\pi:TM\to TM/\Chare{\mcal{V}}$. We call $\mathcal{R}_{\bar{\Sigma}}(\mathcal{V})$ the \textbf{resolvent bundle} of the Weber structure $(\mathcal{V},\bar{\Sigma})$. Additionally, $(\mcal{V},\bar{\Sigma})$ is an \textbf{integrable Weber structure} if its resolvent bundle is integrable. 
\end{defn} 
\begin{prop}\cite{VassiliouGoursat}
Let $(\mathcal{V},\bar{\Sigma})$ be an integrable Weber structure on $M$. Then $\mathcal{R}_{\bar{\Sigma}}(\mathcal{V})$ is the unique, maximal, integrable subbundle of $\mathcal{V}$. 
\end{prop}

\subsection{Control Systems as Geometric Objects}
We want to express Definition \ref{control system def 1} in the language of Pfaffian systems and distributions so that we can better explore the underlying geometry of control systems.
\begin{defn}\label{geometric control system}
Let $M$ be a manifold such that $M\cong_{\mathrm{loc}}\mathbb{R}\times \bX(M)\times\bU(M)$, where $\mathbb{R}$ has $t$ as a local coordinate, $\bX(M)$ is a manifold of dimension $n$ with local coordinates $\bx=(x^1,\ldots,x^n)$, and $\bU(M)$ is a manifold of dimension $m$ with local coordinates $\bu=(u^1,\ldots,u^m)$. Then a control system on $M$ with $n$ states and $m$ controls is given by the rank $n$ linear Pfaffian system 
\begin{equation}\label{control pfaff}
\omega=\langle dx^1-f^1(t,\bx,\bu)\,dt,\ldots,dx^n-f^n(t,\bx,\bu)\,dt \rangle, 
\end{equation}
with independence condition $dt$. In the language of distributions, a control system is given by the rank $m+1$ distribution $\mcal{V}=\ann\omega$, which in local coordinates is given by
\begin{equation}\label{control dist}
\mcal{V}=\{\partial_t+f^1(t,\bx,\bu)\,\partial_{x^1}+\cdots+f^n(t,\bx,\bu)\,\partial_{x^n},\partial_{u^1},\ldots,\partial_{u^m}\}. 
\end{equation}
Additionally, we require that the Cauchy bundles of $\omega$ and $\mcal{V}$ be trivial.  
\end{defn}
The last part of this definition is found in the definition of Sluis in \cite{sluis1994absolute} on page 34/35. Control systems must be underdetermined ODE systems. That is, the differential equations (\ref{control system}) must have nontrivial dependence on \textit{all} the specified control parameters in a way that is not redundant. This is the reason for the Cauchy bundle condition. 

For control systems, we would like to pick the controls $u^1,\ldots,u^m$ to be functions of $t$ so that a corresponding solution of the control system is the graph of a curve in $M$ that passes through two desired points in $\bX(M)$. If we assume that a control system is controllable, then it follows that as a Pfaffian system or distribution, the control system needs to be completely nonintegrable; otherwise integral curves would be ``stuck" in submanifolds that foliate $M$. To see this in a bit more detail, assume that there is some $k$ such that $\coder{\omega}{k+1}=\coder{\omega}{k}$ is nontrivial. That is, the derived flag of $\omega$ terminates in a nontrivial Frobenius system at step $k$. Any integral curve of $\omega$ is also an integral curve of $\coder{\omega}{k}$. Let $\gamma(t)=(t,\bx(t), \bu(t))$ be such an integral curve, and let $\coder{\omega}{k}$ be generated by the exact 1-forms $dy^1,\ldots,dy^s$, so that locally we have a new coordinate system $(t,y^1,\ldots,y^s,z^{s+1},\ldots,z^{n+m})$. Note that neither $dt$ nor the $du^a$ may be in $\coder{\omega}{k}$; otherwise they would belong to $\omega$ as well, and this is prohibited by the independence condition and the Cauchy characterstic condtion, respectively. In these new coordinates, the curve $\gamma(t)$ becomes $\tilde{\gamma}(t)=(t,\by(t),\bz(t))$. However, since $\gamma(t)$ is an integral curve of $\coder{\omega}{k}$, then so is $\tilde{\gamma}(t)$, and thus $\by(t)=\bc$ for some constants $\bc=(c^1,\ldots,c^s)$. Thus the curve $\tilde{\gamma}(t)$ is contained in the submanifold defined by $\by=\bc$. However, if we wanted to connect two points \textit{not} in any such submanifold, then we could not connect those two points via an integral curve of $\omega$. Hence, the controllability property implies that $\omega$ must be completely nonintegrable. Hence, for the remainder of this thesis, we will always assume that control systems are completely nonintegrable. 

\subsection{Lie Transformations}
The classes of diffeomorphisms considered in this thesis fall under the umbrella of Lie transformation (pseudo)groups. The study of transformation pseudogroups has produced a rich literature of interesting results. We cannot give a full account here, but we would like to direct the interested reader to \cite{BryantLieSymplectic}, \cite{KuranishiPseudo1},\cite{KuranishiPseudo2}, and \cite{LieCartanGroups}, as well as \cite{OlverPseudoOverview} for modern perspectives in this area. We will use the following definition of a Lie pseudogroup.
\begin{defn}\label{pseudo def}\cite{LieCartanGroups}
Let $M$ be a differentiable manifold and let $\mcal{P}$ be a collection of diffeomorphisms of open subsets of $M$ into $M$. We say that $\mcal{P}$ is a \textbf{Lie pseudogroup} if:
\begin{enumerate}
\item $\mcal{P}$ is closed under restriction: if $\varphi:U\to M$ belongs to $\mcal{P}$ so does $\varphi|_V$ for any $V\subset U$, open. 
\item Elements of $\mcal{P}$ can be pieced together: If $\varphi: U\to M$ is a diffeomorphism and $U=\cup_{\alpha} U_\alpha$ with $\varphi|_{U_\alpha}\in\mcal{P}$ then $\varphi\in\mcal{P}$.
\item $\mcal{P}$ is closed under inverse: if $\varphi:U\to M$ belongs to $\mcal{P}$ so does $\varphi^{-1}:\varphi(U)\to M$.
\item $\mcal{P}$ is closed under composition: if $\varphi: U\to M$ and $\psi: \varphi(U)\to M$ are in $\mcal{P}$ then $\psi\circ\varphi\in \mcal{P}$. 
\item The identity diffeomorphism belongs to $\mcal{P}$.  
\end{enumerate}
\end{defn}
Although we are interested in studying control systems invariant under certain infinite dimensional Lie pseudogroups of transformations, we will also work with transformations induced by finite dimensional Lie groups. 
\begin{defn}\label{local Lie group}
Let $G$ be a Lie group of dimension $r<\infty$. A \textbf{local Lie group} is (up to Lie group isomorphism) any open subset of $G$ containing the identity element $e$. 
\end{defn}
Unless otherwise specified, all Lie groups $G$ in this thesis will be considered as local Lie groups. Furthermore, we note that Definition \ref{local Lie group} is essentially Theorem 1.22 of \cite{OlverLieBook} which says that every local Lie group may be realized as an open subset of a Lie group that contains the identity element. 
\begin{defn}\cite{OlverLieBook}
Let $M$ be a smooth manifold. A \textbf{(local) Lie group of transformations} acting on $M$ is given by a (local) Lie group $G$, an open subset $U$, with 
\begin{equation}
\{e\}\times M\subset U\subset G\times M,
\end{equation}
which is the domain of definition of the group action, and a smooth map $\Psi:U\to M$ with the following properties:
\begin{enumerate}
\item If $(h,x)\in U,\,(g,\Psi(h,x))\in U$, and also $(g\cdot h,x)\in U$, then 
\begin{equation}
\Psi(g,\Psi(h,x))=\Psi(g\cdot h, x).
\end{equation}
\item For all $x\in M$, 
\begin{equation}
\Psi(e,x)=x.
\end{equation}
\item If $(g,x)\in U$, then $(g^{-1},\Psi(g,x))\in U$ and 
\begin{equation}
\Psi(g^{-1},\Psi(g,x))=x.
\end{equation}
One may also write $g\cdot x$ for $\Psi(g,x)$. 
\end{enumerate}
\end{defn}
When we refer to a Lie group of transformations, we will always be referring to a \textit{local} Lie group of transformations unless otherwise specified. 
In practice, we will usually rely on the infinitesimal version of Lie transformations, which we now define. 
\begin{defn}\cite{OlverLieBook}
Let $G$ be a Lie transformation group acting on a smooth manifold $M$, and let $\mathfrak{g}$ be the Lie algebra of right-invariant vector fields on $G$. Then the \textbf{infinitesimal action} of $\mathfrak{g}$ on $M$ is given by
\begin{equation}
\psi(v)|_x=\frac{d}{ds}\Big|_{s=0}\Psi(\exp(sv),x)=d\Psi_x(v|_e)\label{exp action}
\end{equation}
for all $v\in\mathfrak{g}$ and $x\in U\subset M$, and where $\Psi_x(g)=\Psi(g,x)$. Equation (\ref{exp action}) defines a vector field, $\psi(v)$, on $U\subset M$. 
\end{defn}
The map $\psi:\mathfrak{g}\to \Gamma(TM)$ defined by \ref{exp action}, is a Lie algebra homomorphism and sections of $TM$ in the image of $\psi$ are \textit{infinitesimal generators} of the group action $G$. We then have the following important theorem. 
\begin{thm}\label{inf vf}\cite{OlverLieBook}
Let $w_1,\ldots,w_r$ be vector fields on a manifold $M$ satisfying
\begin{equation}
[w_i,w_j]=\sum_{k=1}^r c^k_{ij}w_k,\quad i,j=1,\ldots,r, 
\end{equation}
for certain constants $c^k_{ij}$. Then there is a Lie group $G$ whose Lie algebra has the given $c^k_{ij}$ as structure constants relative to some basis $v_1,\ldots,v_r,$ and a local group action of $G$ on $M$ such that $\psi(v_i)=w_i$ for $i=1,\ldots,r,$ where $\psi$ is defined by (\ref{exp action}). 
\end{thm}
\begin{defn}
The vector fields in $w_1,\ldots,w_r$ in Theorem \ref{inf vf} are called the \textbf{infinitesimal generators} of the action of $G$ on $M$. 
\end{defn}
Recall Definition \ref{EDS symmetry}, where we defined an infinitesimal symmetry of an EDS. If an EDS has infinitesimal symmetries that satisfy the hypotheses of Theorem \ref{inf vf}, then there will be an associated local Lie group of transformations that, as will be discussed later, takes integral manifolds of an EDS to integral manifolds of the same EDS. It is often easier to work with infinitesimal symmetries, however, and in light of Theorem \ref{inf vf}, the word ``symmetry" will be used to mean either an element of a local Lie group of transformations $G$, or one of the vector fields that arise from an infinitesimal action of $\mathfrak{g}$ on a manifold $M$. In Chapter 3, we will explore symmetries of Pfaffian systems as they apply to control systems. 
Next, we'll present some important facts about local Lie group actions that will allow us to investigate invariant integral manifolds of an EDS. 
\begin{defn}
Let $G$ be a local Lie group of transformations acting on a manifold $M$. Define the \textbf{stabilizer} of a point $x\in M$ as the set
\begin{equation}
G_x=\{g\in G\colon\,g\cdot x=x\}.
\end{equation}
We say the action of $G$ on $M$ is \textbf{free} if for all $x\in M$, $G_x=\{e\}$, where $e$ is the identity element. 
\end{defn}
\begin{defn}
Let $G$ be a local Lie transformation group acting on a manifold $M$. Then the \textbf{orbit} of the action on a point $x\in M$ is
\begin{equation}
G\cdot x=\{y\in M\colon\,y=g\cdot x, \,\mathrm{for}\,\mathrm{some}\,g\in G\}.
\end{equation}
Two points $x$ and $y$ in $M$ are \textbf{equivalent} if and only if they belong to the same orbit. Then the space of equivalence classes endowed with the quotient topology is denoted $M/G$ and is called the \textbf{orbit space} of the action of $G$ on $M$. 
\end{defn}
\begin{thm}\cite{BryantLieSymplectic}
The stabilizer $G_x$ for any point $x\in M$ is a closed Lie subgroup of $G$. Additionally, the orbits $G\cdot x$ are smooth immersed submanifolds of $M$.
\end{thm}
\begin{prop}\cite{OlverSymmetryBook}
Let $G$ be a local Lie group of transformations acting on a manifold $M$. Then the dimensions of the orbits of $G$ are all equal to the dimension of $G$ if and only if $G$ acts freely on $M$. 
\end{prop}
\begin{defn}\label{regular def}\cite{OlverLieBook}
Let $G$ be a local Lie group of transformations acting on a manifold $M$. If the dimensions of the orbits of the action are all constant and equal then we say that $G$ acts \textbf{semi-regularly} on $M$. Furthermore, the action of $G$ on $M$ is called \textbf{regular} if $G$ acts semi-regularly on $M$ and has the additional property that for any $x\in M$, there exist arbitrarily small open sets $U$ containing $x$ such that individual orbits of $G$ intersect $U$ in pathwise connected subsets. 
\end{defn}
\textbf{Remark:} In other contexts, a \textit{regular} group action may refer to a free and transitive group action; however, we have no need for this meaning of the word. We also mention that, if a group action on a manifold is regular as in Definition \ref{regular def}, then the orbits of the action are regular submanifolds, although the converse may not be true. 

A classic example of a semi-regular, but not regular, group action on a manifold $M$ is the case of an irrational flow on the 2-torus. The group $G$ is the whole real line, and although each orbit is a 1-dimensional, immersed submanifold of $M$, every open set of any point on $\mathbb{T}^2$ fails the definition of regularity since the orbits of the action are dense. However, if $G$ is any nontrivial, finite open interval of $\mathbb{R}$ containing zero, then the corresponding irrational flow \textit{is} a regular action by $G$ on $\mathbb{T}^2$.
The definitions of semi-regular and regular actions extend to the infinitesimal action of a Lie group as well. Additionally, the definition extends to any completely integrable distribution on $M$. 
\begin{defn}\cite{OlverLieBook}
Let $\mcal{V}$ be a completely integrable distribution on a manifold $M$. If $\rank\mcal{V}$ is a fixed constant everywhere on $M$, then we say $\mcal{V}$ is \textbf{semi-regular}. Furthermore, a semi-regular distribution $\mcal{V}$ is \textbf{regular} if the integral manifolds of $\mcal{V}$ have the property that for any $x\in M$, there exist arbitrarily small open sets $U$ containing $x$ such that the individual integral manifolds of $\mcal{V}$ intersect $U$ in pathwise connected subsets. 
\end{defn}
Let $\Gamma$ denote the span over $C^\infty(M)$ of the infinitesimal generators of the action on $M$ of a Lie group $G$, which as a distribution, is always completely integrable by virtue of the Jacobi identity and Definition \ref{EDS symmetry}. If $\Gamma$ is semi-regular or regular, then the action of $G$ is also semi-regular or regular, respectively.  As in the previous example, one may not always be guaranteed that a given distribution $\Gamma$ on $M$ corresponding to a Lie group action is regular or even semi-regular. However, we can always restrict to smaller open submanifolds of $M$ and smaller open submanifolds of $G$ containing the identity such that $\Gamma$ is semi-regular or regular. For the remainder of the thesis, we will always assume that we have restricted to sufficiently small open submanifolds of $M$ and $G$ to guarantee that all actions are regular. 
\begin{thm}\cite{OlverLieBook}
Let $M$ be a smooth $n$-dimensional manifold. Suppose $G$ is an $r$-dimensional local Lie group of transformations which acts regularly and freely on $M$. Then the orbit space, or quotient manifold $M/G$, is a smooth $(n-r)$-dimensional manifold with a projection map $\pi:M\to M/G$ such that the following hold. 
\begin{enumerate}
\item $\pi$ is a smooth map between manifolds.\\
\item Two points $x$ and $y$ belong to the same orbit of $G$ in $M$ if and only if $\pi(x)=\pi(y)$.\\
\item If $\Gamma$ denotes the Lie algebra of infinitesimal generators of the action of $G$ on $M$, then the linear map
\begin{equation}
d\pi:TM\big|_{x}\to T(M/G)\big|_{\pi(x)}
\end{equation}
is onto, with kernel $\Gamma\big|_x=\{X\big|_x\colon X\in \Gamma\}$.\\
\item If $\{\eta^i(x)\}_{i=1}^{n-r}$ are independent first integrals of the Lie algebra of infinitesimal generators $\Gamma$, then $(\eta^1,\ldots,\eta^{n-r})$ form a local coordinate system on $M/G$. 
\end{enumerate}
\end{thm}

\subsection{Extended Static Feedback Transformations (ESFTs)}
We are concerned with two types of diffeomorphisms determining equivalence classes of control systems. They are known as \textit{static feedback transformations} and \textit{extended static feedback transformations}. The former of the two types of transformations are of particular interest to control theory as a whole and are generally well studied. The latter are a slightly broader type of diffeomorphism that allows for extra time dependence. To be precise:
\begin{defn}
A diffeomorphism $\Phi:M\to N$ of the form  
\[
\Phi:(t,x,u) \mapsto (t,\varphi(x),\psi(x,u)) 
\]
is called a \textbf{static feedback transformation} (SFT). Two control systems $(M,\omega)$ and $(N,\eta)$ are called \textbf{static feedback equivalent} (SFE) if $\Phi^*\eta=\omega$ for some SFT $\Phi:M\to N$.
\end{defn}
And the slightly broader class of diffeomorphisms can be defined by:
\begin{defn}
A diffeomorphism $\Phi:M\to N$ of the form  
\[
\Phi:(t,x,u) \mapsto (t,\varphi(t,x),\psi(t,x,u)) 
\]
is called an \textbf{extended static feedback transformation} (ESFT). Two control systems $(M,\omega)$ and $(N,\eta)$ are called \textbf{extended static feedback equivalent} (ESFE) if $\Phi^*\eta=\omega$ for some ESFT $\Phi:M\to N$.
\end{defn}
Although SFTs are more common in the control theory literature, we will need the use of both types. In particular, the last chapter necessarily requires that we use ESFTs. Since SFTs are special type of ESFTs, we will always refer to the more general case unless otherwise specified. 
\begin{ex}
The two systems 
\begin{equation}
\omega=\langle dx^1-x^3\,dt,dx^2-x^4\,dt,dx^3-u^1\,dt,dx^4-u^2\,dt\rangle
\end{equation}
on the manifold $M=\mathbb{R}\times\mathbb{R}^4\times\mathbb{R}^2$ and 
\begin{equation}
\eta=\langle dy^1-(y^3e^{-ty^1}-t)\,dt,dy^2-y^4\,dt,dy^3-v^1y^2\,dt,dy^4-y^1\arctan(v^2)\,dt\rangle
\end{equation}
on the manifold $N=\mathbb{R}\times\mathbb{R}^4\times\mathbb{R}^2$ are ESFE. 
\end{ex}
The ESFT that accomplishes the equivalence is given by
\begin{equation}
\Phi^{-1}: (t,\by,\bv)\mapsto (t,e^{ty^1},y^2,y^3,y^4,v^1y^2,y^1\arctan(v^2)).
\end{equation}
Computing the pullback by $\Phi^{-1}$ of the forms that generate $\omega$, we find 
\begin{align}
(\Phi^{-1})^*(dx^1-x^3\,dt)&=te^{ty^1}\left(dy^1-\frac{(y^3e^{-ty^1}-y^1)}{t}\,dt\right),\\
(\Phi^{-1})^*(dx^2-x^4\,dt)&=dy^2-y^4\,dt,\\
(\Phi^{-1})^*(dx^3-u^1\,dt)&=dy^3-v^1y^2\,dt,\\
(\Phi^{-1})^*(dx^4-u^2\,dt)&=dy^4-y^1\arctan(v^2)\,dt.
\end{align}
Hence, $(\Phi^{-1})^*\omega=\eta$, so the two systems are ESFE. 
\subsection{Brunovsk\'y Normal Forms and Goursat Bundles}
In this section we will explore a specific class of controllable linear control systems that are equivalent via ESFTs. We will introduce jet bundles, contact systems, and a generalization of these concepts known as Goursat bundles.
\begin{defn}\label{jet}
Let $f,g: \mathbb{R}\to \mathbb{R}^m$ be two $C^n$ curves in $\mathbb{R}^m$. We say that $f$ and $g$ are equivalent via $n$\textbf{-th order contact at a point} $x_0\in\mathbb{R}$ if the $n$th degree Taylor polynomials for $f$ and $g$ agree at $x_0$. In particular, we denote the equivalence class of $f$ as $j^n_{x_0}f$ and we call $j^n_{x_0}f$ the $n$-\textbf{jet of} $f$ \textbf{at} $x_0$.   
\end{defn}
Two functions are in $0$-th order contact at $x_0$ if the graphs of $f$ and $g$ in $\mathbb{R}\times \mathbb{R}^m$ pass through the same point at $x_0$, $1$-st order contact if the graphs are mutually tangent to each other at $x_0$, and so on. However, we are not only interested in $n$-jets of functions over a single point. 
\begin{defn}
Let $J^n_{x_0}(\mathbb{R},\mathbb{R}^m)$ denote the space of $n$-jets of functions $f:\mathbb{R}\to\mathbb{R}^m$ at $x_0$. Then the \textbf{jet bundle of order} $n$ is defined to be 
\begin{equation}
J^n(\mathbb{R},\mathbb{R}^m)=\coprod_{x_0\in\mathbb{R}}J^n_{x_0}(\mathbb{R},\mathbb{R}^m).
\end{equation}
Furthermore, the space $\mathbb{R}$ in the notation $J^n(\mathbb{R},\mathbb{R}^m)$ may be refered to as the \textbf{source} and is the image of the \textbf{source projection map} 
\begin{align*}
\pi_n:J^n(\mathbb{R},\mathbb{R}^m)\to \mathbb{R},
j^n_{x_0}f\mapsto x_0.
\end{align*}
\end{defn}
We will often abbreviate the notation $J^n(\mathbb{R},\mathbb{R}^m)$ to $J^n$ whenever there is no danger of ambiguity. In general, jet bundles can be defined for maps between any two differentiable manifolds $M$ and $N$. For more on jet bundles see the text \cite{JetBundleText}.  
Let $t$ be the local coordinate for $\mathbb{R}$ in $J^n$ and $(z^1_0,\ldots,z^m_0)$ the local coordinates for $\mathbb{R}^m$ in $J^n$. The jet bundle $J^n$ has local coordinates given by $(t, z^i_0, z^i_1,\ldots, z^i_n)$ where $1\leq i\leq m$. The $n$-jet lift of a function $f:\mathbb{R}\to\mathbb{R}^m$ is the curve $j^nf:\mathbb{R}\to J^n$ that has the parameterization given by $\left(t, f^i(t), \frac{df^i}{dt}(t),\ldots, \frac{d^nf^i}{dt^n}(t)\right)$. Thus one can interpret local coordinates for a jet bundle so that the coordinate from $\mathbb{R}$ is the ``independent variable", the coordinates $z^i_0$ are ``place-holders" for the ``dependent variables", and the $z^i_l$ may be thought of as ``place-holders" for $l$th order derivatives of the ``dependent variables".

Consider the jet space $J^n(\mathbb{R},\mathbb{R}^m)$. There is a natural Pfaffian system on this space whose integral manifolds correspond to the graphs of jets of functions from $\mathbb{R}$ to $\mathbb{R}^m$. This Pfaffian system is called the \textit{contact system} or the \textit{Cartan system} \cite{CartanBeginners}. Note: the terminology ``Cartan system" has another well established meaning in EDS theory as the \textit{retracting space}; see Chapter 6.1 of \cite{CartanBeginners}. 
\begin{defn}
The linear Pfaffian system $\left(J^n(\mathbb{R},\mathbb{R}^m),\beta^n_m\right)$ with 
\begin{align}
\beta^n_m&=\langle \theta^i_l \rangle,\\
\theta^i_l&=dz^i_l-z^i_{l+1}\,dt,
\end{align}
for all $1\leq i\leq m$ and $0\leq l\leq n-1$, is called the \textbf{contact system}. Furthermore, denote by $\mathcal{C}^n_m$ the distribution on $J^n(\mathbb{R},\mathbb{R}^m)$ that is annihilated by the 1-forms $\{\theta^i_l\}$. 
\end{defn}
Let $f:\mathbb{R}\to\mathbb{R}^m$ be any smooth function given in coordinates by $(t,f^1(t),\ldots,f^m(t))$. Then the $n$-jet $j^nf$ is an integral curve of $\beta^n_m$. Indeed, each $\theta^i_l$ is zero when restricted to the curve $j^nf(t)$, since 
\begin{equation}
\begin{aligned}
(j^nf)^*(dz^i_l)&=d\left(\frac{d^lf^i}{dt^l}\right)\\
\,&=\frac{d^{l+1}f^i}{dt^{l+1}}\,dt\\
\,&=(j^nf)^*(z^i_{l+1})\,dt,
\end{aligned}
\end{equation}
and therefore
\begin{equation}
(j^nf)^*(dz^i_l-z^i_{l+1}\,dt)=0.
\end{equation}
We now discuss the notion of \textit{prolongation} of a jet space/contact system. First, notice that there is a surjective submersion $\pi: J^{n+1}(\mathbb{R},\mathbb{R}^m)\to J^n(\mathbb{R},\mathbb{R}^m)$ given by $j^{n+1}f\mapsto j^nf$, or in coordinates, $(t,z^i_0,\ldots,z^i_{n+1})\mapsto(t,z^i_0,\ldots,z^i_n)$. The canonical contact systems on these two jet spaces have the property that
\begin{equation}
\beta^{n+1}_m=\pi^*\left(\beta^n_m\right)\oplus\langle \theta^i_n\rangle,
\end{equation}
so that integral curves $j^n_tf$ of $\beta^n_m$ lift to integral curves $j^{n+1}_tf$ of $\beta^{n+1}_m$. For our purposes, we'll use the following restricted definition:
\begin{defn}
The linear Pfaffian system $\left(J^{n+1}(\mathbb{R},\mathbb{R}^m),\beta^{n+1}_m\right)$ is a \textbf{prolongation}\\ of $\left(J^n(\mathbb{R},\mathbb{R}^m),\beta^n_m\right)$. 
\end{defn}
See \cite{CartanBeginners} and \cite{BCGGG} for the general definition of prolongation of exterior differential systems. In this thesis, we will frequently work on \textit{partial} prolongations of jet spaces. 
\begin{defn}\label{partial prolongation}
A \textbf{partial prolongation} of the Pfaffian system $(J^1(\mathbb{R},\mathbb{R}^m),\beta^1_m)$ is given by the Pfaffian system $(J^\kappa(\mathbb{R},\mathbb{R}^m), \beta^\kappa_m)$ where
\begin{align}
J^\kappa(\mathbb{R},\mathbb{R}^m)&:=\left(\prod_{i\in I}J^i(\mathbb{R},\mathbb{R}^{\rho_i})\right)/\sim,\label{jkappa}\\
\beta^\kappa_m&:=\bigoplus_{i\in I}\beta^i_{\rho_i},
\end{align}
with $I=\{1\leq a\leq k\,\mid\rho_a\neq0 \}$. The equivalence relation `$\sim$' in (\ref{jkappa}) is defined by
\begin{equation}
\pi_i\left(J^i(\mathbb{R},\mathbb{R}^{\rho_i})\right)=\pi_j\left(J^j(\mathbb{R},\mathbb{R}^{\rho_j})\right),
\end{equation}
for all $1\leq i,j\leq k$, where $\pi_i,\pi_j$ are source projection maps. Furthermore, $m=\rho_1+\cdots+\rho_k$, where $k$ is the derived length of $\beta^\kappa$, and $\kappa=\langle \rho_1,\ldots,\rho_k\rangle$ is the list of natural numbers that defines the \textbf{type} of the partial prolongation of $J^1(\mathbb{R},\mathbb{R}^m)$.
\end{defn}
The contact system mentioned in Definition \ref{partial prolongation} represents an important class of control systems. Indeed, in \cite{Brunovsky} it was proven than any controllable linear system is equivalent via a linear feedback transformation to a partial prolongation of the form given in Definition \ref{partial prolongation}. 
\begin{defn}
The contact system $\beta^\kappa_m$ in Definition \ref{partial prolongation} is called a \textbf{Brunovsk\'y normal form} of type $\kappa$, and we will denote by $\mcal{C}^\kappa_m$ the distribution annihilated by $\Omega^1(M)\cap\beta^\kappa_m$.  
\end{defn}
A Brunovsk\'y normal form is uniquely determined by its type $\kappa$. For example, a Brunovsk\'y normal form of type $\kappa=\langle1,2,0,0,1,1\rangle$ on
\begin{equation}
J^\kappa(\mathbb{R},\mathbb{R}^5)=\left(J^1(\mathbb{R},\mathbb{R})\times J^2(\mathbb{R},\mathbb{R}^2)\times J^5(\mathbb{R}\mathbb{R})\times J^6(\mathbb{R},\mathbb{R})\right)/\sim
\end{equation}
is generated by the 1-forms
\begin{equation*}
\begin{array}{ccccc}
\,&\,&\,&\,&\theta^5_5=dz^5_5-z^5_6\,dt,\\
\,&\,&\,&\theta^4_4=dz^4_4-z^4_5\,dt,&\theta^5_4=dz^5_4-z^5_5\,dt,\\
\,&\,&\,&\theta^4_3=dz^4_3-z^4_4\,dt,&\theta^5_3=dz^5_3-z^5_4\,dt,\\
\,&\,&\,&\theta^4_2=dz^4_2-z^4_3\,dt,&\theta^5_2=dz^5_2-z^5_3\,dt,\\
\,&\theta^2_1=dz^2_1-z^2_2\,dt,&\theta^3_1=dz^3_1-z^3_2\,dt,&\theta^4_1=dz^4_1-z^4_2\,dt,&\theta^5_1=dz^5_1-z^5_2\,dt,\\
\theta^1_0=dz^1_0-z^1_1\,dt,&\theta^2_0=dz^2_0-z^2_1\,dt,&\theta^3_0=dz^3_0-z^3_1\,dt,&\theta^4_0=dz^4_0-z^4_1\,dt,&\theta^5_0=dz^5_0-z^5_1\,dt.
\end{array}
\end{equation*}
In this example, one can say that $J^\kappa$ has one variable of order 1, two of order 2, zero of orders 3 and 4, one of order 5, and one of order 6. So the type $\kappa$ is a list of the local coordinates on $J^\kappa$ categorized by order. As we will see later in this section, the type $\kappa$ of a Brunovsk\'y form is a diffeomorphism invariant. However, when working in coordinates on the partial prolongation of a jet space, it is often easier to use an alternative notation. Indeed, we can also write the partial prolongation of a jet space as
\begin{equation}
J^\kappa(\mathbb{R},\mathbb{R}^m)=\left(\prod_{i=1}^m J^{\sigma_i}(\mathbb{R},\mathbb{R})\right)/\sim
\end{equation}
where the equivalence relation is the same as in Definition \ref{partial prolongation} in the sense that all source manifolds for each jet space are identified. For this description of $J^\kappa$ we can write local coordinates as $(t, j^{\sigma_1}z^1,\ldots,j^{\sigma_m}z^m)$, where $j^{\sigma_i}z^i$ represents the $(\sigma_i+1)$-tuple $(z^i_0,z^i_1,\ldots,z^i_{\sigma_i})$. For the example of a partial prolongation of a jet space with type $\kappa=\langle1,2,0,0,1,1\rangle$, one can then write the local coordinates as $(t,j^1z^1,j^2z^2,j^2z^3,j^5z^4,j^6z^5)$. 

\begin{prop}\label{refined derived type numbers}\cite{VassiliouGoursatEfficient}
Let $\mcal{C}^\kappa_m\subset TM$ be the distribution that annihilates the 1-forms in a Brunovsk\'y normal form $\beta^\kappa_m$ with type $\kappa=\langle \rho_1,\ldots,\rho_k\rangle$. Then the entries in the refined derived type 
\begin{equation}
\mathfrak{d}_r(\mcal{C}^\kappa_m)=[[m_0,\chi^0],[m_1,\chi^1_0,\chi^1],\ldots,[m_{k-1},\chi^{k-1}_{k-2},\chi^{k-1}],[m_k,\chi^k]]
\end{equation}
satisfy the following relations:
\begin{align}
\kappa&=\mathrm{deccel}(\mcal{C}^\kappa_m),\quad \Delta_i=\sum_{l=i}^{k}\rho_l,\\ \nonumber
m_0&=1+m,\quad m_j=m_0+\sum_{l=1}^j\Delta_l,\quad1\leq j\leq k,\\ \nonumber
\chi^j&=2m_j-m_{j+1}-1,\quad 0\leq j\leq k-1,\\ \nonumber
\chi^i_{i-1}&=m_{i-1}-1,\quad 1\leq i\leq k-1, \nonumber
\end{align}
where $\Delta_j$ is given in Definition \ref{accel def}.  
\end{prop}
Some of the most important geometric structures for this thesis are the generalized Goursat bundles \cite{VassiliouGoursat}. The prototypical examples of Goursat bundles are exactly those subbundles $\mathcal{B}^\kappa$ of the tangent bundle of some $J^\kappa$ that are annihilated by the Brunovsk\'y 1-forms on $J^\kappa$. Like the example of Brunovsk\'y forms, Goursat bundles have the property of being \textit{completely nonintegrable}. 
\begin{defn}
A (nonsingular) \textbf{rank 2 Goursat bundle} is a rank 2 distribution $\mathcal{V}$ on $M$ with no Cauchy characteristics such that $\mathrm{dim}\mathcal{V}^{(i+1)}=1+\mathrm{dim}\mathcal{V}^{(i)}$ for all $1\leq i\leq k-1$ and $\mathcal{V}^{(k)}=TM$. 
\end{defn}
Furthermore, although Goursat bundles are examples of completely nonintegrable distributions, they are in some sense degenerate among such distributions. The growth of the derived flag is as slow as possible to still guarantee that the distribution is completely nonintegrable. Rank 2 Goursat bundles are the classical Goursat bundles that were studied by Goursat, Engel, and E. Cartan. Indeed, in the case that $M$ is a 4 dimensional manifold, a rank 2 Goursat bundle on $M$ is exactly an Engel structure. 

What about bundles of higher rank? This is the work of Vassiliou in \cite{VassiliouGoursat} and \cite{VassiliouGoursatEfficient}. Indeed, a generalized Goursat bundle, or simply a Goursat bundle, is given by the following definition. 
\begin{defn}\cite{VassiliouGoursat}\label{Goursat}
A subbundle $\mathcal{V}\subset TM$ of derived length $k$ will be called a \textbf{Goursat bundle of type} $\kappa$ if:
\begin{enumerate}
\item $\mathcal{V}$ has the refined derived type of a partial prolongation of $J^1(\mathbb{R},\mathbb{R}^m)$ whose type $\kappa=\mathrm{deccel}(\mathcal{V})$,
\item each intersection $\inChar{V}{i}:= \mathcal{V}^{(i-1)}\cap\Char{V}{i}$ is an integrable subbundle whose rank, assumed to be constant on $M$, agrees with the corresponding rank of $\mathrm{Char}(\mathcal{C}^\kappa_m)^{(i)}_{i-1}$, and
\item in case $\Delta_k>1$, then $\mathcal{V}^{(k-1)}$ determines an integrable Weber structure whose resolvent bundle is of rank $\Delta_k+\chi^{k-1}$. 
\end{enumerate}
\end{defn}
Goursat bundles have a particularly nice normal form, and this is the main result of \cite{VassiliouGoursat}. 
\begin{thm}\label{Generalized Goursat Normal Form}\cite{VassiliouGoursat}
(Generalized Goursat Normal Form). Let $\mathcal{V}\subset TM$ be a Goursat bundle on a manifold $M$, with derived length $k>1$, and type $\kappa=\mathrm{deccel}(\mathcal{V})$. Then there is an open dense subset $U\subset M$ such that the restriction of $\mathcal{V}$ to $U$ is locally equivalent to $\mathcal{C}^\kappa_m$ via a local diffeomorphism of $M$. Conversely, any partial prolongation of $\mathcal{C}^1_m$ is a Goursat bundle. 
\end{thm}
The paper \cite{VassiliouGoursat} establishes the local normal form for generalized Goursat bundles constructively. However, in \cite{VassiliouGoursatEfficient}, the construction of local coordinates is streamlined into a nearly algorithmic procedure. We'll next outline this procedure, often referred to as procedure \textbf{contact}, and apply it to an example in detail.

\subsection{ESF Linearizable Systems and Procedure Contact}
In this section we'll present procedure \textbf{contact} from \cite{VassiliouGoursatEfficient} and then apply the procedure to Goursat bundles that represent control systems. There will be additional requirements to make sure that the diffeomorphism created by procedure \textbf{contact} can be chosen to be an ESFT. At the end of this section, we prove a result that highlights the difference between ESFL and SFL systems.

Let $\mathcal{V}$ be a Goursat bundle on a manifold $M$ with derived length $k$. The Goursat bundle $\mathcal{V}$ will induce one of two possible filtrations of $TM$; one for the case that $\mathcal{V}$ has $\Delta_k>1$ and the other for the case of $\Delta_k=1$. To start, we'll assume that $\Delta_k>1$. The associated filtration of $TM$ for such a Goursat bundle is given by
\begin{equation}\label{filt res}
\text{Char}\,\mathcal{V}^{(1)}_0\subseteq\text{Char}\,\mathcal{V}^{(1)}\subset\cdots\subset\text{Char}\,\mathcal{V}^{(k-1)}_{k-2}\subseteq\text{Char}\,\mathcal{V}^{(k-1)}\subset\mathcal{R}_{\bar{\Sigma}_{k-1}}(\mathcal{V}^{(k-1)})\subset TM.
\end{equation}
Similarly, there is also a filtration of $T^*M$ defined by taking the annihilator of all of the above subbundles,
\begin{equation}
\Upsilon_{\bar{\Sigma}_{k-1}}(\mathcal{V}^{(k-1)})\subset \Xi^{(k-1)}\subseteq \Xi^{(k-1)}_{k-2}\subset\cdots\subset\Xi^{(1)}\subseteq\Xi^{(1)}_0\subset T^*M,
\end{equation}
where $\Xi^{(i)}=\text{ann }\Char{V}{i},\,\Xi^{(i)}_{i-1}=\text{ann }\inChar{V}{i}$, and $\Upsilon_{\bar{\Sigma}_{k-1}}(\mathcal{V}^{(k-1)})=\text{ann }\mathcal{R}_{\bar{\Sigma}_{k-1}}(\mathcal{V}^{(k-1)})$. Each subbundle in these filtrations is integrable by Definition \ref{Goursat}. In particular, these subbundles are diffeomorphism invariants of a given Goursat bundle, and hence their first integrals are also diffeomorphism invariants of the Goursat bundle. Such invariant functions will be used to construct the appropriate contact coordinates. We will not, however, need all first integrals of these subbundles. Notice that $\Char{V}{i}=\inChar{V}{i}$ if and only if $-\Delta^2_i=\rho_i=0$. Specifically, it is the case that $\rank \Xi^{(i)}_{i-1}/\Xi^{(i)}=\rho_i$. It turns out that the first integrals of each nontrivial quotient bundle $\Xi^{(i)}_{i-1}/\Xi^{(i)}$, and also of the annihilator of the resolvent bundle, $\Upsilon_{\bar{\Sigma}_{k-1}}(\mathcal{V}^{(k-1)})$, give the zeroth order contact coordinates on the appropriate $J^\kappa$. 
\begin{defn}\cite{VassiliouGoursatEfficient}
For each $1\leq j\leq k-1$, let $\{\varphi^{l_j,j}\}_{l_j=1}^{\rho_j}$ generate the independent first integrals of $\Xi^{(j)}_{j-1}/\Xi^{(j)}$. Each $\varphi^{l_j,j}$ is called a \textbf{fundamental function of order} $j$. Now let $\{\varphi^{0,k},\ldots,\varphi^{\rho_k,k}\}$ generate the first integrals of $\Upsilon_{\bar{\Sigma}_{k-1}}(\mathcal{V}^{(k-1)})$. These will be the \textbf{fundamental functions of order} $k$. 
\end{defn}
Notice that there are $\rho_j$ fundamental functions of order $j$ and $\rho_k+1$ fundamental functions of order $k$. The fundamental function $\varphi^{0,k}$ will usually denote a local coordinate for the source of $J^\kappa$. 
\begin{thm}\cite{VassiliouGoursatEfficient}
Let $\mathcal{V}\subset TM$ be a Goursat bundle of derived length $k$ with $\kappa=\mathrm{deccel}(\mathcal{V})=\langle \rho_1,\ldots,\rho_k\rangle$, $\rho_k\geq2$. Let $\{x,\varphi^{1,k},\ldots,\varphi^{\rho_k,k}\}$ denote the fundamental functions of order $k$, and for each $\rho_j>0$ let $\{\varphi^{1,j},\ldots,\varphi^{\rho_j,j}\}$ denote the fundamental functions of order $j$ defined on some open subset $U\subseteq M$.
Then there is an open, dense subset $V\subseteq U$ and a section $Z$ of $\mathcal{V}$ such that on $V$, $Zx=1$ and the fundamental functions, $x,\varphi^{l_j,j}_0:=\varphi^{l_j,j}$, together with the functions 
\begin{equation}
\varphi^{l_j,j}_{s_j+1}=Z\varphi^{l_j,j}_{s_j},\,j\in\{1,\ldots,k\},\,1\leq l_j\leq\rho_j,\,0\leq s_j\leq j-1, 
\end{equation} 
are contact coordinates for $\mathcal{V}$, identifying it with $\mathcal{C}^\kappa_m$ on $J^\kappa$. 
\end{thm}
The above theorem gives a way to explicitly construct the coordinates for the Brunovsk\'y normal form for a Goursat bundle (in the case that $\Delta_k>1$). We now give the analogous result in the case that $\Delta_k=1$. In this case, the Goursat bundle $\mathcal{V}$ induces the filtrations 
\begin{equation}\label{filt fun}
\text{Char}\,\mathcal{V}^{(1)}_0\subseteq\text{Char}\,\mathcal{V}^{(1)}\subset\cdots\subset\text{Char}\,\mathcal{V}^{(k-1)}_{k-2}\subseteq\text{Char}\,\mathcal{V}^{(k-1)}\subset\Pi^{k-1}\subset TM, 
\end{equation}
and 
\begin{equation}
(\Pi^{k-1})^\perp\subset \Xi^{(k-1)}\subseteq \Xi^{(k-1)}_{k-2}\subset\cdots\subset\Xi^{(1)}\subseteq\Xi^{(1)}_0\subset T^*M.
\end{equation}
In place of the resolvent bundle is a new integrable bundle $\Pi^{k-1}\subset\mathcal{V}^{(k-1)}$. 
\begin{defn}\cite{VassiliouGoursatEfficient}
Let $\mathcal{V}$ be a Goursat bundle with $\Delta_k=1$, $x$ a first integral of $\Char{V}{k-1}$, and $Z$ any section of $\mathcal{V}$ such that $Zx=1$. Then the \textbf{fundamental bundle} $\Pi^{k-1}\subset\mathcal{V}^{(k-1)}$ is defined inductively as
\begin{equation}
\Pi^{l+1}=\Pi^l+[\Pi^l,Z],\,\,\Pi^0=\inCharOne{V},\,\,0\leq l\leq k-2.
\end{equation}
\end{defn}
In the proof of Theorem 4.2 in \cite{VassiliouGoursat} it is shown that $\Pi^{k-1}$ is integrable and has corank 2 in $TM$. Note also that $x$ is a first integral of $\Pi^{k-1}$ by virtue of filtration (\ref{filt fun}). We can now state the theorem that constructs contact coordinates for $\mathcal{V}$ in the case that $\Delta_k=1$.
\begin{thm}\cite{VassiliouGoursatEfficient}
 Let $\mathcal{V}\subset TM$ be a Goursat bundle of derived length $k$ and $\kappa=\mathrm{deccel}({\mathcal{V}})=\langle \rho_1,\ldots,\rho_k\rangle, \rho_k=1$. Let $\Pi^{k-1}$ be the fundamental bundle, and let $\varphi^{1,k}$ be any first integral of $\Pi^{k-1}$ such that $dx\wedge d\varphi^{1,k}\neq0$ on an open set $U\subseteq M$. Then there is an open dense subset $V\subseteq U$ upon which there is a section $Z$ of $\mathcal{V}$ satisfying $Zx=1$ such that the fundamental functions $x,\varphi^{l_j,j}_0:=\varphi^{l_j,j}$ together with the functions
\begin{equation}
\varphi^{l_j,j}_{s_j+1}=Z\varphi^{l_j,j}_{s_j},\,j\in\{1,\ldots,k\},\,1\leq l_j\leq\rho_j,\,0\leq s_j\leq j-1,
\end{equation}
are contact coordinates for $\mathcal{V}$ on $V$, identifying it with $\mathcal{C}^\kappa_m$ on $J^\kappa$.  
\end{thm}
These results can summed up as a procedure for calculating local contact coordinates for a Goursat bundle.\
\begin{proc contact}\cite{VassiliouGoursatEfficient}\\
\textbf{Procedure A}\label{proc A}\\
Let $\mathcal{V}\subset TM$ be a Goursat bundle with derived length $k>1$ such that $\Delta_k>1$. Then one can do the following to produce local contact coordinates for $\mathcal{V}$:
\begin{enumerate}
\item Build filtration (\ref{filt res}) and its associated filtration of $T^*M$.
\item For each $1\leq j \leq k-1$ such that $\rho_j>0$, compute the quotient bundle $\Xi^{(j)}_{j-1}/\Xi^{(j)}$.
\item Compute the fundamental functions $\varphi^{l_j,j}$ of $\Xi^{(j)}_{j-1}/\Xi^{(j)}$.
\item Fix any fundamental function of order $k$ of the resolvent bundle, denoted $x$, and any section $Z$ of $\mathcal{V}$ such that $Zx=1$.
\item For each $1\leq j\leq k$ such that $\rho_j>0$, define $z^{l_j,j}_0=\varphi^{l_j,j}$, $1\leq l_j\leq \rho_j$. Furthermore, define the remaining contact coordinates to be
\begin{equation}\label{contact coord A}
z^{l_j,j}_{s_j}=Zz^{l_j,j}_{s_j-1}=Z^{s_j}z^{l_j,j}_0,\,1\leq s_j\leq j,\,1\leq l_j\leq \rho_j. 
\end{equation}
\end{enumerate}
The local coordinates for $J^\kappa(\mathbb{R},\mathbb{R}^m)$ are given by $x,z^{l_j,j}_0$, and (\ref{contact coord A}). In these coordinates $\mathcal{V}$ has the form $\mathcal{C}^\kappa_m$.\\ 
\textbf{Procedure B}\label{proc B}\\
Let $\mathcal{V}\subset TM$ be a Goursat bundle with derived length $k>1$ such that $\Delta_k=1$. Then one can do the following to produce local contact coordinates for $\mathcal{V}$:
\begin{enumerate}
\item Build filtration (\ref{filt fun}) and its associated filtration of $T^*M$ up to $\Char{V}{k-1}$.
\item Identify a first integral $x$ of $\Char{V}{k-1}$ such that there is a section $Z$ of $\mathcal{V}$ with the property $Zx=1$. Then construct $\Pi^{k-1}$, thereby completing filtration (\ref{filt fun}). 
\item For each $1\leq j\leq k-1$ such that $\rho_j>0$, compute the quotient bundle $\Xi^{(j)}_{j-1}/\Xi^{(j)}$.
\item Compute the fundamental functions $\varphi^{l_j,j}$ of $\Xi^{(j)}_{j-1}/\Xi^{(j)}$.
\item Define $z^{1,k}_0=\varphi^{1,k}$ to be any first integral of $\Pi^{k-1}$ such that $dx\wedge d\varphi^{1,k}\neq0$. 
\item For each $1\leq j\leq k$ such that $\rho_j>0$, define $z^{l_j,j}_0=\varphi^{l_j,j}$, $1\leq l_j\leq \rho_j$. Furthermore, define the remaining contact coordinates to be
\begin{equation}\label{contact coord B}
z^{l_j,j}_{s_j}=Zz^{l_j,j}_{s_j-1}=Z^{s_j}z^{l_j,j}_0,\,1\leq s_j\leq j,\,1\leq l_j\leq \rho_j. 
\end{equation}
\end{enumerate}
The local coordinates for $J^\kappa(\mathbb{R},\mathbb{R}^m)$ are given by $x,z^{l_j,j}_0$, and (\ref{contact coord B}). In these coordinates $\mathcal{V}$ has the form $\mathcal{C}^\kappa_m$. 
\end{proc contact}

Procedure \textbf{contact} produces a local \textit{diffeomorphism} equivalence between a Goursat bundle and a contact system. In particular, the first integral $x$ in procedure \textbf{contact} plays the role of the source variable of some $J^\kappa$, so that $dx$ forms the independence condition for the linear Pfaffian system $(J^\kappa,\beta^\kappa_m)$. Therefore, if $\mathcal{V}$ represents a control system with $dt$ as the independence condition, then integral curves of $\mathcal{V}$ may not be sent to integrals curves of $\beta^\kappa_m$ that are parameterized by $t$. The following theorem gives additional conditions that ensures that procedure \textbf{contact} produces an ESFT equivalence between a Goursat bundle $\mathcal{V}$ representing a control system and a Brunovsk\'y normal form. 
\begin{thm}\cite{VassiliouSICON}\label{Goursat ESFL}
Let $\mathcal{V}$ be a Goursat bundle of derived length $k>1$ that represents a control system on the manifold $M\cong_{\mathrm{loc}}\mathbb{R}\times\bX(M)\times\bU(M)$. Then $\mathcal{V}$ is ESF equivalent to a Brunovsk\'y normal form if and only if
\begin{enumerate}
\item $\{\partial_{u^a}\}\subset \inCharOne{V}$, 
\item $dt\in \Xi^{(k-1)}$ if $\Delta_k=1$ and $dt\in \Upsilon_{\bar{\Sigma}_{k-1}}(\mathcal{V}^{(k-1)})$ if $\Delta_k>1$.  
\end{enumerate}
\end{thm}
\begin{defn}
Let $\omega$ be a control system on the manifold $M\cong_{loc}\mathbb{R}\bX(M)\times\bU(M)$ . Then $\omega$ is \textbf{intrinsically nonlinear} if it is not ESF equivalent to a Brunovsk\'y normal form. 
\end{defn}
Before we present some examples, it is important to discuss previous work concerning linearization of control systems via feedback transformations. The work of Gardner, Shadwick, and Wilkens \cite{BrunovskySymmetry} solved the recognition problem of understanding when a given control system is SF equivalent to a Brunovsk\'y normal form. Their discovery was that the symmetry pseudogroups of Brunovsk\'y forms completely characterize such systems. In \cite{GSalgorithm}\cite{GSalgorithmExample}, Gardner and Shadwick devised an algorithm for transforming a SFL control system into Brunovsk\'y normal form. This approach used E. Cartan's method of equivalence and EDS theory and is considered the best method for SF linearization of a control system by control theorists. However, the GS algorithm does have some shortcomings. For instance, it only applies to systems that are SF equivalent to a Brunovk\'y normal form. This means it cannot fully address the question for control systems that are nonautonomous. Secondly, although the algorithm does indeed use the minimal number of integrations required to produce the SFT, one generally has to calculate the full structure equations in order to find the systems whose first integrals are used to construct contact coordinates. On the other hand, with Vassiliou's approach, one can first solve the recognition problem by computing the refined derived type and checking the integrability of the subbundles in (\ref{filt res}) and (\ref{filt fun}), as opposed to calculating the full symmetry pseudogroup of the control system. Procedure \textbf{contact} allows for one to be able to find general ESFTs instead of just SFTs, and furthermore, one need only compute first integrals of nonempty quotients of sequential subbundles of either (\ref{filt res}) or (\ref{filt fun}), plus the first integrals of the final subbundle in these filtrations (either the resolvent bundle or fundamental bundle). In this way, procedure \textbf{contact} also accomplishes the construction of an ESFT using the minimal number of integrations possible. Additionally, procedure \textbf{contact} is not restricted to control systems, and may be used to construct contact coordinates for any Goursat bundle (i.e. general diffeomorphism equivalence). The ability of procedure \textbf{contact} to produce ESFTs is especially important for the last steps of the cascade linearization process (see Chapter 4, Section 1). 
We would also like to remark that the construction of the remaining contact coordinates from procedure \textbf{contact} is reminiscent of the computation of higher order invariants from Olver's method of equivariant moving frames for Lie pseudogroups \cite{OlverPseudoMovingFrames}. Indeed, the author believes that Olver's methods would be yet another way to construct an algortihm for producing contact coordinates for a Goursat bundle. 
To further highlight the comparison of procedure \textbf{contact} and the GS algorithm, we present the following example of Hunt-Su-Meyer \cite{HuntSuMeyerLin}, which was then linearized via the GS algorithm in \cite{GSalgorithmExample}. Note that procedure contact can be executed in MAPLE or another suitable computer algebra program in a systematic way. For the sake of completeness, the following example has been computed with almost no details suppressed. 
\begin{ex}\cite{HuntSuMeyerLin}\cite{GSalgorithmExample}
\begin{align*}
&\frac{dx^1}{dt}=\sin(x^2),\,&\frac{dx^2}{dt}=\sin(x^3),\,&\,\phantom{==}&\frac{dx^3}{dt}=(x^4)^3+u^1,&\,\\
&\frac{dx^4}{dt}=x^5+(x^4)^3-(x^1)^{10},\,&\frac{dx^5}{dt}=u^2.\,&\,\phantom{==}&\,
\end{align*}
\end{ex}
First, we'll rewrite the control system as the distribution $\mathcal{V}=\{X,\partial_{u^1},\partial_{u^2}\}$, where
\begin{equation}
X=\partial_t+\sin(x^2)\,\partial_{x^1}+\sin(x^3)\,\partial_{x^2}+((x^4)^3+u^1)\,\partial_{x^3}+(x^5+(x^4)^3-(x^1)^{10})\,\partial_{x^4}+u^2\partial_{x^5}.
\end{equation}
\textbf{Step 1:} The derived flag of $\mathcal{V}$ is given by,
\begin{align}
\der{V}{1}&=\mathcal{V}+\{\partial_{x^3},\partial_{x^5}\},\\
\der{V}{2}&=\der{V}{1}+\{\partial_{x^2},\partial_{x^4}\},\\
\der{V}{3}&=\der{V}{2}+\{\partial_{x^1}\}=TM.
\end{align}
Hence $\mathcal{V}$ has derived length 3, $\vel{V}=\langle 2,2,1 \rangle$, and $\dec{V}=\langle 0,1,1 \rangle$. Since $\Delta_k=1$ we will implement Procedure B. Next we compute the Cauchy bundles for $\der{V}{1}$ and $\der{V}{2}$. Let 
\begin{equation}
C=TX+a^1\partial_{u^1}+a^2\partial_{u^2}+b^1\partial_{x^3}+b^2\partial_{x^5}\in \der{V}{1}
\end{equation} 
be a section of the Cauchy bundle of $\der{V}{1}$, where $T,b^1,b^2,c^1,$ and $c^2$ are smooth functions. Then 
\begin{equation}
[C,Y]\in \der{V}{1}\,\text{ for all }Y\in\der{V}{1}.
\end{equation}
It is enough to check $\mathcal{L}_C$ applied to the linearly independent sections generating $\der{V}{1}$. Doing so, we obtain
\begin{align}
[C,X]&=a^1\partial_{x^3}+a^2\partial_{x^5}+b^1[\partial_{x^3},X]+b^2[\partial_{x^5},X]\label{HSM CX}\\
\,&-(X(T)X+X(a^1)\partial_{u^1}+X(a^2)\partial_{u^2}+X(b^1)\partial_{x^3}+X(b^2)\partial_{x^5})\nonumber,\\
[C,\partial_{u^1}]&=T[X,\partial_{u^1}]-(T_{u^1}X+a^1_{u^1}\partial_{u^1}+a^2_{u^1}\partial_{u^2}+b^1_{u^1}\partial_{x^3}+b^2_{u^1}\partial_{x^5}),\\
[C,\partial_{u^2}]&=T[X,\partial_{u^2}]-(T_{u^2}X+a^1_{u^2}\partial_{u^1}+a^2_{u^2}\partial_{u^2}+b^1_{u^2}\partial_{x^3}+b^2_{u^2}\partial_{x^5}),\\
[C,\partial_{x^3}]&=T[X,\partial_{x^3}]-(T_{x^3}X+a^1_{x^3}\partial_{u^1}+a^2_{x^3}\partial_{u^2}+b^1_{x^3}\partial_{x^3}+b^2_{x^3}\partial_{x^5})\label{HSM T=0 1},\\
[C,\partial_{x^5}]&=T[X,\partial_{x^5}]-(T_{x^5}X+a^1_{x^5}\partial_{u^1}+a^2_{x^5}\partial_{u^2}+b^1_{x^5}\partial_{x^3}+b^2_{x^5}\partial_{x^5}).\label{HSM T=0 2} 
\end{align}
Equation (\ref{HSM CX}) implies that $b_1=b_2=0$, and either of equations (\ref{HSM T=0 1}) or (\ref{HSM T=0 2}) imply that $T=0$ since $[X,\partial_{x^3}]=-\cos(x^3)\partial_{x^2}$ and $[X,\partial_{x^5}]=-\partial_{x^4}$ are in $\der{V}{2}$ and not $\der{V}{1}$. Therefore, any section of the Cauchy bundle must be of the form $a^1\partial_{u^1}+a^2\partial_{u^2}$ for arbitrary functions $a^1$ and $a^2$. Therefore,
\begin{equation}
\Char{V}{1}=\{\partial_{u^1},\partial_{u^2}\}. 
\end{equation}
Now let $C=TX+b^1\partial_{x^3}+b^2\partial_{x^5}+c^1\partial_{x^2}+c^2\partial_{x^4}\in \der{V}{2}$, where $T,b^1,b^2,c^1,$ and $c^2$ are smooth functions, such that $C$ is a section of the Cauchy bundle of $\der{V}{2}$. Then applying the Lie derivative $\mathcal{L}_C$ to the generating sections of $\der{V}{2}$, we find
\begin{align}
[C,X]&=b^1[\partial_{x^3},X]+b^2[\partial_{x^5},X]+c^1[\partial_{x^2},X]+c^2[\partial_{x^4},X]\label{HSM CX 2}\\
\,&-(X(T)X+X(b^1)\partial_{x^3}+X(b^2)\partial_{x^5}+X(c^1)\partial_{x^2}+X(c^2)\partial_{x^4}),\nonumber\\
[C,\partial_{x^3}]&=T[X,\partial_{x^3}]-(T_{x^3}X+b^1_{x^3}\partial_{x^3}+b^2_{x^3}\partial_{x^5}+c^1_{x^3}\partial_{x^2}+c^2_{x^3}\partial_{x^4}),\\
[C,\partial_{x^5}]&=T[X,\partial_{x^5}]-(T_{x^5}X+b^1_{x^5}\partial_{x^3}+b^2_{x^5}\partial_{x^5}+c^1_{x^5}\partial_{x^2}+c^2_{x^5}\partial_{x^4}),\\
[C,\partial_{x^2}]&=T[X,\partial_{x^2}]-(T_{x^2}X+b^1_{x^2}\partial_{x^3}+b^2_{x^2}\partial_{x^5}+c^1_{x^2}\partial_{x^2}+c^2_{x^2}\partial_{x^4}),\label{HSM T=0 1 2}\\
[C,\partial_{x^4}]&=T[X,\partial_{x^4}]-(T_{x^4}X+b^1_{x^4}\partial_{x^3}+b^2_{x^4}\partial_{x^5}+c^1_{x^4}\partial_{x^2}+c^2_{x^4}\partial_{x^4}).\label{HSM T=0 2 2} 
\end{align}
Since 
\begin{align}
[\partial_{x^2},X]&=\cos(x^2)\partial_{x^1}\not\in\der{V}{2},\\
[\partial_{x^4},X]&=3(x^4)^2(\partial_{x^3}+\partial_{x^4})\in\der{V}{1},
\end{align}
equations (\ref{HSM CX 2}) and (\ref{HSM T=0 1 2}) force $c^1=0$ and $T=0$, respectively. Hence $C=b^1\partial_{x^3}+b^2\partial_{x^5}+c^2\partial_{x^4}$ for arbitrary smooth functions $b^1,b^2,$ and $c^2$. Notice also that there is no need to check sections of $\der{V}{2}$ with components from $\Char{V}{1}$, since $\Char{V}{1}\subset \Char{V}{2}$. Therefore,  
\begin{equation}
\Char{V}{2}=\{ \partial_{u^1},\partial_{u^2},\partial_{x^3},\partial_{x^4},\partial_{x^5}\}.
\end{equation}
From here, it is easily deduced that
\begin{align}
\inCharOne{V}&=\{\partial_{u^1},\partial_{u^2}\},\\
\mathrm{Char}\der{V}{2}_1&=\{\partial_{u^1},\partial_{u^2},\partial_{x^3},\partial_{x^5}\}. 
\end{align}
Thus the refined derived type of $\mathcal{V}$ is 
\begin{equation}
\mathfrak{d}_r(\mathcal{V})=[[3,0],[5,2,2],[7,4,5],[8,8]].
\end{equation}
Checking that the relations in Proposition \ref{refined derived type numbers} are true and seeing that all the bundles in (\ref{filt fun}) (up to the fundamental bundle) are integrable, we see that $\mathcal{V}$ must be a Goursat bundle. Furthermore, since $dt\in \ann\Char{V}{2}$; by Theorem \ref{Goursat ESFL} we deduce that $\mathcal{V}$ must be ESFL. Constructing the filtration of $T^*M$ (excluding the fundamental bundle) induced by $\mathcal{V}$, we find 
\begin{equation}
\coder{\Xi}{2}=\{dt,dx^1,dx^2\}\subset \coder{\Xi}{2}_1=\{dt,dx^1,dx^2,dx^4\}\subset\coder{\Xi}{1}=\{dt,dx^1,dx^2,dx^3,dx^4,dx^5\}=\coder{\Xi}{1}_0.
\end{equation}
\textbf{Step 2:} Notice that $t$ is a first integral of $\Char{V}{2}$ and that $X(t)=1$. Now the fundamental bundle $\Pi^2$ is given by 
\begin{equation}
\Pi^2=\{\partial_{u^1},\partial_{u^2},\partial_{x^2},\partial_{x^3},\partial_{x^4},\partial_{x^5}\}.
\end{equation}
\textbf{Steps 3 and 4:} There is only one non-empty quotient bundle to be computed for this step, 
\begin{equation}
\coder{\Xi}{2}_1/\coder{\Xi}{2}=\{dx^4\},
\end{equation}
and therefore $z_0^{1,2}=x^4$.\\
\textbf{Step 5:} From $\Pi^2$, we deduce that the other zeroth order contact variable is given by $z_0^{1,3}=x^1$ since $dt\wedge dx^1\neq0$. For simplicity, we shall relabel these zeroth order contact coordinates as $z_0^{1,2}=z_0^1$ and $z_0^{1,3}=z_0^2$.\\ 
\textbf{Step 6:} Applying the final step of the procedure, we conclude that the remaining contact coordinates are 
\begin{align}\label{HSM contact coords start}
z^1_1&=X(z_0^1)=x^5+(x^4)^3-(x^1)^{10},\\
z^1_2&=X(z^1_1)=u^2+3(x^4)^2(x^5+(x^4)^3-(x^1)^{10})-10(x^1)^9\sin(x^2),\\
z^2_1&=X(z^2_0)=\sin(x^2),\\
z^2_2&=X(z^2_1)=\cos(x^2)\sin(x^3),\\
z^2_3&=X(z^2_2)=-\sin(x^2)\sin^2(x^3)+((x^4)^3+u^1)\cos(x^2)\cos(x^3).\label{HSM contact coords end}
\end{align}
Thus $t$, $z^1_0=x^5$, $z^2_0=x^4$, and (\ref{HSM contact coords start})-(\ref{HSM contact coords end}) define a static feedback transformation of $\mathcal{V}$ to the Brunovsk\'y normal form $\beta^{\langle 0,1,1\rangle}$. 
Next, we will present an example in which Procedure A must be applied. 
\begin{ex}
Consider the following control system that arises from selecting $c_1=e_3=a_0=a_1=b_0=1,\,b_3=-1$, and all other constants equal to zero in Example \ref{Sluis}. 
\begin{align*}
\dot{x}^1&=x^1+u^1(1+x^1),\\
\dot{x}^2&=x^3+u^2(1-x^3),\\
\dot{x}^3&=u^1,\\
\dot{x}^4&=u^2.
\end{align*}
The control system as a distribution is $\mcal{V}=\{X,\partial_{u^1},\partial_{u^2}\}$, where 
\begin{equation}
X=\partial_t+(x^1+u^1(1+x^1))\partial_{x^1}+(x^3+u^2(1-x^3))\partial_{x^2}+u^1\partial_{x^3}+u^2\partial_{x^4}.
\end{equation}
The distribution $\mcal{V}$ is a Goursat bundle with type $\kappa=\langle 0,2 \rangle$, and we will find contact coordinates by applying Procedure A.
\end{ex}
\textbf{Step 1:} First we calculate the derived flag, and then the filtration (\ref{filt res}), stopping short of the resolvent bundle. Using MAPLE, we find that the derived flag is given by
\begin{align*}
\der{V}{1}&=\mcal{V}+\{(1+x^1)\partial_{x^1}+\partial_{x^3},(1-x^3)\partial_{x^2}+\partial_{x^4}\},\\
\der{V}{2}&=\der{V}{1}+\{\partial_{x^1}+(1-u^2)\partial_{x^2},u^1\partial_{x^2}\}\\
\,&=TM.
\end{align*}
Using MAPLE to compute Cauchy bundles, we find
\begin{equation}
\inCharOne{V}=\{\partial_{u^1},\partial_{u^2}\}=\Char{V}{1}=\{\partial_{u^1},\partial_{u^2}\}\subset\mcal{R}_{\bar{\Sigma}}(\der{V}{1}).
\end{equation}
\tab\textbf{Step 2:} Next we calculate the resolvent bundle $\mcal{R}_{\bar{\Sigma}}(\der{V}{1})$. First, we need to compute the quotient $\der{V}{1}/\Char{V}{1}=\bar{\mcal{V}}^{(1)}$. Doing so gives
\begin{equation}
\bar{\mcal{V}}^{(1)}=\{X,Y_1=(1+x^1)\partial_{x^1}+\partial_{x^3},Y_2=(1-x^3)\partial_{x^2}+\partial_{x^4}\}.
\end{equation}
Now let $E=[a_0X+a_1Y_1+a_2Y_2]\in \mathbb{P}\bar{\mcal{V}}^{(1)}$ and let
\begin{equation}
\overline{TM}=\{X,Y_1,Y_2,Y_3=\partial_{x^1}+(1-u^2)\partial_{x^2},Y_4=u^1\partial_{x^2}\},
\end{equation}
so that $\overline{TM}/\bar{\mcal{V}}^{(1)}=\{Y_3,Y_4\}$. Note that $Y_1=[\partial_{u^1},X], Y_2=[\partial_{u^2},X], Y_3=[Y_1,X]$, and $Y_4=[Y_2,X]$. Now we can compute the polar matrix of $E$ (see Definition \ref{singular polar}),
\begin{equation}
\sigma(E)=
\begin{bmatrix}
a_1& -a_0 & 0\\
a_2& a_2/u^1 & -a_0-a_1/u^1
\end{bmatrix}.
\end{equation}
The polar matrix has less than generic rank when $a_0=-a_1/u^1$. This means that the singular bundle is given by
\begin{equation}
\bar{\mcal{B}}=\{X-u^1Y_1,Y_2\}\subset\mathbb{P}\bar{\mcal{V}}^{(1)},
\end{equation}
and therefore the resolvent bundle is
\begin{equation}
\mcal{R}_{\bar{\Sigma}}(\mcal{V}^{(1)})=\{X-u^1Y_1,Y_2,\partial_{u^1},\partial_{u^2}\}.
\end{equation}
Notice that the resolvent bundle is integrable.\\
\tab\textbf{Steps 3 and 4:} We see that there are no nontrivial quotient bundles $\coder{\Xi}{i}_{i-1}/\coder{\Xi}{i}$, and hence no fundamental functions of order less than $2$.\\
\tab\textbf{Step 5:} Now we compute the first integrals of the resolvent bundle. By use of MAPLE, we find the first integrals to be
\begin{equation}
F:=\{x=x^3,z^1_0=x^1e^{-t},z^2_0=(x^4-t)x^3+x^2-x^4\}.
\end{equation}
\tab\textbf{Step 6:} Let $Z=\frac{1}{u^1}X$, so that $Z(x)=1$. Then we can construct the remaining contact coordinates as
\begin{align}
z^1_1&=Z(z^1_0)=e^{-t}(1+x^1),\\
z^1_2&=Z(z^1_1)=e^{-t}\left(1+x^1-\frac{1}{u^1}\right),\\
z^2_1&=Z(z^2_0)=x^4-t,\\
z^2_2&=Z(z^2_1)=u^2-\frac{1}{u^1}.
\end{align}
Thus we have found local contact coordinates for this Goursat bundle $\mcal{V}$.\\

\tab Notice that although we have found contact coordinates that put the Goursat bundle into normal form, it is \textit{not} via an ESFT. Indeed, this is not possible since $dt\not\in\Upsilon_{\bar{\Sigma}}(\der{V}{1})$, and therefore the time coordinate cannot be singled out as the parameter for integral curves to $\mcal{V}$ in Brunovsk\'y normal form. However, as mentioned in Section 3 of Chapter 1, this example can be prolonged twice to  an ESFL system. The author has observed this property in a few other examples of control systems and conjectures the following:
\begin{conj}
If a control system with at least 2 controls is a Goursat bundle, but cannot be transformed to Brunovsk\'y normal form via an ESFT, then there exists a DF linearization of the control system. 
\end{conj}


\subsection{Background on the Euler Operator}
In Chapter 4, we will introduce an operator known as a \textit{truncated} Euler operator, which will be used to establish the main results of this thesis. In this section, we will introduce some basic properties of the Euler operator from the theory of the calculus of variations. There is a vast literature on the calculus of variations, and much of the theory goes beyond our needs in this thesis. We primarily consider the geometric approach taken in \cite{OlverLieBook} and to some extent \cite{VarBi}. Furthermore, we will restrict ourselves to real valued functions of a single real variable, but we mention that generalizations are straightforward and can be found in any of the works referenced in this section. The motivating problem in the calculus of variations is given by the following: Let $L:\mathbb{R}\to\mathbb{R}$ be a smooth function and $A=(a,c)$ and $B=(b,d)$ be two fixed points in the plane. Then for what functions $u(t)$ whose graphs connect $A$ and $B$ does the integral,
\begin{equation}\label{lagrangian functional}
\sL[u]=\int_a^bL(t,u(t),\dot{u}(t),\ldots,u^{(n)}(t))\,dt, 
\end{equation}
attain a minimum or maximum? In the physics literature, $\sL$ is a functional associated to some physical system, and asking that there be a smooth function $u(t)$ that minimizes this functional is to say the physical system possesses a \textit{principle of least action}.  
\begin{defn}
A function $u(t)$ is an \textbf{extremal} of $\sL$ if $\sL[u]$ is a local maximum or local minimum on a space of functions with a given topology containing $u(t)$. 
\end{defn}
Typically, the function space in question is some type of Banach space, and there are many considerations from functional analysis one has to check to ensure that extremals exist. We start with an analogy to optimization of real valued functions. We will compute a type of derivative of the functional and subsequently check if any associated critical points give rise to optimal solutions. The precise arguments needed to make this idea rigorous will not be presented here, but can be found in any introductory text on the calculus of variations such as \cite{VarCalc}. The derivative we will calculate is called a variational derivative. 
\begin{defn}
The \textit{variational derivative} $\delta\sL[u]$ is defined by the condition that
\begin{equation}
\frac{d}{d\epsilon}\Big|_{\epsilon=0}\sL[u+\epsilon v]=\int_a^b\delta\sL[u]\,v(t)\,dt
\end{equation}
for any smooth function $v(t)$ such that $v(a)=v(b)=0$. 
\end{defn}
Analogously with optimization of functions of real variables, we have the following proposition.
\begin{prop}
If $u(t)$ is an extremal for $\sL[u]=\int_a^bL(t,u,\dot{u},\ldots,u^{(n)})\,dt$, then 
\begin{equation}
\delta\sL[u(t)]=0
\end{equation}
for all $t\in[a,b]$. 
\end{prop}
A simple example is that of the arc length functional
\begin{equation}
\int_a^b\sqrt{1+\dot{u}^2}\,dt,
\end{equation}
which returns the length of the graph of the function $u(t)$ connecting two fixed points $A=(a,u(a))$ and $B=(b,u(b))$. Indeed, the idea of a variational derivative is to ``perturb'' the curve $u(t)$ by $\epsilon v(t)$ for any smooth function $v(t)$ such that $v(a)=v(b)=0$, and some small $\epsilon$. First we compute the derivative of $\sL[u+\epsilon v]$ with respect to $\epsilon$ and evaluate at $\epsilon=0$. Doing so, we obtain
\begin{equation}
\frac{d}{d\epsilon}\Big|_{\epsilon=0}\sL[u+\epsilon v]=\int_a^b \frac{\dot{v}\dot{u}}{\sqrt{1+\dot{u}^2}}\,dt, 
\end{equation}
so that upon performing integration by parts, we arrive at
\begin{equation}
\frac{d}{d\epsilon}\Big|_{\epsilon=0}\sL[u+\epsilon v]=-\int_a^b \frac{\ddot{u}}{(1+\dot{u}^2)^{3/2}}\,v(t)\,dt, 
\end{equation}
and hence the variational derivative is $\delta\sL[u]=-\frac{\ddot{u}}{(1+\dot{u}^2)^{3/2}}$. The only way for a function $u(t)$ to be an extremal is if the graph of $u(t)$ is the line segment connecting the two points $A$ and $B$. We now repeat this process for the more general case of (\ref{lagrangian functional}). Doing so, we obtain
\begin{equation}\label{derivation of EO}
\begin{aligned}
\frac{d}{d\epsilon}\Big|_{\epsilon=0}\sL[u+\epsilon v]&=\int_a^b \sum_{i=0}^nv^{(i)}\partiald{L}{u^{(i)}}\,dt\\
\,&=\int_a^b\left(\sum_{i=0}^n (-1)^i\frac{d^i}{dt^i}\partiald{L}{u^{(i)}}\right)v(t)\,dt, 
\end{aligned}
\end{equation}
where $v(t)$ is any smooth function that forces all boundary terms in the repeated integration by parts to vanish. Here, $\frac{d}{dt}$ is the total derivative in the multivariable calculus sense. The variational derivative is therefore
\begin{equation}
\delta\sL[L]=\sum_{i=0}^n (-1)^i\frac{d^i}{dt^i}\partiald{L}{u^{(i)}}. 
\end{equation}
It turns out that we can use the language of jets to describe the the variational derivative in a more geometric way. We will not go too deeply into this subject here; however, \cite{VarBi} is an excellent reference for a modern geometric formulation of the calculus of variations. 
\begin{defn}
Given a function $L:J^n(\mathbb{R},\mathbb{R})\to \mathbb{R}$, we call the functional 
\begin{equation}
\sL[u]=\int_a^bL(t,z_0,z_1,\ldots,z_n)\,dt
\end{equation}
a \textbf{cost functional} or \textbf{action} and the function $L$ a \textbf{Lagrangian}. Here, 
\end{defn}
\begin{defn}
The operator defined by
\begin{equation}
D_t=\partiald{\,}{t}+\sum_{i=0}^\infty z_{i+1}\partiald{\,}{z_i}
\end{equation}
is the \textbf{total derivative operator}, and may be considered as a map from $J^n(\mathbb{R},\mathbb{R})$ to $J^{n+1}(\mathbb{R},\mathbb{R})$ for all $n\geq0$. 
\end{defn}
Properly, the total derivative operator is a vector field on an infinite jet bundle. We will not go through the details here; however, we want to emphasize that the action of the total derivative operator on a function $f\in C^\infty\left( J^n(\mathbb{R},\mathbb{R})\right)$, for some non-negative integer $n$, produces no issues of convergence since the function $f$ will have no dependence on jet variables with order greater than $n$.
When the total derivative operator is applied to $L$, we find that $D_t(L)\circ (j^{n+1}f(t))$ agrees with $\frac{d}{dt}\left(L(j^n_tf)\right)$. The variational derivative can also be written in terms of jet coordinates, and we give it a special name. 
\begin{defn}
The variational derivative of a functional with Lagrangian $L$ is obtained by applying an operator to $L$ called the \textbf{Euler operator}. It is given by 
\begin{equation}
E(L)=\sum_{i=0}^\infty(-1)^iD^i_t\left(\partiald{L}{z_i}\right). 
\end{equation}
If $L$ is a function on $J^n(\mathbb{R},\mathbb{R})$, then $E(L)$ defines a function on $J^{2n}(\mathbb{R},\mathbb{R})$. When $E(L)$ is restricted to the $2n$-jet of some function $f:\mathbb{R}\to\mathbb{R}$, then the equation 
\begin{equation}
E(L)\circ j^{2n}f(t)=0
\end{equation}defines an order $2n$ ODE known as the \textbf{Euler-Lagrange equation}. 
\end{defn}  
Although variational questions are of deep interest, we will be primarily concerned with properties of the total derivative operator and the Euler operator. In Chapter 4 we will introduce $\textit{truncated}$ versions of these operators, and it important to understand how they differ from each other.
\begin{prop}\label{tot kernel const}
The total derivative operator of a function $f$ on some connected subset of $J^n(\mathbb{R},\mathbb{R})$ is zero if and only if $f$ is constant. 
\end{prop}
\begin{proof}
Let $f\in C^\infty(J^n(\mathbb{R},\mathbb{R}))$. Then 
\begin{equation}
D_tf=\partiald{f}{t}+z_1\partiald{f}{z_0}+\cdots+z_{n+1}\partiald{f}{z_n}.
\end{equation}
If $f$ is constant, then $D_tf=0$ immediately. Thus, assume that $D_tf=0$. Indeed, this means that 
\begin{equation}\label{Dt const}
z_{n+1}\partiald{f}{z_n}=-\left(\partiald{f}{t}+z_1\partiald{f}{z_0}+\cdots+z_n\partiald{f}{z_{n-1}}\right).
\end{equation}
Since $f$ has no dependence on $z_{n+1}$, the right hand side of (\ref{Dt const}) has no dependence on $z_{n+1}$, so we must have $\partiald{f}{z_n}=0$. This means that $f\in C^\infty(J^{n-1}(\mathbb{R},\mathbb{R}))$. We can then iterate this argument to conclude that $\partiald{f}{z_i}=0$ for all $0\leq i\leq n$. On the final iteration we can then conclude that $\partiald{f}{t}=0$ as well. Therefore, $f$ must be a constant.  
\end{proof}
\begin{thm}
Let $E$ be the Euler operator. Then
\begin{equation}
\ker E=\{f : f=D_t\,g,\,\,\mathrm{where}\,\,g\in C^\infty(J^n(\mathbb{R},\mathbb{R}))\,\,\mathrm{for}\,\mathrm{any}\,\,n\geq0\}. 
\end{equation}
\end{thm}
We will not prove this theorem; however, a proof may be found in Chapter 4, Section 1 of \cite{OlverLieBook}. Importantly, this means that two different Lagrangians may have the same Euler-Lagrange equations. Indeed, two Lagrangians
\begin{equation}
\sL[u]=\int_a^bL(t,z_0,\ldots,z_n)\,dt
\end{equation}
and 
\begin{equation}
\tilde{\sL}[u]=\int_a^b\tilde{L}(t,z_0,\ldots,z_n)\,dt
\end{equation}
have the same Euler-Lagrange equation if and only if $\tilde{L}(t,z_0,\ldots,z_n)=L(t,z_0,\ldots,z_n)+D_tf$ for some smooth function $f$ on $J^n(\mathbb{R},\mathbb{R})$, since $E(L)=E(\tilde{L})$. 

\section{Invariant Control Systems and Reconstruction \label{chap:three}}
In this chapter we will explore certain phenomena of control systems that admit particular kinds of symmetries. In particular, we will be interested in studying quotient control systems that arise from a control system's ``special" symmetries. We can use solutions for the quotient control system to construct individual trajectories to the original control system by essentially solving an ODE that arises from the group action. We are particularly interested in the case that the associated quotient control system is ESFL. 
\subsection{Exterior Differential Systems with Symmetry}
In this section we discuss some important results concerning EDS with symmetry. The material in this section is primarily from \cite{AndersonFelsBacklund} and \cite{AndersonFelsGroupInvSol}. Recall from Definition \ref{EDS symmetry} that a vector field $X$ is an infinitesimal symmetry of an EDS on a manifold $M$ if the Lie derivative with respect to $X$ of any form in $\mathcal{I}$ is in $\mathcal{I}$. It is possible that an EDS has no nontrivial symmetries whatsoever, and this will usually be the case. However, control systems that arise from application will usually have plenty of symmetries because of some underlying physics. So in terms of applications, studying control systems with symmetry can be quite enlightening. It is also possible that an EDS has enough symmetries such that the associated group of symmetries $G$ for the EDS has the general structure of a Lie pseudogroup, as in Definition \ref{pseudo def}. We will always restrict our attention to finite dimensional \textit{subgroups} of symmetries. In particular, we will choose subgroups of small enough dimension so that the subgroup in question has strictly smaller dimension than the dimension of the manifold $\bX(M)$. The reason for this restriction will become clear in our subsequent discussion. Our first goal is to recognize when a linear Pfaffian system $\mathcal{I}$ on a manifold $M$ with a finite-dimensional group of symmetries $G$ has the property that on the quotient manifold $M/G$, the forms in $\mathcal{I}$ descend to another linear Pfaffian system $\mathcal{I}/G$, called the \textit{quotient system}. 
\begin{defn}\cite{AndersonFelsGroupInvSol}
Let $\mcal{I}$ be an EDS with symmetry group $G$. The \textbf{quotient system} or \textbf{reduced system} of $\mcal{I}$ is defined as
\begin{equation}
\mathcal{I}/G=\{\theta\in\Omega(M/G): \pi^*\theta\in\mathcal{I}\}, 
\end{equation}
where $\pi:M\to M/G$ is the orbit projection map. 
\end{defn}
We are specifically interested in the case that the quotient of a Pfaffian system is again a Pfaffian system. To that end, we need the following definition.  
\begin{defn}\cite{AndersonFelsGroupInvSol}
Let $\Gamma$ be a Lie algebra of infinitesimal symmetries of a Pfaffian system $\mathcal{I}=\langle I\rangle$. Then we say that $\Gamma$ is \textbf{transverse} to  $\mathcal{I}$ if  $\Gamma\cap\mathrm{ann}\,\,I=\{0\}$. We say that the symmetries are \textbf{strongly transverse} if  $\Gamma\cap\mathrm{ann}\,\,I^{(1)}=\{0\}$. 
\end{defn}  
\begin{thm}\cite{AndersonFelsGroupInvSol}
Let $M$ be a manifold and consider a Pfaffian system $\mathcal{I}$ on $M$ with finite dimensional Lie group of symmetries $G$ such that $\dim{G}<\dim{M}$. Furthermore, assume that the Lie algebra of infinitesimal symmetries $\Gamma$ for the action of $G$ on $M$ is strongly transverse to $\mathcal{I}$. Then the quotient system $\mcal{I}/G$ is also a Pfaffian system. 
\end{thm}
\begin{ex}\label{Quotient Example}
Consider the following 5 state and 2 control Pfaffian system,
\begin{align}
\omega=\langle \theta^1,\theta^2,\theta^3,\theta^4,\theta^5\rangle
\end{align}
with 
\begin{align}
\theta^1&= dx^1-((x^2)^2+x^1f(t,x^3,x^4,x^5,u^2))\,dt,\\
\theta^2&=dx^2-x^2f(t,x^3,x^4,x^5,u^2)\,dt,\\
\theta^{i+2}&=dx^{i+2}-g^i(t,x^3,x^4,x^5,u^2)\left(x^2e^{-u^1}\right)^{a_i}\,dt\,\,\mathrm{ for }\,\,1\leq i\leq 3,
\end{align}
where $a_i$ are constants not all zero. Then the Lie algebra 
\begin{equation}\label{affine alg}
\Gamma=\{x^2\partial_{x^1},2x^1\partial_{x^1}+x^2\partial_{x^2}+\partial_{u^1}\}
\end{equation}
of infinitesimal symmetries of $\omega$ is strongly transverse to $\omega$ (for generic functions $f$ and $g^i$), and hence the quotient system is a Pfaffian system as well. 
\end{ex}
It is reasonably direct to check that $\Gamma$ forms a Lie algebra of infinitesimal symmetries for $\omega$. Denote the generating vector fields of \ref{affine alg} as $X_1$ and $X_2$ respectively. Then
\begin{align}
\mathcal{L}_{X^1}\theta^1&=dx^2-x^2f(t,x^3,x^4,x^5,u^2)\,dt=\theta^2,\\
\mathcal{L}_{X^2}\theta^1&=2dx^1-2((x^2)^2+x^1f(t,x^3,x^4,x^5,u^2))\,dt=2\theta^1,\\
\mathcal{L}_{X^1}\theta^2&=0,\\
\mathcal{L}_{X^2}\theta^2&=dx^2-x^2f(t,x^3,x^4,x^5,u^2)\,dt=\theta^2,\\
\mathcal{L}_{X^i}\theta^{j+2}&=0\text{ for }i=1,2,\text{ and }j=1,2,3,
\end{align}
the right hand sides of which all clearly belong to $\omega$. Next we need to check that the symmetries are strongly transverse to $\omega$. Let $\mathcal{V}=\mathrm{ann}\,\omega$, $A^i=g^i(t,x^3,x^4,x^5,u^2)\left(x^2e^{-u^1}\right)^{a_i}$, and notice that 
\begin{multline}
\mathcal{V}=\{\partial_t+((x^2)^2+x^1f(t,x^3,x^4,x^5,u^2))\partial_{x^1}+x^2f(t,x^3,x^4,x^5,u^2)\partial_{x^2}\\
+A^3\partial_{x^3}+A^4\partial_{x^4}+A^5\partial_{x^5},\partial_{u^1},\partial_{u^2}\}.
\end{multline}
Next we compute the derived system of $\mathcal{V}$:
\begin{equation}
\mathcal{V}^{(1)}=\mathcal{V}+\{A^3_{u^1}\partial_{x^3}+A^4_{u^1}\partial_{x^4}+A^5_{u^1}\partial_{x^5},\,x^1f_{u^2}\partial_{x^1}+x^2f_{u^2}\partial_{x^2}+A^3_{u^2}\partial_{x^3}+A^4_{u^2}\partial_{x^4}+A^5_{u^2}\partial_{x^5}\}.
\end{equation}
So for sufficiently generic functions $g^1,g^2,$ and $g^3$, we can see that $\Gamma$ is strongly transverse to $\omega$. If $G$ is the 2-dimensional Lie group whose action on $M$ is defined by the flows of $\Gamma$, then $\omega/G$ on $M/G$ will be a Pfaffian system. Local coordinates on $M/G$ can be defined in terms of the invariant functions of $\Gamma$. Indeed, 
\begin{equation}
\mathrm{Inv}\,\Gamma=\left\{t, y^1=x^3, y^2=x^4, y^3=x^5, v^1=u^2,v^2=x^2e^{-u^1}\right\}
\end{equation}
and hence they may also be chosen to represent local coordinates on the quotient manifold $M/G$. In these coordinates on $M/G$ we find that 
\begin{equation}\label{quotient h}
\omega/G=\langle dy^1-h^1(v^2)^{a_1}\,dt,  dy^2-h^2(v^2)^{a_2}\,dt, dy^3-h^3(v^2)^{a_3}\,dt\rangle,
\end{equation}
where $h^i=h^i(t,y^1,y^2,y^3,v^1)$, so that $\pi^*h^i=g^{i}$. Not only is $\omega/G$ a Pfaffian system, but it is also representative of a control system on $M/G$ with 3 states and 2 controls. For this example, one can pick the $a^i$ and $h^i$ to make $\omega/G$ fit into nearly any of the normal forms presented in \cite{Wilkens3s2cEquiv} (the only exceptions are normal form III of Theorem 2 and possibly the classes determined by case IV of Theorems 2 and case III of Theorem 1). Example \ref{Quotient Example} has, or can be made to have, other nice properties which we explore in the forthcoming sections of this chapter. 
\subsection{Control Admissible Symmetry Groups}
Given a control system that can be represented by a completely nonintegrable distribution or Pfaffian system on a manifold $M$, there may be many different kinds of symmetry groups. However, we are interested in a particular class of symmetries that are specific to the study of control systems. We want to make sure that any action by a control system's symmetries will not mix up time, state variables, and control variables in any way inappropriate for control theory purposes. To be precise, we present the following definition. 
\begin{defn}\cite{VassiliouSICON}
Let $M\cong_{loc}\mathbb{R}\times\bX(M)\times\bU(M)$, $\omega$ a Pfaffian system representing a control system, and let $\mu:G\times M\to M$ be a Lie transformation group with Lie algebra $\Gamma$ that has the following properties:
\begin{enumerate}\label{control symmetries}
\item $\Gamma$ is a Lie algebra of infinitesimal symmetries of $\omega$,
\item the action of $G$ on $M$ is free and regular,
\item $\mu^*_gt=t$, for all $g\in G$, where $\mu_g(x)=\mu(g,x)$,
\item if $\pi:M\to\mathbb{R}\times\bX(M)$ is the projection map, then $\mathrm{rank}\,\left({d\pi(\Gamma)}\right)=\dim(G)$.
\end{enumerate}
We say that such a group $G$ is a \textbf{control admissible symmetry group}. We may abuse this language somewhat by using the word ``symmetries'' to reference either a control admissible symmetry group or its infinitesimal generators. 
\end{defn}
In particular, items (3) and (4) of Definition \ref{control symmetries} force elements of $G$ to act as ESFTs. It turns out that we have already encountered an example of such a symmetry group. 
\begin{ex}\label{quotient action}
The symmetry group $G$ generated by $\Gamma=\{x^2\partial_{x_1},2x^1\partial_{x^1}+x^2\partial_{x^2}+\partial_{u^1}\}$ for the control system $\omega$ in Example \ref{Quotient Example} is an example of a control admissible symmetry group.  
\end{ex}
Let $(\epsilon_1,\epsilon_2)$ be local coordinates on the Lie group associated to $\Gamma$. To compute the group action from $\Gamma$, we simply find the flows of each generator in $\Gamma$ on $M$. These flows are
\begin{align}
\Phi^1_{\epsilon_1}(t,\bx,\bu)&=(t, x^1+\epsilon_1x^2,x^2,x^3,x^4,x^5,u^1,u^2),\\
\Phi^2_{\epsilon_2}(t,\bx,\bu)&=(t,x^1e^{2\epsilon_2}, x^2e^{\epsilon_2}, x^3,x^4,x^5,u^1+\epsilon_2,u_2),
\end{align}
so that the action may be written as the composition
\begin{align}
\mu((\epsilon_1,\epsilon_2),(t,\bx,\bu))&=\Phi^1_{\epsilon_1}\circ\Phi^2_{\epsilon_2}(t,\bx,\bu)\\
\,&=(t,x^1e^{2\epsilon_2}+x^2\epsilon_1e^{\epsilon_2},x^2e^{\epsilon_2},x^3,x^4,x^5,u^1+\epsilon_2,u^2).
\end{align}
This action has exactly the form of an ESFT for any $\epsilon_1$ and $\epsilon_2$. Thus, the action by any element $g\in G$ on $M$ is by an ESFT. Furthermore, the control admissible symmetry group $G$ is diffeomorphic to $\text{Aff}(\mathbb{R})$, since the 2D Lie algebra $\Gamma$ is not abelian. We remark that this group $G$ may not be the entire (possibly pseudo-) group of control admissible symmetries for $\omega$. A class of examples of such systems are those that are ESFL, since they will necessarily be invariant under the pseudogroup of contact transformations of their equivalent Brunovsk\'y normal form. Interestingly, there is at least one control system that is provably \textit{not} ESFL and has an infinite dimensional control admissible symmetry group. Consider the control system in 7 states and 3 controls which is given as Example 2 in \cite{BC3control}.
\begin{ex}\label{BC system} The Battilotti-Califano (BC) system is the 7 state, 3 control, Pfaffian system generated by the forms
\begin{align}
\,&\theta^1=dx^1-u^1\,dt,\,&\theta^2=dx^2-x^1\,dt,\,&\,\phantom{==}&\theta^3=dx^3-(x^2+x^6+x^2u^1)\,dt&\,,\\
\,&\theta^4=dx^4-(u^2+x^1u^3)\,dt,\,&\theta^5=dx^5-x^4\,dt,\,&\,\phantom{==}&\theta^6=dx^6-(x^5+x^2x^4)\,dt&\,,\\
\,&\theta^7=dx^7-u^3\,dt.\,&\,&\,&\,&\,
\end{align}
The control admissible symmetry group of the BC system is generated by the infinitesimal symmetries $\Gamma=\{X_1,X_2,X_3\}$, where
\begin{align}
X_1=&\frac{t^2}{2}\partial_{x^3}+\partial_{x^5}+t\partial_{x^6}+F\partial_{x^7}-x^1K\partial_{u^2}-K\partial_{u^3},\\
X_2=&t\partial_{x^3}+\partial_{x^6}+F\partial_{x^7}-x^1K\partial_{u^2}-K\partial_{u^3},\\
X_3=&\partial_{x^3}+F\partial_{x^7}-x^1K\partial_{u^2}-K\partial_{u^3},
\end{align}
and 
\begin{equation}
K=(x^2+x^2u^1+x^6)F_3+(u^3x^1+u^2)F_4+(x^2x^4+x^5)F_6+u^1F_1+x^1F_2+x^4F_5+u^3F_7+F_t)
\end{equation}
with $F$ any real-valued smooth function on $M$ that has no dependence on the controls, $F_i=\partiald{F}{x^i}$, and $F_t=\partiald{F}{t}$.
\end{ex}
Thus we see that the BC system has an infinite dimensional control admissible symmetry group due to the dependence on $F$. Furthermore, using procedure \textbf{contact} in Maple, we find that the refined derived type of the BC system is [[4,0],[7,3,4],[9,5,5],[11,11]], which does not agree with the refined derived type of a Goursat bundle presented in Proposition \ref{refined derived type numbers}. Hence the BC system is not ESFL. 

The following theorem is an important result that guarantees that the quotient of a Pfaffian system by a control admissible symmetry group is again representative of a control system. 
\begin{thm}\cite{VassiliouSICON}
Let $G$ be a control admissible symmetry group of a control system $\omega$ on a manifold $M$ such that $G$ is strongly transverse to $\omega$ and $\dim(G)<\dim\bX(M)$. Then the quotient system $\omega/G$ is a control system on $M/G$ and has the same number of controls as $\omega$. 
\end{thm}
One can easily verify that Example \ref{Quotient Example} has the property that its quotient system is again a control system. It is also true that there are subgroups of the infinite dimensional control symmetry group of the BC system that are strongly transverse to the BC system. Indeed, choose $H$ to be the subgroup of control admissible symmetries of the BC system are generated by
\begin{equation}\label{BC sym}
\Gamma_H=\left\{\frac{t^2}{2}\partial_{x^3}+\partial_{x^5}+t\partial_{x^6}+\partial_{x^7},t\partial_{x^3}+\partial_{x^6}+\partial_{x^7},\partial_{x^3}+\partial_{x^7}\right\},
\end{equation}
which arises from choosing $F=1$. The annihilator of $\omega^{BC}$ is given by
\begin{equation}
\mathcal{W}=\{X,\partial_{u^1},\partial_{u^2},\partial_{u^3}\},
\end{equation}
where
\begin{multline}
X=\partial_t+u^1\partial_{x^1}+x^1\partial_{x^2}+(x^2+x^2u^1+x^6)\partial_{x^3}+(x^1u^3+u^2)\partial_{x^4}+x^4\partial_{x^5}\\
+(x^2x^4+x^5)\partial_{x^6}+u^3\partial_{x^7}.
\end{multline}
Thus, 
\begin{equation}
\mathcal{W}^{(1)}=\mathcal{W}+\{\partial_{x^1}+x^2\partial_{x^3},\,\partial_{x^4},\,x^1\partial_{x^4}+\partial_{x^7}\},
\end{equation}
and one can now see that $\Gamma_H$ is strongly transverse to $\omega^{BC}$. Hence $\omega^{BC}/H$ is a control system of 4 states and 3 controls. We see also that the ESFTs generated by $\Gamma_H$ are
\begin{multline}\label{BC Flow}
\Phi^{BC}_{\epsilon_1,\epsilon_2,\epsilon_3}(t,\bx,\bu)=(t,x^1,x^2,\frac{t^2}{2}\epsilon_1+t\epsilon_2+\epsilon_3+x_3,x^4,\epsilon_1+x^5,t\epsilon_1+\epsilon_2+x^6,\\
\epsilon_1+\epsilon_2+\epsilon_3+x^7,u^1,u^2,u^3).
\end{multline}
Interestingly, these maps are necessarily ESFTs, as opposed to SFTs, despite the fact that the original system is autonomous. 
\subsection{ESFL Quotient Systems}
In general, we have seen that we can find reductions of control systems and obtain control systems again provided that the group action belongs to the Lie psuedogroup of ESFTs. The resulting control systems in Example \ref{Quotient Example} are of 3 states and 2 controls and can therefore be classified according to \cite{Wilkens3s2cEquiv}. However, we need not restrict ourselves to control systems whose quotients will fit into a broad classification scheme (as in general, there are presently none for higher numbers of states and controls). There is, however, a nice class of control systems that were introduced in Chapter 1 that are generated by Brunovsk\'y/contact differential forms as discussed in Chapter 2. Given a control system $\omega$ with control admissible symmetry group $G$, one may wonder when the resulting control system $\omega/G$ is ESFL. This is the content of \cite{VassiliouSICON}, and we will state some of those key results here. 
\begin{defn}\cite{VassiliouSICON}
A \textbf{relative Goursat bundle} $\mathcal{V}\subset TM$ is a distribution of derived length $k>1$ that has the following properties:
\begin{enumerate}
\item the type numbers satisfy the same relations as those listed in Proposition \ref{refined derived type numbers},
\item the $\inChar{V}{i}$ are all integrable,
\item if $\Delta_k>1$, then $\mathcal{V}^{(k-1)}$ determines an integrable Weber structure whose resolvent bundle is of rank $\Delta_k+\chi^{k-1}$.
\end{enumerate}
\end{defn}
Note that a relative Goursat bundle may have a non-trivial Cauchy bundle. 
Next we state the most important theorem of this section. 
\begin{thm}\label{ESFL Quotient}\cite{VassiliouSICON}
Let $\Gamma$ be the Lie algebra of a strongly transverse, control admissible symmetry group $G$ of a control system $\mathcal{V}$. If $\mathrm{Char}(\mathcal{V})=\{0\}$ and $\hat{\mcal{V}}:=\mathcal{V}\oplus\Gamma$ is a relative Goursat bundle, then the quotient system $\omega/G$ is a control system that is locally equivalent to a Brunovsk\'y normal form via a diffeomorphism. Furthermore, one can choose the diffeomorphism to be an ESFT if points (1) and (2) of Theorem \ref{Goursat ESFL} are true for $\hat{\mcal{V}}$.
\end{thm}
This gives a direct way to check whether a control system $\omega$ with control admissible symmetry group $G$ admits ESFL quotients. 
\begin{defn}
A relative Goursat bundle will be called an \textbf{ESF relative Goursat bundle} if it satisfies the points (1) and (2) of Theorem \ref{Goursat ESFL}. 
\end{defn}
We will once again use the examples that have previously been explored in this chapter. Indeed, let $\omega$ be the control system from Example \ref{Quotient Example} with $f=e^{u^2}$, $g^1=\ln\left(1+(u^2)^2\right)$, $g^2=\sin(x^5)$, $g^3=x^3$, $a_1=5,a_2=0$, and $a_3=1$. The control system is now generated by
\begin{align}
\theta^1&= dx^1-((x^2)^2+x^1e^{u^2})\,dt,\\
\theta^2&=dx^2-x^2e^{u^2}\,dt,\\
\theta^3&=dx^3-\ln\left(1+(u^2)^2\right)(x^2)^5e^{-5u^1}\,dt,\\
\theta^4&=dx^4-\sin(x^5)\,dt,\\
\theta^5&=dx^5-x^2x^3e^{-u^1}\,dt,
\end{align}
and has annihilator given by $\mathcal{V}=\{X,\partial_{u^1},\partial_{u^2}\}$, where
\begin{multline}
X=\partial_t+((x^2)^2+x^1e^{u^2})\partial_{x^1}+x^2e^{u^2}\partial_{x^2}+\ln\left(1+(u^2)^2\right)(x^2)^5e^{-5u^1}\partial_{x^3}
+\sin(x^5)\partial_{x^4}\\+x^2x^3e^{-u^1}\partial_{x^5}.
\end{multline}
Next we need to calculate the refined derived type of $\widehat{\mathcal{V}}=\mathcal{V}\oplus\Gamma$ in order to apply Theorem \ref{ESFL Quotient}. We can once again use procedure \textbf{contact} to determine the refined derived type of $\widehat{\mathcal{V}}$. Doing so, we find 
\begin{equation}
\mathfrak{d}_r(\mathcal{\widehat{V}})=[[5,2],[7,4,5],[8,8]].
\end{equation}
We then discover that the associated type numbers are those of a relative Goursat bundle by checking the conditions in Proposition \ref{refined derived type numbers}. Next we check that $\hinCharOne{V}$ is integrable. To do so, we use MAPLE to calculate the derived flag and the associated Cauchy bundles. We find that 
\begin{equation}
\hChar{V}{1}=\{\partial_{x^1},\partial_{x^2},\partial_{x^3},\partial_{u^1},\partial_{u^2}\},
\end{equation}
and since $\Gamma\oplus\{\partial_{u^1},\partial_{u^2}\}=\{\partial_{x^1},\partial_{x^2},\partial_{u^1},\partial_{u^2}\}\subset \widehat{\mcal{V}}$, we have
\begin{equation}
\inCharOne{\widehat{V}}=\Gamma\oplus\{\partial_{u^1},\partial_{u^2}\}.
\end{equation}
Hence $\hChar{V}{1}_0$ is integrable, and thus $\widehat{\mcal{V}}$ is a relative Goursat bundle. Therefore, by Theorem \ref{ESFL Quotient} we can conclude that $\omega/G$ is a Pfaffian system representing a control system on $M/G$ that is diffeomorphism equivalent to a Brunovsk\'y normal form. Furthermore, the diffeomorphism equivalence is an ESFT since $dt\in\coder{\Xi}{1}$ and $\{\partial_{u^1},\partial_{u^2}\}\subset \hinCharOne{V}$. Of course, one can see this directly by explicitly constructing $\omega/G$ as in equation (\ref{quotient h}), and then checking  the refined derived type of $\omega/G$.  However, Theorem \ref{ESFL Quotient} provides a much simpler determination when an example presents itself with several control admissible symmetry groups. Thus Theorem \ref{ESFL Quotient} allows us to avoid needless computation--and subsequent ESFL testing of--several different quotient systems. 
 
The quotient system for this example is
\begin{equation}\label{quotient f1}
\omega/G=\{dy^1-(v^2)^5\ln\left(1+(v^1)^2\right)\,dt,dy^2-\sin\left(y^3\right)\,dt, dy^3-y^1v^2\,dt\}.
\end{equation}
We can explicitly find the ESF linearization of (\ref{quotient f1}) via procedure \textbf{contact} in MAPLE. The ESFT is given by
\begin{equation}
(t,y^1,y^2,y^3,v^1,v^2)\mapsto (t,z^1_0,z^1_1,z^2_0,z^2_1,z^2_2),
\end{equation}
where
\begin{align}
\,&\,&z^2_2&=y^1v^2\cos\left(y^3\right),\\
z^1_1&=(v^2)^5\ln\left(1+(v^1)^2\right),&\,z^2_1&=\sin\left(y^3\right),\\
z^1_0&=y^1,\,&z^2_0&=y^2.
\end{align}

The BC system also admits an ESFL quotient. Let $\mcal{W}=\ann \omega^{BC}$ and $\Gamma_H$ be as in (\ref{BC sym}). Furthermore, denote $\widehat{\mcal{W}}=\mcal{W}\oplus \Gamma_H$. Using MAPLE to calculate the refined derived type of $\wh{\mcal{W}}$, we find 
\begin{equation}
\mathfrak{d}_r(\widehat{\mcal{W}})=[[7, 3], [10, 6, 8], [11, 11]],
\end{equation}
which are precisely the type numbers of a relative Goursat bundle. We also find that 
\begin{equation}
\hinCharOne{W}=\{\partial_{u^1},\partial_{u^2},\partial_{u^3}\}\oplus\Gamma_H,
\end{equation}
and we easily see that $\hinCharOne{W}$ is integrable and is annihilated by $dt$. Hence by Theorem \ref{ESFL Quotient}, we find that $\omega/H$ is ESFL. We will further confirm that $\omega/H$ is ESFL by constructing an explicit ESFT to a Brunovsk\'y normal form. First, the invariant functions of $\Gamma_H$ are 
\begin{multline}
\mathrm{Inv}\,\Gamma_H=\{t, y^1=x^1, y^2=x^2, y^3=\frac{1}{2}((t-1)^2+1)x^5+(t-1)x^6+x^7-x^3,\\
y^4=x^4, v^1=u^1,v^2= u^2, v^3=u^3\}, 
\end{multline}
and they form a local coordinate system for the quotient manifold $M/H$. In these coordinates, the quotient system is given by $\omega/H=\langle \theta^1,\ldots,\theta^4\rangle$, where
\begin{align}
\theta^1&=dy^1-v^1\,dt,\\
\theta^2&=dy^2-y^1\,dt,\\
\theta^3&=dy^3-\left(\left((t-1)y^2-\frac{1}{2}((t-1)^2+1)\right)y^4-(1+v^1)y^2+v^3\right)\,dt,\\
\theta^4&=dy^4-(y^1v^3+v^2)\,dt.
\end{align}
The annihilator of $\omega/H$ will be denoted $\mcal{W}/H=\{X_H,\partial_{u^1},\partial_{u^2},\partial_{u^3}\}$, where
\begin{multline}
X_H=\partial_t+v^1\partial_{y^1}+y^1\partial_{y^2}+\left(\left((t-1)y^2-\frac{1}{2}((t-1)^2+1)\right)y^4-(1+v^1)y^2+v^3\right)\partial_{y^3}\\
+(y^1v^3+v^2)\partial_{y^4}.
\end{multline}
Checking the refined derived type, we find 
\begin{equation}
\mathfrak{d}_r(\mcal{W}/H)=[[4, 0], [7, 3, 5], [8, 8]],
\end{equation}
which are the type numbers of the canonical contact system on $J^{\langle 2,1\rangle}$. Since $\rho_2=1$, we construct the filtration (\ref{filt fun}) for $\mcal{W}/H$ and find that
\begin{align}
\left(\Pi^1\right)^\perp&=\{dt,dy^2\},\\
\coder{\Xi}{1}&=\{dt,dy^1,dy^2\},\\
\coder{\Xi}{1}_0&=\{dt,dy^1,dy^2,dy^3,dy^4\}.
\end{align}
All these bundles are integrable, and thus we have confirmed that $\mcal{W}/H$ is ESFL. Next, we use procedure contact B to build an ESFT between $\beta^{\langle 2,1\rangle}$ and $\omega/H$. Indeed, the invariants of  $\coder{\Xi}{1}_0/\coder{\Xi}{1}=\{dy^3,dy^4\}$ are $y^3$ and $y^4$, while the needed invariant of $\Pi^1$ is $y^2$. Hence, the ESFT is given by
\begin{align}\label{BC quotient ESFL}
z^1_0&=y^2,&z^2_0&=y^3,&z^3_0&=y^4,\\
z^1_1&=y^1,&z^2_1&=\left((t-1)y^2-\frac{1}{2}((t-1)^2+1)\right)y^4-(1+v^1)y^2+v^3,&z^3_1&=(y^1v^3+v^2),\\
z^1_2&=v^1.&\,&\,&\,&\,
\end{align}
\subsection{Reconstruction of Integral Manifolds and the Contact Sub-connection}
In the previous sections of this chapter we have seen when and how one can perform symmetry reduction of a control system so that the resulting reduced system is again a control system. Now, we want to be able to construct solutions to the original control system using the reduced control system. This is the content of \cite{AndersonFelsEDSwithSymmetry},  \cite{AndersonFelsBacklund}, \cite{AndersonFelsGroupInvSol} for general Pfaffian systems, and \cite{VassiliouCascadeI} when applied to control systems.

We'll start with the following definition. 
\begin{defn}\cite{VassiliouCascadeI}
Let $G$ be a Lie group and let $M$ and $M_G$ be manifolds such that $\pi: M\to M_G$ is a right principal $G$-bundle, and let $VM$ be the vertical bundle $\ker \pi_*$. Let $\Pi^G\subset TM_G$ be a subbundle. A constant rank distribution $u\mapsto H_u\subset T_uM$ is called a \textbf{principal sub-connection relative to} $\Pi^G$ if the following are true:
\begin{enumerate}
\item $H_u\cap V_uM=\{0\}$,
\item $d\pi(H_u)=\Pi^G_{\pi(u)}$,
\item $(\mu_g)_*H_u=H_{u\cdot g}$, for all $g\in G$,
\item $u\mapsto H_u$ is smooth.
\end{enumerate}
\end{defn}
If $\Pi^G=T(M_G)$ then this is the usual definition of a connection on a principal $G$-bundle. We state the following proposition about principal sub-connections. 
\begin{prop}
Let $\pi:M\to M_G$ be a right principal $G$-bundle with principal sub-connection $H$ relative to a subbundle $\Pi^G$ of $TM_G$. If $c:[a,b]\to M_G$ is any integral curve of $\Pi^G$, then for any $u\in \pi^{-1}(c(a))$ there is a unique curve $\tilde{c}: [a,b]\to M$ such that $\tilde{c}(a)=u$ and $\tilde{c}(t)$ is an integral curve of $H$ and $(\pi\circ \tilde{c})(t)=c(t)$. 
\end{prop}

First we present a result which gives conditions for when an integral manifold of an EDS descends to an integral manifold of a quotient system. 
\begin{prop}\label{int q int}\cite{AndersonFelsEDSwithSymmetry}
Let $G$ be a symmetry group of the EDS $\mcal{I}$ whose action is regular and strongly transverse to $\mcal{I}$. If $s:N\to M$ is an integral manifold of $\mcal{I}$, then $\pi\circ s$ is an integral manifold of $\mcal{I}/G$, where $\pi:M\to M/G$ is the quotient map.
\end{prop}
A more interesting result concerns when one can reverse this process. Given an integral manifold of the quotient EDS can it be used to construct an integral manifold to the original EDS?
\begin{thm}\label{lifted solutions}\cite{AndersonFelsEDSwithSymmetry}
Let $\mcal{I}$ be an EDS with strongly transverse symmetry group $G$ acting freely and regularly on a manifold $M$ with quotient map $\pi: M\to M/G$, and let $s:N\to M$ be an integral manifold of $\mcal{I}$. Then for every point $s(x_0)\in s(N)$, there exists a $G$-invariant open neighborhood $U$ of $s(x_0)$ and a cross-section $\sigma:U/G\to U$ such that $s(x)=\mu(\epsilon(x), \sigma\circ s_G(x))$ for all $x\in s^{-1}(U)$, where
\begin{enumerate}
\item $s_G=\pi\circ s$ is an integral manifold of $\mcal{I}/G\big|_{U/G}$, and
\item $\epsilon:s^{-1}(U)\to G$, where the graph of $\epsilon$ satisfies a completely integrable Pfaffian system on $s^{-1}(U)\times G$. 
\end{enumerate}
\end{thm}
We would like to emphasize that this theorem may be used to construct \textit{any} integral manifold of $\omega$.

We will now work through a simple example. We will use Example \ref{Quotient Example}, but for ease of demonstration we now choose $f=e^{u^2},g^1=u^2, g^2=x^5, g^3=1$ and $a_1=a_2=0, a_3=1$. It then follows that the quotient system in Equation (\ref{quotient h}) takes the form
\begin{equation}
\omega/G=\langle dy^1-v^1\,dt,dy^2-y^3\,dt,dy^3-v^2\,dt\rangle,
\end{equation}
which is precisely in Brunovsk\'y normal form and has type $\langle 1,1\rangle$. Hence, integral curves are given by
\begin{equation}
s_G(t)=(t,F_1(t), F_2(t), \dot{F}_2(t), \dot{F}_1(t), \ddot{F}_2(t)),
\end{equation}
where $F_1$ and $F_2$ are arbitrary smooth functions of $t$. To compute a lifted integral manifold of $\omega$ using $s_G$, we need an appropriate cross-section of $M\to M/G$, and then we will need to find a curve $\epsilon(t)$ in $G$ in order to construct $s(t)$. We may pick $\sigma$ to be the cross section given by
\begin{equation}
\sigma: (t,y^1,y^2,y^3,v^1,v^2)\mapsto (t,0,1,y^1,y^2,y^3,-\ln(v^2), v^1).
\end{equation}
Using this cross-section to lift $s_G(t)$ to $M$, we find
\begin{equation}
(\sigma\circ s_G)(t)=(t,0,1,F_1(t),F_2(t),\dot{F}_2(t),-\ln(\ddot{F}_2(t)), \dot{F}_1(t)).
\end{equation}
Now we explicitly use the action of $G$ on $M$. Recall that we have already computed the action in Example \ref{quotient action}. Applying this action to $\sigma\circ s_G(t)$ with $\gamma=(\epsilon_1,\epsilon_2)$, we find
\begin{equation}\label{pre lie equation}
\mu(\epsilon,(t,\bx,\bu))=(t,\epsilon_1e^{\epsilon_2}, e^{\epsilon_2},F_1(t),F_2(t),\dot{F}_2(t),\epsilon_2-\ln(\ddot{F}_2(t)), \dot{F}_1(t)). 
\end{equation}
Since we want (\ref{pre lie equation}) to be an integral manifold of $\omega$, we need to pull back $\omega$ by (\ref{pre lie equation}) to determine the necessary conditions on $\epsilon_1$ and $\epsilon_2$. Doing so gives
\begin{equation}\label{lie type}
\mu^*_t(\omega)=\langle d\epsilon_1-e^{\epsilon_2+\dot{F}_1(t)}\,dt,d\epsilon_2-e^{\dot{F}_1(t)}\,dt\rangle,
\end{equation} 
where $\mu_t$ denotes (\ref{pre lie equation}). Therefore, in order for $s(t)=\mu_t$ to be an integral manifold for $\omega$, we need the curve $\epsilon: s^{-1}(U)\to G$ to solve the Frobenius system (\ref{lie type}). Putting everything together, we can explicitly write down a formula for an integral manifold of $\omega$ given an integral manifold to $\omega/G$. Indeed, let $I(t)=\int_{t_0}^te^{\dot{F}_1(\tau)}\,d\tau$ and $J(t)=\int_{t_0}^te^{I(\tau)+\dot{F}_1(\tau)}\,d\tau$; then
\begin{equation}
s(t)=\left(t,J(t)e^{I(t)},e^{I(t)},F_1(t),F_2(t),\dot{F}_2(t),e^{I(t)}-\ln(\ddot{F}_2(t)), \dot{F}_1(t) \right)
\end{equation}
is an integral curve of $\omega$. The use of several quadratures here is undesirable from a control theory perspective, since one must compute these integrals for each individual trajectory from the quotient manifold.

We now present one of the main theorems of \cite{VassiliouCascadeI}, which is essentially an application of Theorem \ref{lifted solutions} in the case that a control system with control admissible symmetry group admits an ESFL quotient system. 
\begin{figure}[h]\label{reconstruction diagram}
\centering
\begin{tikzcd}
{(M,\omega)} \arrow[dd, "\pi"] \arrow["G"', loop, distance=2em, in=215, out=145] &  & {(J^\kappa\times G,\gamma^G)} \arrow[dd, "\pi_0"] \arrow[ll, "\tilde{\varphi}"'] &                                                      \\
                                                                               &  &                                                                                &                                                      \\
{(M/G,\omega/G)}                                                               &  & {(J^\kappa, \beta^\kappa)} \arrow[ll, "\varphi"']                              & \mathbb{R} \arrow[l, "c"'] \arrow[luu, "\tilde{c}"']
\end{tikzcd}
\caption{Decomposition of integral curves to a non-ESFL control system with control symmetry $G$ and ESFL quotient.}
\end{figure}
\begin{thm}\label{control reconstruction}\cite{VassiliouCascadeI}
Let $(M,\omega)$ be a non-ESFL control system invariant under the Lie group $G$ acting on $M$ freely, regularly, and strongly transversely, via $\mu: G\times M\to M$, and with Lie algebra of infinitesimal symmetries $\Gamma$. Denote $\mcal{V}=\ann \omega$; let $\widehat{\mcal{V}}=\mcal{V}\oplus\Gamma$ be an ESF relative Goursat bundle of type $\kappa$, and let $\varphi$ and $\tilde{\varphi}$ be the principal bundle maps of Figure 3.1, which satisfy $(\varphi^{-1})_*(\mcal{V}/G)=\mcal{C}^\kappa_m$. Then there is a principal sub-connection $\mcal{H}_G$ for $\pi_0:J^\kappa\times G\to J^\kappa$ whose lift $\tilde{c}$ of a contact curve $c:\mathbb{R}\to J^\kappa$ is such that $t\mapsto (\mu\circ \Gamma_\sigma\circ \tilde{c})(t)$ is an integral manifold of $(M,\omega)$. Here $\Gamma_\sigma: J^\kappa\times G\to M\times G$ via $(\bz,g)\mapsto (\sigma\circ \varphi(\bz), g)$ where $\sigma$ is a cross-section of $\pi:M\to M/G$. Furthermore, the principal sub-connection is of the form
\begin{equation}\label{contact connection}
\mcal{H}_G=\left\{\partial_t+\sum_{i=1}^m\sum_{l_i=0}^{\sigma_i-1}z^i_{l_i+1}\partial_{z^i_{l^i}}+\sum_{a=1}^rp^a(t,\bz)R_a, \partial_{z^i_{\sigma_i}}\right\},
\end{equation}
where $\{R_a\}$ spans the right invariant vector fields on $G$. 
\end{thm}
\begin{defn}\label{contact sub-connection def}
The principal sub-connection given in (\ref{contact connection}) is called a \textbf{contact sub-connection}. It is a principal sub-connection relative to the subbundle
\begin{equation}
\Pi^G=\left\{\partial_t+\sum_{i=1}^m\sum_{l_i=0}^{\sigma_i-1}z^i_{l_i+1}\partial_{z^i_{l^i}},\partial_{z^i_{\sigma_i}}\right\},
\end{equation}
and may equivalently be written as
\begin{equation}
\mcal{H}_G=\left\{\partial_t+\sum_{i=1}^m\sum_{l_i=0}^{\sigma_i-1}z^i_{l_i+1}\partial_{z^i_{l^i}}+\sum_{a,b=1}^rp^a(t,\bz)\rho_a^b(\epsilon)\partial_{\epsilon_b}, \partial_{z^i_{\sigma_i}}\right\},
\end{equation}
where $R_a=\rho_a^b(\epsilon)\partial_{\epsilon_b}$. As a Pfaffian system, it may also be written as
\begin{equation}
\gamma^G=\beta^\kappa\oplus \Theta^G,
\end{equation}
where $\Theta^G=\langle d\epsilon_a-p^b(t,\bz)\rho_b^a(\epsilon)\,dt\rangle$. 
\end{defn}
Theorem \ref{control reconstruction} says that the control system may be ``decomposed" to a Brunovsk\'y normal form plus an underdetermined \textit{equation of Lie type}, that is, an underdetermined ODE arising from the action of a Lie group on a manifold. Although equations of Lie type have very nice properties, as explained in \cite{BryantLieSymplectic}, most of these properties only apply in the case that the associated group has nontrivial isotropy subgroups, i.e. the group does not act freely on $M$. The group actions appearing in this thesis are free and hence much of the larger theory of equations of Lie type does not apply. The only exception to this is the case when the Lie group $G$ is solvable. In this situation we are guaranteed to be able to find solutions with a finite number of integrations, but only for fixed trajectories of the quotient system. 

Notice that the map $\tilde{\varphi}$ in Theorem \ref{control reconstruction} is given by $\mu\circ \Gamma_\sigma$, and hence $\tilde{\varphi}$ is an ESFT since $\varphi$ is an ESFT and $G$ is an admissible control symmetry group. 

Before starting an example, we wish to emphasize the difference between Theorem \ref{control reconstruction} and Theorem \ref{lifted solutions}, as well as the importance of the contact sub-connection. Theorem \ref{lifted solutions} should be thought of as a ``decomposition" of an integral manifold to an EDS $\mcal{I}$ via an integral manifold of the quotent system and a Frobenius system induced by both the group action and the integral manifold of the quotient system. In Theorem \ref{control reconstruction}, a normal form for the given control system $\omega$ is given via the contact sub-connection. One may still consider Theorem \ref{control reconstruction} as providing a ``decomposition" of integral manifolds, or in this case, trajectories; however, it is the special structure of the quotient system that allows for an explicit formulation of $\Theta^G$ in coordinates. 

We now construct the contact connection $\mcal{H}_G$ for the BC system from Example \ref{BC system}. We already have the action of the group $H$ from (\ref{BC Flow}), as well as the projection map $\pi:M\to M/H$ defined by $\mathrm{Inv}\,\Gamma_H$. So we pick a cross-section of the projection map to be 
\begin{equation}
\sigma:(t,y^1,y^2,y^3,y^4,v^1,v^2,v^3)\mapsto (t,y^1,y^2,-y^3,y^4,0,0,0,v^1,v^2,v^3).
\end{equation}
Now by (\ref{BC quotient ESFL}), we can conclude that integral curves of $\omega/H$ are given by
\begin{equation}
c(t)=(t,y^1=\dot{F}_1(t),y^2=F_1(t),y^3=F_2(t),y^4=F_3(t),v^1=\ddot{F}_1(t),v^2=A(t),v^3=B(t)),
\end{equation}
where
\begin{align}
A(t)&=\dot{F}_3(t)-\dot{F}_1(t)\left(\dot{F}_2(t)+(1+\ddot{F}_1(t))F_1(t)-\left((t-1)F_1(t)-\frac{1}{2}((t-1)^2+1)\right)F_3(t)\right),\\
B(t)&=\dot{F}_2(t)+(1+\ddot{F}_1(t))F_1(t)-\left((t-1)F_1(t)-\frac{1}{2}((t-1)^2+1)\right)F_3(t), 
\end{align}
and $F_1(t), F_2(t),$ and $F_3(t)$ are arbitrary smooth functions. Next we wish to construct the contact connection $\mcal{H}_H$. We will use the dual form of the contact connection, which in this case is given by 
\begin{equation}
\gamma^H:=\beta^\kappa\oplus \Theta^H,
\end{equation}
where $\Theta^H=\langle d\epsilon^a-p^a(t,\bz)\,dt \rangle$ for $1\leq a\leq 3$. The form of $\Theta^H$ follows from the fact that $H$ is abelian. Hence, we need only determine the functions $p^a(t,\bz)$. In order to find these functions, we use MAPLE to compute $\mu_t^*(\omega^{BC})$, where $\mu_t=(\mu\circ \Gamma_{H,\sigma}\circ \tilde{c})(t)$, and this leads to
\begin{align}\label{BC p(t) 1}
(p^1(t,\bz)\circ c)(t)&=F_3(t),\\
(p^2(t,\bz)\circ c)(t)&=F_3(t)(F_1(t)-t),\\
(p^3(t,\bz)\circ c)(t)&=B(t)-F_3(t)(F_1(t)-t+1).\label{BC p(t) 2}
\end{align}
Each of (\ref{BC p(t) 1})-(\ref{BC p(t) 2}) gives the form of one of the $p^a(t,\bz)$. Indeed, 
\begin{align}
p^1(t,\bz)&=z^3_0,\\
p^2(t,\bz)&=z^3_0(z^1_0-t),\\
p^3(t,\bz)&=B(t,\bz)-z^3_0(z^1_0+1-t),
\end{align}
where
\begin{equation}
B(t,\bz)=z^3_1+(z^1_2-1)z^1_0-\left((t-1)z^1_0-\frac{1}{2}((t-1)^2+1)\right)z^2_0.
\end{equation}
Hence the contact connection for the BC system with respect to symmetry group $H$ is
\begin{equation}
\mcal{H}_H=\{\partial_t+\sum_{i=1}^3\sum_{l_i=0}^{\sigma_i-1}z^i_{l_i+1}\partial_{z^i_{l_i}}+z^3_0\partial_{\epsilon_1}+z^3_0(z^1_0-t)\partial_{\epsilon_2}+\left(B(t,\bz)-z^3_0(z^1_0+1-t)\right)\partial_{\epsilon_3},\partial_{z^i_{\sigma_i}}\},
\end{equation}
or the dual form,
\begin{equation}
\gamma^H=\beta^{\langle 2,1\rangle}\oplus\langle d\epsilon_1-z^3_0\,dt,\,d\epsilon_2-z^3_0(z^1_0-t)\,dt,\,d\epsilon_3-(B(t,\bz)-z^3_0(z^1_0+1-t))\,dt\rangle.
\end{equation}

In the next chapter, we will introduce another linearization that may arise from the contact sub-connection. In particular, in Chapter 4 we will prove new results on the form of the contact sub-connection $\gamma^G$ that allow one to determine whether $\gamma^G$--and hence $\omega$--is or is not EDFL. 

\section{Cascade Feedback Linearization and the Truncated Euler Operator\label{chap:four}}
\subsection{Overview of Cascade Feedback Linearization} 
\tab In the previous chapter we learned that a control system $(M,\omega)$ with control admissible symmetries can be put into a normal form adapted to said symmetries. In particular, if there is a quotient system $(M/G,\omega_G)$ that is ESFL, then the original control system is ESFT equivalent to a ``linear" system plus an equation of Lie type. In this chapter we will explore the fourth item in the following definition of Cascade Feedback Linearization:
\begin{defn}\label{CFL def}
Let $(M,\omega)$ be a control system with control admissible symmetry group $G$. Then we say that $(M,\omega)$ is \textbf{cascade feedback linearizable} (CFL) if:
\begin{enumerate}
\item The right group action of $G$ on $M$ is such that the orbit space $M/G$ is again a manifold and the associated \textit{quotient system} $\omega_G$ is again a linear Pfaffian system with the same number of controls.
\item $(M/G,\omega_G)$ is equivalent to a Brunovsk\'y normal form on the partial prolongation of a jet space $(J^\kappa,\beta^\kappa)$ via an ESFT, where $\beta^\kappa$ is the Pfaffian system of Brunovsk\'y normal forms, i.e. the canonical contact system on $J^\kappa$.
\item The original control system $(M,\omega)$ is ESFT equivalent to a normal form $(J^\kappa\times G, \gamma^G)$, where $\gamma^G= \beta^\kappa\oplus\Theta^G$ with $\Theta^G$ a 1-form associated to the action of $G$ on $M$. This may be interpreted as the local trivialization of a principal $G$-bundle over $J^\kappa$ with \textit{contact sub-connection} 1-form $\gamma^G$.
\item The restrictions of $(J^\kappa\times G, \gamma^G)$ to a certain family of submanifolds known as \textit{partial contact curves} become ESFT equivalent to a Brunovsk\'y normal form. 
\end{enumerate}
\end{defn} 
The last item in the definition for a CFL system is possibly the most mysterious, and the main results in the remainder of this thesis concern necessary and sufficient conditions for a control system to have this property. It turns out that, at least in the case $\mathrm{dim}\,G=1$, the last step is related to \textit{truncated} versions of familiar operators from the calculus of variations. These operators and some of their properties are described in Section 4.3 below.
\subsection{Partial Contact Curve Reduction}
\tab The final requirement for cascade feedback linearization is ESFT equivalence to Brunovsk\'y normal forms when the system on the principal $G$-bundle is restricted to what may be called ``partial integral manifolds" of $\gamma^G$ on $J^\kappa(\mathbb{R},\mathbb{R}^m)\times G$. For $m\geq 2$, we can always rewrite a Brunovsk\'y normal form as $\beta^\kappa=\beta^\nu\oplus\beta^{\nu^\perp}$, where $\kappa=\nu+\nu^\perp$ and $m=m_\nu+m_{\nu^\perp}$, so that $\beta^\nu$ and $\beta^{\nu^\perp}$ are the canonical contact systems on $J^\nu(\mathbb{R},\mathbb{R}^{m_\nu})$ and $J^{\nu^\perp}(\mathbb{R},\mathbb{R}^{m_{\nu^\perp}})$, respectively.
\begin{defn}
We say that a submanifold $\Sigma^\nu_f\subset J^\kappa\times G$ is a \textbf{codimension $s$ partial contact curve} of $\beta^\kappa=\beta^\nu\oplus\beta^{\nu^\perp}$ if $\Sigma^\nu_f$ is an integral manifold of $\beta^\nu$ and $s$ is the sum of the entries in $\nu^\perp$. It is described by the image of a map $C_f^{\nu}=j^{\nu}f\times Id_{J^{\nu^\perp}}\times Id_G:\mathbb{R}\times J^{\nu^\perp}\times G\to J^\kappa\times G$ for a choice of sufficiently differentiable $f:\mathbb{R}\to\mathbb{R}^{m_\nu}$. In particular, we refer to a system $\gamma^G$ restricted to a family of such submanifolds  of the form $\{\Sigma_f^\nu\colon f\in C^\infty(\mathbb{R},\mathbb{R}^{m_\nu})\}$ as a \textbf{partial contact curve reduction} of $\gamma^G$ and denote it by $\bar{\gamma}^G$. 
\end{defn}
One may find it odd that we will be restricting our control system to submanifolds. However, we note that the definition for a partial contact curve leaves open an arbitrary choice for the function $f$. Indeed, restriction to a particular partial contact curve is equivalent to a choice of $m_\nu$ controls and the states determined by that choice. If the resulting system $\bar{\gamma}^G$ is ESFL, then integral curves of $\bar{\gamma}^G$ can be expressed in terms of an arbitrary (up to some mild genericity conditions) function $g:\mathbb{R}\to\mathbb{R}^{m_{\nu^\perp}}$, as well as in terms of the arbitrary choice of partial contact curve (again up to mild genericity conditions to be elaborated on later). Thus no real freedom of choice for the controls is lost by this process, since the choice of a partial contact curve is a choice of $m-s$ controls. 

We now present an example. We will again use the BC system from Example \ref{BC system}. At the end of Chapter 3, we found the contact sub-connection of the BC system associated to the admissible control symmetry group $H$ whose infinitesimal generators are given by (\ref{BC sym}). The contact sub-connection has the Brunovsk\'y normal form $\beta^{\langle 2,1\rangle}$ as a component. we will decompose this Brunovsk\'y normal form as $\beta^{\langle 2,1\rangle}=\beta^{\langle 0,1\rangle}\oplus \beta^{\langle 2 \rangle}$. Specifically, we will be choosing our reduction along the copy of $\beta^{\langle 2\rangle}$ given by $j^2(z^1_0)$ by choosing $z^1_0=f(t)$ for an arbitrary smooth function $f$. This defines a codimension 2 partial contact curve and leads to the following reduced contact sub-connection on $J^{\langle 0,1\rangle}\times H$:
\begin{equation}\label{BC reduced}
\bar{\gamma}^H=\beta^{\langle 2\rangle}\oplus\langle d\epsilon_1-z^3_0\,dt,\,d\epsilon_2-z^3_0(f(t)-t)\,dt,\,d\epsilon_3-(B(t,j^2f(t),\bz)-z^3_0(f(t)+1-t))\,dt \rangle,
\end{equation}
where 
\begin{equation}
B(t,j^2f(t),\bz)=z^3_1+(\ddot{f}(t)-1)f(t)-\left((t-1)f(t)-\frac{1}{2}((t-1)^2+1)\right)z^2_0.
\end{equation}
The refined derived type of the reduced contact sub-connection is $[[3, 0], [5, 2, 3], [6, 4, 4], [7,5,5], [8, 8]]$, which agrees with a Brunovsk\'y normal form with type $\bar{\kappa}=\langle 1,0,0,1\rangle$. Since $\bar{\rho}_3=1$,  we can construct the filtration (\ref{filt fun}):
\begin{equation}\label{BC reduced filt}
\begin{aligned}
(\bar{\Pi}^3)^\perp&=\{dt,dz^3_0-\left(f(t)^2-tf(t)+\frac{1}{2}t^2\right)d\epsilon_1+(f(t)-t)d\epsilon_2-d\epsilon_3\},\\
\coder{\Xi}{3}&=\coder{\Xi}{3}_2=\{dt,dz^2_0-\frac{1}{2}t(2f(t)-t)d\epsilon_1-d\epsilon_3,(f(t)-t)d\epsilon_1-d\epsilon_2\},\\
\coder{\Xi}{2}&=\coder{\Xi}{2}_1=\{dt,d\epsilon_1,d\epsilon_2,dz^2_0-d\epsilon_3\},\\
\coder{\Xi}{1}&=\{dt,d\epsilon_1,d\epsilon_2,dz^2_0-d\epsilon_3,dz^3_0\},\\
\coder{\Xi}{1}_0&=\{dt,d\epsilon_1,d\epsilon_2,d\epsilon_3,dz^2_0,dz^3_0\},
\end{aligned}
\end{equation}
and observe that all subbundles of the filtration are integrable and contain $dt$. Therefore, we can conclude that $\bar{\gamma}^G$ is ESFL to a Brunovsk\'y normal form of signature $\langle 1,0,0,1\rangle$ by Definition \ref{Goursat}, Theorem \ref{Generalized Goursat Normal Form}, and Theorem \ref{Goursat ESFL}. Thus we conclude that the BC system is CFL with respect to the symmetry group $H$. 

We now mention an important consequence of a control system satisfying Definition \ref{CFL def}. 
\begin{thm}\label{CFL is EI}\cite{VassiliouCascadeI}
If a control system $\omega$ on a smooth manifold $M$ is cascade feedback linearizable, then it is explicitly integrable. 
\end{thm}
As example demonstrating Theorem \ref{CFL is EI}, we once again use the BC system. First we construct the ESF linearization for (\ref{BC reduced}) via procedure \textbf{contact} B since $\bar{\rho}_k=1$. Using (\ref{BC reduced filt}), we find that $\coder{\Xi}{1}_0/\coder{\Xi}{1}=\{d\epsilon_3\}$. Thus the new contact coordinates on $J^{\langle 1,0,0,1\rangle}$ can be given by
\begin{equation}\label{BC reduced ESF}
\begin{aligned}
w^1_0&=z^2_0,&w^2_0&=z^2_0-\left(f(t)^2-tf(t)+\frac{1}{2}t^2\right)\epsilon_1+(f(t)-t)\epsilon_2-\epsilon_3,\\
w^1_1&=z^2_1,&w^2_1&=(f(t)-t)\epsilon_1-\epsilon_2-f(t)(1+f(t)\ddot{f}(t)),\\
\,&\,&w^2_2&=-\epsilon_1,\\
\,&\,&w^2_3&=-z^3_0,\\
\,&\,&w^2_4&=-z^3_1.
\end{aligned}
\end{equation} 
Let $j^1g_1(t)$ and $j^4g_2(t)$ be arbitrary smooth solutions to the canonical contact system on $J^{\langle 1,0,0,1 \rangle}$. Inverting (\ref{BC reduced ESF}) and solving for $(\bz,\epsilon)$ in terms of $j^1g_1(t)$ and $j^4g_2(t)$, we find the following solution to $\bar{\gamma}^H$:\begin{equation}
\begin{aligned}
z^2_0(t)&= g_1(t), & \epsilon_1(t) &=-\ddot{g}_2(t),\\
z^2_1(t)&=\dot{g}_1(t), & \epsilon_2(t) & = -(f(t)-t)\ddot{g}_2-\dot{g}_2-f(t)(1+f(t)\ddot{f}(t)),\\
z^3_0(t)&= -\dddot{g_2}(t), & \epsilon_3(t) &= g_1(t)+\left(f(t)^2-tf(t)+\frac{1}{2}t^2\right)\ddot{g}_2(t)\\
z^3_1(t) &= -\ddddot{g_2}(t), &\, &-(f(t)-t)((f(t)-t)\ddot{g}_2+\dot{g}_2\\
\,&\,&\,&+f(t)(1+f(t)\ddot{f}(t)))-g_2(t).
\end{aligned}
\end{equation}
The contact curve reduction depends on an arbitrary smooth function $f(t)$. Thus, if we append $j^2z^1_0(t)=j^2f(t)$ to our above solution, then we have described all integral curves to $\gamma^H$. Furthermore, if we pass through the ESFT $\tilde{\varphi}$ from Theorem \ref{control reconstruction}, then we can find an explicit solution to the BC system in terms of arbitrary functions and their derivatives alone. That is, no integration is required to describe the trajectories of the BC system. 
\subsection{Truncated Euler Operator}
\tab In the last step of the cascade feedback linearization process, we want to know \textit{why} a contact sub-connection $\gamma^G$ admits an ESFL partial contact curve reduction. Towards this goal, we will explore the structure of ESFL reductions more closely by directly analyzing the PDEs associated to the calculation of the refined derived type. Recall that a contact sub-connection on $J^\kappa(\mathbb{R},\mathbb{R}^m)\times G$ has the form
\begin{equation}
\mathcal{H}_G=\{X,\partial_{u^a}\},
\end{equation}
where $1\leq a\leq m$ and 
\begin{equation}
X=\partial_t+\sum_{i=1}^m\sum_{l_i=0}^{\sigma_i-1}z^i_{l_{i}+1}\partiald{\,}{z^i_{l_i}}+\sum_{a,b}\rho^a_b(\epsilon)p^b(z)\partial_{\epsilon^a}.
\end{equation}
Note that the terms involving only jet bundle coordinates look very similar to the notion of a total derivative on an infinite jet bundle \cite{VarBi}. In computing the derived flag of an ESFL reduction for this contact sub-connection, one finds that this \textit{truncated} total derivative operator is iterated in such a way that there is a \textit{truncated} Euler operator that naturally appears. Hence, to better understand how the properties of these operators impact the refined derived type of an ESFL reduction, we will spend this section building some results about these types of operators. These truncated operators have similar--but not exactly the same--properties as their full infinite jet bundle analogues. The difference between the truncated versions and non-truncated versions arise naturally in proofs of results in this chapter, especially in Theorem \ref{sufficiency}. 
\begin{defn}
The \textbf{truncated total derivative} and \textbf{truncated sub-fiber total derivative} operators on $J^\kappa(\mathbb{R},\mathbb{R}^m)$ are given by
\begin{align*}
\sD_t:=&\partial_t+\sum_{i=1}^m\sum_{l_i=0}^{\sigma_i-1}z^i_{l_{i}+1}\partiald{\,}{z^i_{l_i}},\\
\sD_{t,i}:=&\partial_t+\sum_{l_i=0}^{\sigma_i-1}z^i_{l_i+1}\partiald{\,}{z^i_{l_i}}
\end{align*}
respectively, where $\sigma_i$ is the order of the jet for each $z^i_0$. 
\end{defn}
\begin{prop}\label{Dt inv 1}
The first integrals of the truncated total derivative operator $\sD_{t}$ are generated by
\begin{equation}
I^i_k(t,z)=z^i_{\sigma_i-k}-\int_{t_k}^tI^i_{k-1}(s,z)\,ds,
\end{equation}
where $I^i_0=z^i_{\sigma_i}$, $I^i_{-1}=0$, $0\leq k\leq \sigma_i$, $1\leq i\leq m$ and the $t_k$ are arbitrary constants.
\end{prop}
\begin{proof}
The proof follows by induction on $k$ for each $i$. It is immediate that $\sD_t(I_0^i)=0$, thus establishing the base case. Now assume that the above identity holds for some $k>0$ and that $\sD_t(I^i_{k-1})=0$. Then 
\begin{equation}
\begin{aligned}
\sD_t(I^i_k(t,z))&=z^i_{\sigma_i-k+1}-I^i_{k-1}(t,z)-\int_{t_k}^t \sD_s(I^i_{k-1}(s,z))-\partiald{I^i_{k-1}(s,z)}{s}\,ds\\
\,&=z^i_{\sigma_i-k+1}-I^i_{k-1}(t,z)+\int_{t_k}^t\partiald{I^i_{k-1}(s,z)}{s}\,ds\\
\,&=z^i_{\sigma_i-k+1}-I^i_{k-1}(t_k,z).
\end{aligned}
\end{equation}
In the last line we have the invariant $I^i_{k-1}$ evaluated at $t=t_k$. The induction hypothesis means that the recursive formula is true for $k-1$ and hence $I^i_{k-1}(t_k,z)=z^i_{\sigma_i-k+1}$. Thus, $\sD_t(I^i_k(t,z))=0$ for all $k\leq \sigma_i$. It is easy to see that the all $dI^i_k$ are linearly independent of each other, and upon a quick dimension count we find that there are precisely $\sum_{i\leq m}(\sigma_i+1)$ functions $I^i_k$. The total dimension of $J^\kappa$ is $1+\sum_{i\leq m}(\sigma_i+1)$; hence, the $I^i_k$ are a complete set of invariants for $\sD_t$. 
\end{proof}
Note that we can explicitly write formulas for the invariant functions in Proposition \ref{Dt inv 1}. Indeed, for each $t_k=0$,  
\begin{equation}
\begin{aligned}
I^i_0&=z^i_{\sigma_i},\\
I^i_1&=z^i_{\sigma_i-1}-tz^i_{\sigma_i},\\
I^i_2&=z^i_{\sigma_i-2}-tz^i_{\sigma_i-1}+\frac{1}{2}t^2z^i_{\sigma_i},\\
\,&\vdots\\
I^i_k&=\sum_{a=0}^k\frac{(-1)^{k-a}}{(k-a)!}t^{k-a}z^i_{\sigma_i-a},\\
\,&\vdots\\
I^i_{\sigma_i}&=\sum_{a=0}^{\sigma_i}\frac{(-1)^{\sigma_i-a}}{(\sigma_i-a)!}t^{\sigma_i-a}z^i_{\sigma_i-a}.
\end{aligned}
\end{equation}
An important observation about the proposition above is that if $f:J^\kappa\to\mathbb{R}$ is a differentiable function such that $\sD_t(f)=0$, then $f$ cannot have dependence on arbitrary functions of $t$ alone. With this in mind, we will now characterize the time-independent invariant functions of $\sD_t$. The first two of these $t$-independent invariants are
\begin{align}
J^i_2&=-\frac{1}{2}(z^i_{\sigma_i-1})^2+z^i_{\sigma^i-2}z^i_{\sigma_i},\\
J^i_3&=\frac{1}{3}(z^i_{\sigma_i-1})^3-z^i_{\sigma_i-2}z^i_{\sigma_i-1}z^i_{\sigma_i}+z^i_{\sigma_i-3}(z^i_{\sigma_i})^2.
\end{align}
\begin{prop}\label{t indep tot invs}
The $t$-independent invariant functions of $\sD_t$ are generated by the functions $J^i_1:=z^i_{\sigma_i}$ together with the functions
\begin{equation}
J^i_k=(-1)^{k-1}\frac{k-1}{k!}(z^i_{\sigma_i-1})^k+\sum_{a=2}^k\frac{(-1)^{k-a}}{(k-a)!}z^i_{\sigma_i-a}(z^i_{\sigma_i})^{a-1}(z^i_{\sigma_i-1})^{k-a}
\end{equation}
for $2\leq k\leq \sigma_i$ and $i\leq m$. 
\end{prop}
\begin{proof}
To prove that each $J^i_k$ is an invariant of $\sD_t$, we will simply compute $\sD_t(J^i_k)$ for all $1\leq i\leq m$ and $2\leq k\leq \sigma_i$. Indeed, since 
\begin{align}
\sD_t(z^i_{l_i})&=z^i_{{l_i}+1},\\
\sD_t(z^i_{\sigma_i})&=0
\end{align}
for all $0\leq l_i\leq \sigma_i-1$ and $1\leq i\leq m$, then
\begin{equation}\label{telescope} 
\begin{aligned}
\sD_t(J^i_k)&=(-1)^{k-1}\frac{1}{(k-2)!}z^i_{\sigma_i}\left(z^i_{\sigma_i-1}\right)^{k-1}\\
\,&+\sum_{a=2}^k\frac{(-1)^{k-a}}{(k-a)!}\left(z^i_{\sigma_i-a+1}\left(z^i_{\sigma_i}\right)^{a-1}\left(z^i_{\sigma_i-1}\right)^{k-a}+(k-a)z^i_{\sigma_i-a}(z^i_{\sigma_i})^a(z^i_{\sigma_i-1})^{k-a-1}\right)\\
\,&=(-1)^{k-1}\frac{1}{(k-2)!}z^i_{\sigma_i}\left(z^i_{\sigma_i-1}\right)^{k-1}\\
\,&+\sum_{a=2}^{k-1}\left(\frac{(-1)^{k-a}}{(k-a)!}z^i_{\sigma_i-a+1}\left(z^i_{\sigma_i}\right)^{a-1}\left(z^i_{\sigma_i-1}\right)^{k-a}-\frac{(-1)^{k-a}(k-a-1)}{(k-a)!}z^i_{\sigma_i-a}(z^i_{\sigma_i})^a(z^i_{\sigma_i-1})^{k-a-1}\right).\\
\end{aligned}
\end{equation}
The sum in equation (\ref{telescope}) is telescopic and the last term vanishes. Hence,
\begin{equation}
\begin{aligned}
\sD_t(J^i_k)&=(-1)^{k-1}\frac{1}{(k-2)!}z^i_{\sigma_i}\left(z^i_{\sigma_i-1}\right)^{k-1}+(-1)^{k-2}\frac{1}{(k-2)!}z^i_{\sigma_i-1}z^i_{\sigma_i}\left(z^i_{\sigma_i-1}\right)^{k-2}\\
\,&=0.
\end{aligned}
\end{equation}
We finish with a dimension count. The $J^i_k$ are invariants of both $\sD_t$ and $\partial_t$. Thus, for each $i=1,2,\ldots,m$, there are precisely $\sigma_i$ independent invariants of $\sD_t$ and $\partial_t$ that may be chosen to generate all invariants of $\sD_t$ and $\partial_t$. It is clear that all the $dJ^i_k$ are linearly independent. Therefore, the $J^i_k$ will generate the $t$-independent invariants of $\sD_t$. 
\end{proof}

Next, we define the truncated Euler operator and prove some properties about its kernel. Its definition is in terms of the truncated total derivative operators. The truncated Euler operator appears naturally when computing terms arising in the derived flag of a contact sub-connection. In particular, an understanding of its kernel will lead to insight about how the symmetry group of a control system can act on the manifold and have the property that the control system is CFL with respect to that symmetry. 
\begin{defn}
For each $\tau_i\leq \sigma_i$, define the \textbf{truncated Euler operator} $\sE_{\tau_i}$ on $J^\kappa(\mathbb{R},\mathbb{R}^m)$ to be
\begin{equation}
\sE_{\tau_i}(f)=\partiald{f}{z^i_0}-\sD_t\left(\partiald{f}{z^i_1}\right)+\sD_t^2\left(\partiald{f}{z^i_2}\right)-\cdots+(-1)^{\tau_i}\sD_t^{\tau_i}\left(\partiald{f}{z^i_{\tau_i}}\right)
\end{equation}
for any sufficiently differentiable function $f$ on $J^\kappa$, where $\sD_t$ is the truncated total derivative operator on $J^\kappa$. We can also define the \textbf{truncated sub-fiber Euler operator} $\sE_{\tau_i,j}$ by using the truncated sub-fiber total derivative operator $\sD_{t,j}$ in place of $\sD_t$. 
\end{defn}
\begin{prop}\label{EO kernel}
The kernel of the truncated Euler operator of any order $\tau_i\leq \sigma_i$ on $J^\kappa(\mathbb{R},\mathbb{R}^m)$ contains the set
\begin{equation}
K_{\sE_{\tau_i}}:=\left\{f(z)=\sD_tg(z)\,|\,\,g\in C^1(J^\kappa(\mathbb{R},\mathbb{R}^m))\,\text{and}\,\partiald{g}{z^i_{\tau_i}}\in\ker\,\sD_t^{\tau_i+1}\right\},
\end{equation}
where $\sD_t$ is the truncated total derivative operator on $J^\kappa(\mathbb{R},\mathbb{R}^m)$. 
\end{prop}
\begin{proof}
The result above effectively follows from repeated applications of the identity 
\begin{equation}
\partiald{\,}{z^i_{l_i+1}}\circ \sD_t=\partiald{\,}{z^i_{l_i}}+\sD_t\circ \partiald{\,}{z^i_{l_i+1}}.
\end{equation}
In the case that $l_i=0$, the identity is $\partiald{\,}{z^i_0}\circ \sD_t=\sD_t\circ \partiald{\,}{z^i_0}$. Applying these identities, we have
\begin{equation}
\begin{aligned}
\sE_{\tau_i}(f)&=\sD_t\left(\partiald{g}{z^i_0}\right)-\sD_t\left(\partiald{g}{z^i_0}+\sD_t\left(\partiald{g}{z^i_1}\right)\right)+\sD_t^2\left(\partiald{g}{z^i_1}+\sD_t\left(\partiald{g}{z^i_2}\right)\right)-\cdots\\
\,&+(-1)^{\tau_i-1}\sD_t^{\tau_i-1}\left(\partiald{g}{z^i_{\tau_i-2}}+\sD_t\left(\partiald{g}{z^i_{\tau_i-1}}\right)\right)+(-1)^{\tau_i}\sD_t^{\tau_i}\left(\partiald{g}{z^i_{\tau_i-1}}+\sD_t\left(\partiald{g}{z^i_{\tau_i}}\right)\right),
\end{aligned}
\end{equation}
which reduces to 
\begin{equation}
\sE_{\tau_i}(f)=(-1)^{\tau_i}\sD_t^{\tau_i+1}\left(\partiald{g}{z^i_{\tau_i}}\right).
\end{equation}
But since $\partiald{g}{z^i_k}$ is in the kernel of $\sD_t^{\tau_i+1}$, this implies that $\sE_{\tau_i}(\sD_tg)=0$.
\end{proof}
\begin{cor}\label{kernel cor}
If $f=\sD_tg$, where $\sD_t$ is the truncated total derivative operator on the jet space $J^\kappa(\mathbb{R},\mathbb{R}^m)$, $g$ is a function on $J^\kappa(\mathbb{R},\mathbb{R}^m)$, and $\sE_{\tau_i}$ is the truncated Euler operator of order $\tau_i\leq\sigma_i$, then 
\begin{equation}
\sE_{\tau_i}(f)=(-1)^{\tau_i}\sD_t^{\tau_i+1}\left(\partiald{g}{z^i_{\tau_i}}\right).
\end{equation}
\end{cor}

This is in contrast to the classical theory of calculus of variations, which has a modern geometric formulation on $J^\infty$ outlined wonderfully in \cite{VarBi}. In that work, the kernel of the Euler operator is the space of all functions on $J^\infty$ that are equal to the total derivative of some other function on $J^\infty$. Theorem \ref{EO kernel} and Corollary \ref{kernel cor} above highlight the difference between the kernel of the full Euler operator and that of our truncated version. This discrepancy is necessary for our work, and for finding examples of CFL systems. Theorem \ref{sufficiency} below makes this fact clear. One more important remark about the truncated Euler operator is required here. If $f\in C^\infty(J^{2n}(\mathbb{R},\mathbb{R}^m), \mathbb{R})$ for some $n>0$, then
\begin{equation}
\sE_{2n}(f)=E(f)
\end{equation}
if and only if either $f$ is constant or $\partiald{f}{z^i_k}=0$ for all $k\geq  n$ and $1\leq i\leq m$.

\subsection{PDEs for the Refined Derived Type} In the first chapter, the idea of the ``refined derived type" of a distribution/EDS was discussed. Recall that a given distribution admits a local normal form via diffeomorphism to a standard Goursat bundle if and only if both the refined derived type of the distribution in question is the same as one for a Goursat bundle and the appropriate filtration, given by either \ref{filt fun} and \ref{filt res}, is integrable. Furthermore, to ensure that the diffeomorphism is an ESF transformation, the control directions must be contained in the Cauchy characteristic bundle of the first derived distribution, and the 1-form $dt$ must either annihilate the Cauchy characteristic bundle of the penultimate derived distribution or annihilate the resolvent bundle, whichever applies by procedure \textbf{contact}.
In this section, we will look at the equivalent PDE conditions for the refined derived type of distributions on $J^{\sigma_1}(\mathbb{R},\mathbb{R})\times G$ that are ESF equivalent to the Goursat bundle associated with some $J^{\sigma_1+1}(\mathbb{R},\mathbb{R})$ when $\mathrm{dim}\,G=1$. For this case in particular, we mention that the integrable filtration condition will be satisfied automatically due to the necessary rank conditions on the derived flags of our distributions. 
\begin{remark}
We emphasize here that the above-mentioned $J^{\sigma_1+1}(\mathbb{R},\mathbb{R})$ with contact distribution is \textbf{not} a prolongation of the contact distribution on $J^{\sigma_1}(\mathbb{R},\mathbb{R})$. It is better to think of the relationship as ``anti-prolongation," in that, instead of a derivative being added to the represented control system in Brunovsk\'y normal form, a new state is being added via some kind of anti-differentiation. 
\end{remark}

Below is a proposition that is equivalent to a special case of Theorem 13 in \cite{VassiliouSICON}. The purpose of this proposition is to recognize explicit PDE conditions that will be necessary and sufficient for the reduction of a contact sub-connection to be ESFL. As will be seen in the next section, the specific PDE conditions to be satisfied will be conditions on \textit{truncated} Euler operators of the function associated to the right hand side of the equation of Lie type.
\begin{prop}\label{PDE form}
Let
\begin{equation}
\mathcal{H}_G=\{X,\partial_{z^1_{\sigma_1}},\ldots,\partial_{z^m_{\sigma_m}}\}
\end{equation}
be a contact sub-connection on $J^{\sigma_1}(\mathbb{R},\mathbb{R}^m)\times G$, with $\text{dim}\,G=1$ and with
\begin{equation}
X=\sD_t+p(z)\,\partial_\epsilon.
\end{equation}
If 
\begin{equation}
\bar{\mathcal{H}}_G=\{\bar{X},\partial_{z^1_{\sigma_1}}\}
\end{equation}
is the partial contact curve reduction along the partial contact curves that annihilate the Brunovsk\'y forms $\theta^i_{l_i}$ for $2\leq i\leq m$, then $\bar{\mathcal{H}}_G$ is ESFL if and only if the following hold:
\begin{itemize}
\item For $1\leq k\leq \sigma_1$, the $k$th derived flag of the reduced contact sub-connection is
\begin{align}
\bar{\mathcal{H}}_G^{(k)}=\{\bar{X},\bar{Y}_0,\bar{Y}_1,\cdots,\bar{Y}_k\},
\end{align}
where
\begin{equation}
\begin{aligned}
\bar{Y}_k&=[\bar{X},\bar{Y}_{k-1}]\\
\,&=Q_{k}\partial_\epsilon+(-1)^k\partial_{z_{\sigma_1-k}^1}
\end{aligned}
\end{equation}
with $\bar{Y}_0=\partial_{z^1_{\sigma_1}}$, and the coefficients $Q_k$ defined recursively by
\begin{align}
Q_{k}&=\bar{\sD}_t\left(Q_{k-1}\right)+(-1)^{k}\partiald{\bar{p}}{z_{\sigma_1-k+1}^1},
\end{align}
initialized with $Q_0=0$. For $k=\sigma_1+1$, we have 
\begin{equation}\label{Qn+1 not 0}
\begin{aligned}
\bar{Y}_{\sigma_1+1}&=[\bar{X},\bar{Y}_{\sigma_1}]\\
\,&=Q_{\sigma_1+1}\partial_{\epsilon},
\end{aligned}
\end{equation}
with
\begin{equation}
Q_{\sigma_1+1}=\bar{\sD}_t\left(Q_{\sigma_1}\right)+(-1)^{\sigma_1+1}\partiald{\bar{p}}{z^1_0}.
\end{equation}
Furthermore, $\bar{\mathcal{H}}^{(\sigma_1+1)}_G=TU$ for some open $U\subset J^{\sigma_1}(\mathbb{R},\mathbb{R})\times G$. That is, $\bar{\mathcal{H}}_G$ is bracket generating. 
\item For each $0\leq k \leq \sigma_1-1$, the Cauchy characteristics of $\bar{\mathcal{H}}^{(k+1)}_G$ are spanned by $\{\bar{Y}_l\}_{l=0}^{k}$, and hence the system of PDE
\begin{equation}\label{CC PDE}
(-1)^i\partiald{Q_k}{z_{\sigma_1-i}^1}+(-1)^{k+1}\partiald{Q_i}{z_{\sigma_1-k}^1}=0
\end{equation}
must be satisfied for all $0\leq i\leq k\leq \sigma_1$.
\end{itemize}
\end{prop}
\begin{proof}
We start by noting that for the first bullet point, the requirement that the rank of the derived flag of $\bar{\mathcal{H}}_G$ increases by one at each step is a necessary condition for ESF linearizability. Indeed, this condition would be sufficient to show that $\bar{\mathcal{H}}_G$ is local diffeomorphism equivalent to a Goursat bundle. However, for the stricter class of ESF equivalence, the second bullet point adds an additional condition to guarantee sufficiency and necessity. The second bullet point is equivalent to the condition from Theorem \ref{Goursat ESFL} that $\partial_{z^1_{\sigma_1}}$ is a Cauchy characteristic for $\bar{\mathcal{H}}_G^{(1)}$ and $dt\in\ann\left(\text{Char}\left(\bar{\mathcal{H}}_G^{(\sigma_1)}\right)\right)$.  

To prove the form above for the $\bar{Y}_k$, we use induction. First, we have
\begin{equation}
\begin{aligned}
\bar{Y}_1&=[\bar{X},\bar{Y}_0]\\
\,&=\partiald{\bar{p}}{z^1_{\sigma_1}}\,\partial_\epsilon-\partial_{z^1_{\sigma_1-1}},
\end{aligned}
\end{equation}
thus establishing the base case. Now notice that 
\begin{equation}\label{bracket fact}
[\bar{\sD}_t,\partial_{z^1_l}]=-\partial_{z^1_{l-1}}
\end{equation} 
for $1\leq l \leq \sigma_1$. Hence equation (\ref{bracket fact}) tells us that $[\bar{X},\bar{Y}_l]$ will be linearly independent from $\bar{X}$ and all the previous $\bar{Y}_i$ with $i\leq l$. For the inductive step, we assume that for $1\leq l \leq k-1$, 
\begin{equation}
\bar{Y}_l=Q_{l}\partial_\epsilon+(-1)^l\partial_{z_{\sigma_1-l}^1}.
\end{equation}
Computing $[\bar{X},\bar{Y}_{k-1}]$, we find 
\begin{equation}
\begin{aligned}
[\bar{X},\bar{Y}_{k-1}]&=[\bar{\sD}_t+\bar{p}\,\partial_\epsilon,\,Q_{k-1}\partial_\epsilon+(-1)^{k-1}\partial_{z_{\sigma_1-k+1}^1}]\\
\,&=\bar{\sD}_t(Q_{k-1})\partial_\epsilon-(-1)^{k-1}\partial_{z^1_{\sigma_1-k}}+[\bar{p}\partial_\epsilon,\,Q_{k-1}\partial_\epsilon+(-1)^{k-1}\partial_{z_{\sigma_1-k+1}^1}]\\
\,&=\left(\bar{\sD}_t\left(Q_{k-1}\right)+(-1)^{k}\partiald{p}{z_{\sigma_1-k+1}^1}\right)\partial_\epsilon+(-1)^k\partial_{z^1_{\sigma_1-k}}\\
\,&=Q_{k}\partial_\epsilon+(-1)^k\partial_{z^1_{\sigma_1-k}}.
\end{aligned}
\end{equation}
Next we address the second bullet. We need to show that all $\bar{Y}_l$ for $1\leq l\leq k$ are Cauchy characteristics for each $\bar{\mathcal{H}}_G^{(k)}$. In particular, the requirement that $dt\in\ann\left(\text{Char}\left(\bar{\mathcal{H}}_G^{(\sigma_1)}\right)\right)$ is equivalent to the condition that the $\bar{X}$-projection of any element of $\Char{\bar{H}}{\sigma_1}$ is zero (except in the last step of the derived flag, since it generates the tangent bundle). To see this, notice that in order for $\bar{\mcal{H}}_G$ to be bracket generating, we must have $\text{Char}\left(\bar{\mathcal{H}}^{(l)}_G\right)\subset\text{Char}\left(\bar{\mathcal{H}}^{(l+1)}_G\right)$ for all $l\leq n-1$, and hence $\ann\left(\text{Char}\left(\bar{\mathcal{H}}^{(l+1)}_G\right)\right)\subset\ann\left(\text{Char}\left(\bar{\mathcal{H}}^{(l)}_G\right)\right)$. Therefore, if $\bar{X}$ is not a Cauchy characteristic for any of the derived distributions save the last, then $dt\in\ann\left(\text{Char}\left(\bar{\mathcal{H}}_G^{(\sigma_1)}\right)\right)$ and vice versa. 

Next, we observe that 
\begin{equation}
[\bar{Y}_i,\bar{Y}_k]=\left((-1)^i\partiald{Q_k}{z_{\sigma_1-i}^1}+(-1)^{k+1}\partiald{Q_i}{z_{\sigma_1-k}^1}\right)\partial_\epsilon,
\end{equation}
so in order to not violate the condition that the dimension at each step in the derived flag increases by only one we must have
\begin{equation}\label{4.49}
(-1)^i\partiald{Q_k}{z_{\sigma_1-i}^1}+(-1)^{k+1}\partiald{Q_i}{z_{\sigma_1-k}^1}=0.
\end{equation}
Since (\ref{4.49}) holds by assumption, it is evident that $[\bar{Y}_i,\bar{Y}_k]=0$ for all $i\leq k$ in $\bar{\mathcal{H}}_G^{(k+1)}$ for all $0<k\leq \sigma_1$. Since $[\bar{X},\bar{Y}_i]=\bar{Y}_{i+1}\in\bar{\mathcal{H}}_G^{(k+1)}$ for $0\leq i\leq k$, there exist a maximal number of linearly independent vector fields in $\bar{\mathcal{H}}_G^{(k+1)}$ that are Cauchy characteristics. Were it not a maximal number, then $\bar{\mathcal{H}}^{(k+1)}_G$ would be Frobenius, contradicting the non-integrability of $\bar{\mathcal{H}}^{(k+1)}_G$. Therefore, we conclude that for each $k\leq \sigma_1$, 
\begin{equation}
\text{Char}\left(\bar{\mathcal{H}}_G^{(k+1)}\right)=\{\bar{Y}_0,\bar{Y}_1,\ldots,\bar{Y}_{k}\}.
\end{equation}
\end{proof}
\subsection{Necessity for ESFL Reductions}
\tab In this section we give an important necessary condition for a contact sub-connection to be ESFL via partial contact curve reduction. It is a condition on the truncated Euler operator applied to a function arising from the group action. we will prove the following theorem:
\begin{thm}\label{necessity}
Let $(t,z,\epsilon)=(t,j^{\sigma_1}z_0^1,\ldots,j^{\sigma_m}z_0^m,\epsilon)$ be local coordinates for $J^\kappa\times G$, where $\mathrm{dim}\,G=1$. Furthermore, assume that $\Theta^G=d\epsilon-p(z)\,dt$, and let $\bar{\gamma}^G$ be the reduction of $\gamma^G$ by codimension 1 partial contact curves defined by $j^{\sigma_l}z^l_0=j^{\sigma_l}f_l(t)$, for all $l\neq i$ for some $1\leq i\leq m$ and $m-1$ arbitrary smooth functions $f_l(t)$. If $\bar{\gamma}^G$ is ESFL, then the truncated Euler operator applied to $\bar{p}(z)$ must be degenerate in the sense that $\bar{\sE}_{\sigma_i}(\bar{p}(z))$ has no dependence on $z^i_k$ for $k\geq 1$. 
\end{thm}
\begin{proof}
Without loss of generality, take $i=1$. For notational simplicity we will drop the superscript of `1' on all $z^1_l$ variables. Now recall the fundamental bundle in filtration (\ref{filt fun}) from Chapter 2. In this case, our fundamental bundle $\Pi^{\sigma_1}$ is defined recursively by
\begin{equation}
\Pi^{k+1}=\Pi^k+[\bar{X},\Pi^k],\text{ where } \Pi^0=\{\partial_{z^1_{\sigma_1}}\}.
\end{equation} 
In particular,
\begin{equation}
\begin{aligned}
\Pi^1&=\{\partial_{z_{\sigma_1}},\bar{Y}_1\},\\
\Pi^2&=\{\bar{Y}_0,\bar{Y}_1,\bar{Y}_2\},\\
\vdots&\,\\
\Pi^{\sigma_1}&=\{\bar{Y}_0,\ldots,\bar{Y}_{\sigma_1}\}.
\end{aligned}
\end{equation}
Each $\Pi^k$ for $k\leq \sigma_1$ must be Frobenius for an ESFL system. Consider the Pfaffian system defined by
\begin{equation}
\begin{aligned}
\mathcal{F}^{\sigma_1}&:=\langle \ann\left(\Pi^{\sigma_1}\right)\rangle\\
\,&=\langle dt,d\epsilon-\psi\rangle.
\end{aligned}
\end{equation}
Here $\psi$ is a \textit{truncated} version of the Poincar\'e-Cartan form from the calculus of variations \cite{CartanForm2} \cite{CartanForm3} and is defined by
\begin{equation}
\psi:=\bar{p}\,dt+\sum_{i=1}^{\sigma_1}(-1)^{i+1}Q_i\theta^{\sigma_1-i},
\end{equation}
where the $\theta^{\sigma_1-i}$ are the canonical contact forms on $J^{\sigma_1}(\mathbb{R},\mathbb{R})$.  We can actually say a bit more. Indeed, for each bundle $\Pi^k$ in the recursive definition of the fundamental bundle, we can similarly define 
\begin{equation}
\mathcal{F}^k:=\langle\ann(\Pi^k)\rangle=\langle dt,d\epsilon-\psi_{k},\theta^{0},\cdots,\theta^{\sigma_1-k-1} \rangle,
\end{equation}
where
\begin{equation}
\psi_{k}=\bar{p}\,dt+\sum_{i=1}^{k}(-1)^{i+1}Q_i\theta^{\sigma_1-i}.
\end{equation}
As a brief remark, this also means that $\bar{\gamma}^G=\langle \theta^a, d\epsilon-\psi\rangle$. Since $\mathcal{F}^{\sigma_1}$ must be Frobenius for an ESFL system, we know that $d(d\epsilon-\psi)=-d\psi=\alpha\wedge\,dt$ for some 1-form $\alpha$. We will determine what this 1-form is and then check the wonderful identity $d^2=0$ to make our conclusion. Indeed, 
\begin{equation}
\begin{aligned}
d\psi&=d\bar{p}\wedge dt+\sum_{i=1}^{\sigma_1}(-1)^{i}z_{\sigma_1-i+1}dQ_i\wedge dt+(-1)^iQ_idz_{\sigma_1-i+1}\wedge dt+(-1)^{i+1}dQ_i\wedge dz_{\sigma_1-i}\\
\,&=d\bar{p}\wedge dt+\sum_{i=1}^{\sigma_1}\sum_{j=0}^{\sigma_1}(-1)^iz_{\sigma_1-i+1}\partiald{Q_i}{z_{\sigma_1-j}}dz_{\sigma_1-j}\wedge dt+(-1)^iQ_idz_{\sigma_1-i+1}\wedge dt\\
\,&+(-1)^{i+1}\partiald{Q_i}{z_{\sigma_1-j}}dz_{\sigma_1-j}\wedge dz_{\sigma_1-i}+(-1)^{i}\partiald{Q_i}{t}dz_{\sigma_1-i}\wedge dt.
\end{aligned} 
\end{equation}
However, since $\mathcal{F}^{\sigma_1}$ is Frobenius, we must have
\begin{equation}\label{frob 1}
(-1)^{i+1}\partiald{Q_i}{z_{\sigma_1-j}}+(-1)^j\partiald{Q_j}{z_{\sigma_1-i}}=0
\end{equation}
and
\begin{equation}\label{frob 2}
\partiald{Q_i}{z_{\sigma_1}}=0.
\end{equation}
Now rearranging indices to collect like forms in $d\psi$ while applying (\ref{frob 2}) gives
\begin{equation}
\begin{aligned}
d\psi&=d\bar{p}\wedge dt-Q_1\,dz_n\wedge dt\\
\,&+\sum_{i=1}^{\sigma_1-1}\left((-1)^{i+1}Q_{i+1}+(-1)^i\partiald{Q_i}{t}+\sum_{j=1}^{\sigma_1}(-1)^{j}z_{\sigma_1-j+1}\partiald{Q_j}{z_{\sigma_1-i}}\right)dz_{\sigma_1-i}\wedge dt\\
\,&+\left((-1)^{\sigma_1}\partiald{Q_{\sigma_1}}{t}+\sum_{j=1}^{\sigma_1}(-1)^{j}z_{\sigma_1-j+1}\partiald{Q_j}{z_0}\right)dz_{0}\wedge dt.
\end{aligned}
\end{equation}
Then applying (\ref{frob 1}), we arrive at
\begin{equation}
\begin{aligned}
d\psi&=d\bar{p}\wedge dt-Q_1\,dz_n\wedge dt\\
\,&+\sum_{i=1}^{\sigma_1-1}(-1)^i\left(-Q_{i+1}+\partiald{Q_i}{t}+\sum_{j=1}^nz_{\sigma_1-j+1}\partiald{Q_i}{z_{\sigma_1-j}}\right)dz_{\sigma_1-i}\wedge dt\\
\,&+(-1)^{\sigma_1}\left(\partiald{Q_{\sigma_1}}{t}+\sum_{j=1}^nz_{\sigma_1-j+1}\partiald{Q_{\sigma_1}}{z_{\sigma_1-j}}\right)dz_{0}\wedge dt\\
\,&=d\bar{p}\wedge dt-Q_1\,dz_n\wedge dt\\
\,&+\sum_{i=1}^{\sigma_1-1}(-1)^i\left(-Q_{i+1}+\bar{\sD}_t(Q_i)\right)dz_{\sigma_1-i}\wedge dt+(-1)^{\sigma_1}\bar{\sD}_t(Q_{\sigma_1})dz_{0}\wedge dt.
\end{aligned}
\end{equation}
Now recalling the recursive formula $Q_{i+1}=\bar{\sD}_t\left(Q_i\right)+(-1)^{i+1}\bar{p}_{z_{\sigma_1-i}}$, we arrive at
\begin{equation}
\begin{aligned}
d\psi&=d\bar{p}\wedge dt+\partiald{\bar{p}}{z_{\sigma_1}}\,dz_n\wedge dt\\
\,&+\sum_{i=1}^{\sigma_1-1}\partiald{\bar{p}}{z_{\sigma_1-i}}\,dz_{\sigma_1-i}\wedge dt+\left(\partiald{\bar{p}}{z_0}+(-1)^{\sigma_1}\bar{\sE}_{\sigma_1}(\bar{p})\right)dz_0\wedge dt,
\end{aligned}
\end{equation}
and thus we have 
\begin{equation}
d\psi=2\,d\bar{p}\wedge dt +(-1)^{\sigma_1}\bar{\sE}_{\sigma_1}(\bar{p})\,dz_0\wedge dt.
\end{equation}
So, upon computing $d^2\psi=0$, we quickly conclude that $\bar{\sE}_{\sigma_1}(\bar{p})$ has no dependence on $z_i$ for $1\leq i\leq n$ and must therefore be degenerate. 
\end{proof}

Theorem \ref{necessity} provides, for the first time, a coordinate specific obstruction to ESFL partial contact curve reducibility. In particular, if an ESFL partial contact curve reduction exists, then computation of $\bar{\sE}_{\sigma_1}(\bar{p}(z))$ can inform one about an appropriate choice of codimension 1 partial contact curve reduction to achieve the ESF linearization. There is also the added bonus that computing $\bar{\sE}_{\sigma_1}(\bar{p}(z))$ is a straightforward, albeit potentially tedious, calculation. As an application of Theorem \ref{necessity}, recall Example \ref{reduction example} from Chapter 1. It already has the form of a contact sub-connection on $J^{\langle 0,2\rangle}\times G$, where $G\cong \mathbb{R}$. The contact sub-connection may be written as
\begin{equation} 
\gamma^G=\langle dz^i_0-z^i_1\,dt, dz^i_1-z^i_2\,dt,d\epsilon-e^{z^1_1z^2_0}\,dt\rangle,\, i=1,2. 
\end{equation}
Reduction along the $z^1_0$ jet coordinates does not produce an ESFL system. Let $z^2_l=f^{(l)}(t)$ for all $l=0,1,2$, so that $\bar{p}(z)=e^{z^1_1f(t)}$. Computing the truncated Euler operator of this function, we find
\begin{equation}
\bar{\sE}_2(\bar{p})=-e^{z^1_1 f(t)}(z^1_1f(t)\dot{f}(t)+z^1_2f(t)^2+\dot{f}(t)).
\end{equation}
Hence, by Theorem \ref{necessity}, $\bar{\gamma}^G$ is not ESFL since $d\bar{\sE}_2(\bar{p})\wedge dz^1_0\wedge dt\neq0$. However, $\bar{\gamma}^G$ does have the form of a Goursat bundle. Although it is a Goursat bundle, the last Cauchy bundle in filtration (\ref{filt fun}) has the form
\begin{equation}
\mathrm{Char}\bar{\mcal{H}}_G^{(2)}=\left\{\partial_t+\partial_{z^1_0}-\frac{\dot{f}(t)(z^1_1f(t)+1)}{f(t)^2}\partial_{z^1_1}+e^{z^1_1f(t)}\partial_{\epsilon},\partial_{z^1_2}\right\}, 
\end{equation}
where $\bar{\mcal{H}}_G=\ann \bar{\gamma}^G$. Notice that $dt\not\in\ann\mathrm{Char}\bar{\mcal{H}}_G^{(2)}$ and therefore $\bar{\gamma}^G$ cannot be ESFL by Theorem \ref{Goursat ESFL}. Thus, we have agreement with Theorem \ref{necessity}. 

However, if we reduce along the other contact coordinates $(z^2_0,z^2_1,z^2_2)$, then the reduced system \textit{is} ESFL. For this reduction $z^1_l=g^{(l)}(t)$ for all $l=0,1,2$, and hence $\bar{p}(z)=e^{\dot{g}(t)z^2_0}$. This time, not only is the reduced system a Goursat bundle, but its last Cauchy bundle is given by
\begin{equation}
\mathrm{Char}\bar{\mcal{H}}_G^{(2)}=\left\{\partial_{z^2_1},\partial_{z^2_2}\right\}. 
\end{equation}
Thus we see from Theorem \ref{Goursat ESFL} that the system is ESFL. Computing the truncated Euler operator of $\bar{p}(z)$, we find
\begin{equation}
\bar{\sE}_2(\bar{p})=\dot{g}(t)e^{\dot{g}(t)z^2_0},
\end{equation}
and so $d\bar{\sE}_2(\bar{p})\wedge dz^2_0\wedge dt=0$, which agrees with Theorem \ref{necessity}. 
\subsection{Sufficiency for ESFL Reductions}
The main theorem of this section gives the first known class of control systems that will always admit an ESFL reduction. The first lemma below establishes a relationship between the truncated Euler operators and the Cauchy characteristics of each $\bar{\mathcal{H}}_G^{(k)}$ for $0\leq k\leq \sigma_1$. In particular, it is important to notice that since $Q_{\sigma_1+1}$ must be nonzero from Proposition \ref{PDE form}, the lemma below implies that $\bar{\sE}_{\sigma_1}(\bar{p})$ must also be nonzero. 
\begin{lem}\label{EOQ}
Let each $Q_{k+1}$, $0\leq k\leq \sigma_1$ be as defined in Proposition \ref{PDE form}. Then 
\begin{equation}
(-1)^{\sigma_1-1}\bar{\sD}^{\sigma_1-k}_t(Q_{k+1})=\bar{\sE}_{\sigma_1}(\bar{p})-\bar{\sE}_{\sigma_1-k-1}(\bar{p}).
\end{equation}
\end{lem}
\begin{proof}
This is by direct computation. Since the $Q_{k+1}$ are defined recursively, we can easily compute that
\begin{equation}\label{Qk+1 expanded}
\begin{aligned}
Q_{k+1}&=\bar{\sD}_t(Q_k)+(-1)^{k+1}\partiald{\bar{p}}{z^1_{\sigma_1-k}}\\
\,&=\bar{\sD}^2_t(Q_{k-1})+(-1)^{k}\bar{\sD}_t\left(\partiald{\bar{p}}{z^1_{\sigma_1-k+1}}\right)+(-1)^{k+1}\partiald{\bar{p}}{z^1_{\sigma_1-k}}\\
\,&\vdots\\
\,&=-\bar{\sD}^{k}_t\left(\partiald{\bar{p}}{z^1_{\sigma_1}}\right)+\bar{\sD}^{k-1}_t\left(\partiald{\bar{p}}{z^1_{\sigma_1-1}}\right)+\cdots+(-1)^{k}\bar{\sD}_t\left(\partiald{\bar{p}}{z^1_{\sigma_1-k+1}}\right)+(-1)^{k+1}\partiald{\bar{p}}{z^1_{\sigma_1-k}}.
\end{aligned}
\end{equation}
Applying $(-1)^{\sigma_1-1}\bar{\sD}^{\sigma_1-k}_t$ to \ref{Qk+1 expanded} gives
\begin{equation}
\begin{aligned}
(-1)^{\sigma_1-1}\bar{\sD}^{\sigma_1-k}_t(Q_{k+1})&=(-1)^{\sigma_1}\bar{\sD}^{\sigma_1}_t\left(\partiald{\bar{p}}{z^1_{\sigma_1}}\right)+\cdots+(-1)^{\sigma_1-k+1}\bar{\sD}^{\sigma_1-k-1}_t\left(\partiald{\bar{p}}{z^1_{\sigma_1-k+1}}\right)\\
\,&\phantom{=}\,+(-1)^{\sigma_1-k}\bar{\sD}^{\sigma_1-k}_t\partiald{\bar{p}}{z^1_{\sigma_1-k}},
\end{aligned}
\end{equation}
which is precisely the difference between the truncated Euler operators $\bar{\sE}_{\sigma_1}$ and $\bar{\sE}_{\sigma_1-k-1}$ acting on $\bar{p}$.  
\end{proof}
From Lemma \ref{EOQ} and Theorem \ref{necessity} we know that $\bar{\sE}_{\sigma_1}(\bar{p})$ must be both nonzero and depend on only $z^1_0$ and $t$, respectively. Given that the truncated Euler operator has nontrivial kernel, it is possible that there is some amount of freedom in the form of $\bar{p}(z)$. That is, it may have terms that are annihilated by the Euler operator. Such terms would not lead to a violation of Theorem \ref{sufficiency}. Thus, for control systems in the form of a contact sub-connection with a 1-dimensional control admissible symmetry, we would like to establish explicit forms for the $p(z)$ that will guarantee that there is an ESFL codimension 1 reduction by partial contact curves. To this end, we would like to know how codimension 1 partial contact curve reduction impacts the truncated total derivative as it applies to functions on $J^\kappa$ that have no dependence on higher order terms along some sub-fiber. 
\begin{lem}\label{reduced}
Let $p=\sD^r_{t,1}(A)$ for $A=A(z_0^1,j^{\sigma_2}z^2_0,j^{\sigma_3}z^3_0,\cdots,j^{\sigma_m}z^m_0)$ and any integer $r\geq0$. Then
\begin{equation}\label{reduced D eq} 
\bar{p}=\overline{\sD_{t,1}^r A}=\left(\bar{\sD}_t-\partiald{\,}{t}\right)^r\bar{A},
\end{equation}
or, in expanded form,
\begin{equation}
\bar{p}=\sum_{l=0}^{r}{r \choose l}(-1)^l\bar{\sD}_t^{r-l}\circ\frac{\partial^l}{\partial t^l}\bar{A},
\end{equation}
where $\bar{\cdot}$ denotes reduction along the jet coordinates $j^{\sigma_l}z^l_0$ for $2\leq l\leq m$.
\end{lem}
\begin{proof}
Since $p=\sD^r_{t,1}(A)$ has no time dependence on $J^\kappa(\mathbb{R},\mathbb{R}^m)$, the reduction must have no time derivatives. To see this, notice that for any $t$-independent function $F$ on $J^\kappa$, 
\begin{equation}
\bar{F}(t,j^{\sigma_1}z^1_0)=F(j^{\sigma_1}z^1_0,j^{\sigma_2}f^2(t),j^{\sigma_3}f^3(t),\cdots,j^{\sigma_m}f^m(t)).
\end{equation}
This implies that
\begin{equation}
\overline{\partiald{F}{z^1_{l_1}}}=\partiald{\bar{F}}{z^1_{l_1}}
\end{equation}
for $0\leq l_1\leq \sigma_1$. Therefore, 
\begin{equation}\label{reduced degree 1}
\begin{aligned}
\overline{\sD_{t,i}F}&=\overline{\sum_{l_1=0}^{\sigma_1-1}z^1_{l_1+1}\partiald{F}{z^1_{l_1}}}\\
\,&=\sum_{l_1=0}^{\sigma_1-1}z^1_{l_1+1}\partiald{\bar{F}}{z^1_{l_1}}\\
\,&=\bar{\sD}_t\bar{F}-\partiald{\bar{F}}{t}.
\end{aligned}
\end{equation}
Now if $F=\sD_{t,1}G$ for another $t$-independent function $G$ on $J^\kappa$, then
\begin{equation}
\begin{aligned}
\overline{\sD_{t,i}F}&=\bar{\sD}_t\bar{F}-\partiald{\bar{F}}{t}\\
\,&=\bar{\sD}_t\left(\bar{\sD}_t\bar{G}-\partiald{\bar{G}}{t}\right)-\partiald{\,}{t}\left(\bar{\sD}_t\bar{G}-\partiald{\bar{G}}{t}\right)\\
\,&=\left(\bar{\sD}_t-\partiald{\,}{t}\right)^2\bar{F},
\end{aligned}
\end{equation}
where the second line is by application of equations (\ref{reduced degree 1}). Iterating this argument gives the general form in the lemma.  
\end{proof}
We are now ready to state and prove the second main result of this thesis: an explicit description of an entire class of control systems that admit an ESFL codimension 1 partial contact curve reduction.  
\begin{thm}\label{sufficiency}
Let
\begin{equation}
\mathcal{H}_G=\left\{\partial_t+\sum_{i=1}^{m}\sum_{l_{i}=0}^{\sigma_i-1}z^i_{l_i+1}\partial_{z^i_{l_i}}+p(z)\,\partial_\epsilon\right\}
\end{equation}
be a partial contact sub-connection on $J^\kappa(\mathbb{R},\mathbb{R}^m)\times G$ with $\text{dim}(G)=1$ such that
\begin{equation}\label{p(z) form}
p(z)=\sum_{l=0}^{N} \sD^l_{t,i}A_l
\end{equation}
for any $0\leq N\leq \sigma_i$ and some $i=1,\ldots,m$, where the $\{A_l\}_{l=0}^{\sigma_i}$ are functions on $J^\kappa(\mathbb{R},\mathbb{R}^m)$ with no dependence on $t$ or $j^l(z_0^i)$ for any $l\geq1$. Then the reduction of $\mathcal{H}_G$ by the codimension 1 partial contact curves that annihilate all Brunovsk\'y forms except $\beta^{\sigma_i}$ is an ESFL system. 
\end{thm}
\begin{proof} Without loss of generality, take $i=1$. For notational simplicity we will drop the superscript of `1' on all $z^1_l$ variables. Furthermore, we can take $p(z)=\sD^{\sigma_1}_{t,i}A$, where $A=A_{\sigma_1}$, and ignore the lower order terms in (\ref{p(z) form}). We can do so because all the operations applied to $p(z)$ in the proof are linear, and any power of $\sD_t$ that is smaller than $\sigma_1$ will satisfy all the important inequalities concerning $\sigma_1$ in the proof. 

We need to compute each $Q_{k}$ and check that the Cauchy characteristic conditions (\ref{CC PDE}) on the $Q_k$ are satisfied. Indeed, for $k\leq \sigma_1-1$, we use Lemma \ref{EOQ} to see that
\begin{equation}
\begin{aligned}
(-1)^{\sigma_1-1}\bar{\sD}^{\sigma_1-k}_t(Q_{k+1})&=\bar{\sE}_{\sigma_1}\left(\sum_{l=0}^{\sigma_1}{\sigma_1 \choose l}(-1)^l\bar{\sD}_t^{\sigma_1-l}\circ\frac{\partial^l}{\partial t^l}\bar{A}\right)\\
\,&-\bar{\sE}_{\sigma_1-k-1}\left(\sum_{l=0}^{\sigma_1}{\sigma_1 \choose l}(-1)^l\bar{\sD}_t^{\sigma_1-l}\circ\frac{\partial^l}{\partial t^l}\bar{A}\right).
\end{aligned}
\end{equation}
Now by Proposition \ref{EO kernel}, we have
\begin{equation}
\begin{aligned}
(-1)^{\sigma_1-1}\bar{\sD}^{\sigma_1-k}_t(Q_{k+1})&=\bar{\sE}_{\sigma_1}\left((-1)^{\sigma_1}\frac{\partial^{\sigma_1}}{\partial t^{\sigma_1}}\bar{A}+\bar{\sD}^{\sigma_1}_t\bar{A}\right)\\
\,&-\bar{\sE}_{\sigma_1-k-1}\left((-1)^{\sigma_1}\frac{\partial^{\sigma_1}}{\partial t^{\sigma_1}}\bar{A}+\sum_{l=0}^{k+1}{\sigma_1 \choose l}(-1)^l\bar{\sD}_t^{\sigma_1-l}\circ\frac{\partial^l}{\partial t^l}\bar{A}\right).
\end{aligned}
\end{equation}
However, since $\bar{A}$ has no dependence on $z_l$ for $l\geq 1$, 
\begin{equation}
\bar{\sE}_{\sigma_1}\left((-1)^{\sigma_1}\frac{\partial^{\sigma_1}}{\partial t^{\sigma_1}}\bar{A}\right)=\bar{\sE}_{\sigma_1-k-1}\left((-1)^{\sigma_1}\frac{\partial^{\sigma_1}}{\partial t^{\sigma_1}}\bar{A}\right)=\frac{\partial^{\sigma_1+1}}{\partial z_0\,\partial t^{\sigma_1}}\bar{A}.
\end{equation}
Hence, these terms cancel, leaving
\begin{equation}
\begin{aligned}
(-1)^{\sigma_1-1}\bar{\sD}^{\sigma_1-k}_t(Q_{k+1})&=\bar{\sE}_{\sigma_1}\left(\bar{\sD}^{\sigma_1}_t\bar{A}\right)-\bar{\sE}_{\sigma_1-k-1}\left(\sum_{l=0}^{k+1}{\sigma_1 \choose l}(-1)^l\bar{\sD}_t^{\sigma_1-l}\circ\frac{\partial^l}{\partial t^l}\bar{A}\right).
\end{aligned}
\end{equation}
Now by use of Corollary \ref{kernel cor},
\begin{equation}
\begin{aligned}
(-1)^{\sigma_1-1}\bar{\sD}^{\sigma_1-k}_t(Q_{k+1})&=(-1)^{\sigma_1}\bar{\sD}^{\sigma_1+1}_t\left(\partiald{\,}{z_{\sigma_1}}\bar{\sD}^{\sigma_1-1}_t\bar{A}\right)\\
\,&-\sum_{l=0}^{k+1}{\sigma_1 \choose l}(-1)^l\bar{\sE}_{\sigma_1-k-1}\left(\bar{\sD}_t^{\sigma_1-l}\circ\frac{\partial^l}{\partial t^l}\bar{A}\right)\\
\,&= (-1)^{\sigma_1-k}\sum_{l=0}^{k+1}{\sigma_1 \choose l}(-1)^l\bar{\sD}^{\sigma_1-k}_t\left(\partiald{\,}{z_{\sigma_1-k-1}}\circ\bar{\sD}_t^{\sigma_1-l-1}\circ\frac{\partial^l}{\partial t^l}\bar{A}\right).
\end{aligned}
\end{equation}
Formally inverting $\bar{\sD}^{\sigma_1-k}_t$ as a linear differential operator, we find that
\begin{equation}
\begin{aligned}
Q_{k+1}&=C_{\sigma_1-k}+(-1)^{k+1}\sum_{l=0}^{k+1}{\sigma_1 \choose l}(-1)^l\partiald{\,}{z_{\sigma_1-k-1}}\circ\bar{\sD}_t^{\sigma_1-l-1}\circ\frac{\partial^l}{\partial t^l}\bar{A}\\
\,&=C_{\sigma_1-k}+(-1)^{k+1}\sum_{l=0}^{k}{\sigma_1 \choose l}(-1)^l\partiald{\,}{z_{\sigma_1-k-1}}\circ\bar{\sD}_t^{\sigma_1-l-1}\circ\frac{\partial^l}{\partial t^l}\bar{A},
\end{aligned}
\end{equation}
where $C_{\sigma_1-k}\in \text{ker}\,\bar{\sD}^{\sigma_1-k}_t$ and the upper limit in the sum decreases by one because $\bar{\sD}^{\sigma_1-k-2}\circ\frac{\partial^l}{\partial t^l}\bar{A}$ has no dependence on $z_{\sigma_1-k-1}$. Thus we also have
\begin{equation}
Q_{k}=C_{\sigma_1-k+1}+(-1)^k\sum_{l=0}^{k-1}{\sigma_1 \choose l}(-1)^l\partiald{\,}{z_{\sigma_1-k}}\circ\bar{\sD}_t^{\sigma_1-l-1}\circ\frac{\partial^l}{\partial t^l}\bar{A}.
\end{equation}
We need to show that the $C_{\sigma_1-k}$ terms are zero for all defined values of $k$. we will check that the recursive definition for the $Q_{k+1}$'s is satisfied. Indeed, 
\begin{equation}
Q_{k+1}=\bar{\sD}_t(Q_k)+(-1)^{k+1}\partiald{\bar{p}}{z_{\sigma_1-k}}.
\end{equation}
Now, for our choice of $\bar{p}$ we have
\begin{equation}
\begin{aligned}
\partiald{\bar{p}}{z_{\sigma_1-k}}&=\sum_{l=0}^{\sigma_1}{\sigma_1\choose l}(-1)^l\frac{\partial}{\partial z_{\sigma_1-k}}\circ\bar{\sD}^{\sigma_1-l}_t\circ\frac{\partial^l}{\partial t^l}\bar{A}\\
\,&=\sum_{l=0}^k{\sigma_1\choose l}(-1)^l\frac{\partial}{\partial z_{\sigma_1-k}}\circ\bar{\sD}^{\sigma_1-l}_t\circ\frac{\partial^l}{\partial t^l}\bar{A},
\end{aligned}
\end{equation}
where the terms in the sum for $l>k$ are all zero since $\bar{\sD}^{\sigma_1-l}_t\circ\frac{\partial^l}{\partial t^l}\bar{A}$ has dependence on jet coordinates only up to order $\sigma_1-l$. Now using
\begin{equation}
\frac{\partial}{\partial z_{\sigma_1-k}}\circ\bar{\sD}^{\sigma_1-l}_t=\frac{\partial}{\partial z_{\sigma_1-k-1}}\circ\bar{\sD}^{\sigma_1-l-1}_t+\bar{\sD}_t\circ\frac{\partial}{\partial z_{\sigma_1-k}}\circ \bar{\sD}^{\sigma_1-l-1}_t,
\end{equation}
we see that
\begin{multline}
\partiald{\bar{p}}{z_{\sigma_1-k}}=\sum_{l=0}^k{\sigma_1\choose l}(-1)^l\frac{\partial}{\partial z_{\sigma_1-k-1}}\circ\bar{\sD}^{\sigma_1-l-1}_t\circ\frac{\partial^l}{\partial t^l}\bar{A}\\
+\bar{\sD}_t\sum_{l=0}^{k-1}{\sigma_1\choose l}(-1)^l\frac{\partial}{\partial z_{\sigma_1-k}}\circ\bar{\sD}^{\sigma_1-l-1}_t\circ\frac{\partial^l}{\partial t^l}\bar{A}.
\end{multline}
Hence the recursion relation simplifies to 
\begin{equation}
Q_{k+1}=\bar{\sD}_t(C_{\sigma_1-k+1})+(-1)^{k+1}\sum_{l=0}^k{\sigma_1\choose l}(-1)^l\frac{\partial}{\partial z_{\sigma_1-k-1}}\circ\bar{\sD}^{\sigma_1-l-1}_t\circ\frac{\partial^l}{\partial t^l}\bar{A}.
\end{equation}
This is precisely the form we have already derived, where $C_{\sigma_1-k}=\bar{\sD}_t(C_{\sigma_1-k+1})$. This implies that $C_{\sigma_1-k}=\bar{\sD}^k_t(C_{\sigma_1})$ for all $k$. Now, checking directly that
\begin{equation}
Q_1=-\partiald{\bar{p}}{z_{\sigma_1}},
\end{equation}
by the recursion relation definition of the $Q_k$ we can conclude that $C_{\sigma_1}=0$, and hence all $C_{\sigma_1-k}=0$. Now we check the Cauchy characteristic PDE (\ref{CC PDE}) and find that
\begin{equation}
\begin{aligned}
\,&(-1)^{k+i}\sum_{l=0}^{k-1}{\sigma_1 \choose l}(-1)^l\partialds{\,}{z_{\sigma_1-k}}{z_{\sigma_1-i}}\circ\bar{\sD}_t^{\sigma_1-l-1}\circ\frac{\partial^l}{\partial t^l}\bar{A}\\
\,&+(-1)^{k+i+1}\sum_{l=0}^{i-1}{\sigma_1 \choose l}(-1)^l\partialds{\,}{z_{\sigma_1-i}}{z_{\sigma_1-k}}\circ\bar{\sD}_t^{\sigma_1-l-1}\circ\frac{\partial^l}{\partial t^l}\bar{A}=0,
\end{aligned}
\end{equation}
which simplifies to
\begin{equation}\label{final form for proof}
\sum_{l=i}^{k-1}{\sigma_1 \choose l}(-1)^l\partialds{\,}{z_{\sigma_1-k}}{z_{\sigma_1-i}}\circ\bar{\sD}_t^{\sigma_1-l-1}\circ\frac{\partial^l}{\partial t^l}\bar{A}=0.
\end{equation}
However, since $\bar{A}$ has dependence only on $t$ and $z_0$, the function $\bar{\sD}_t^{\sigma_1-l-1}\circ\frac{\partial^l}{\partial t^l}\bar{A}$ has no dependence on $z_{\sigma_1-i}$ for $i\leq l\leq k-1$, and hence equation (\ref{final form for proof}) is true. 
\end{proof}
Theorem \ref{sufficiency} gives an entire class of examples of control systems that have the property that they admit ESFL reductions. So for the first time, we can now construct explicit examples of control systems that are CFL and hence DFL. One can in principle ask how such functions $p(z)$ arise from the reconstruction process of Anderson and Fels, and this will be a focus of later work. Furthermore, the author informally conjectures the following:
\begin{conj}
The \textit{only} examples of systems with ESFL reductions by codimension 1 partial contact curves must have $p(z)$ in the kernel of the truncated Euler operator and also satisfy truncated versions of identities relating higher order Euler operators to the total derivative operator. 
\end{conj}

We mention that in Theorem \ref{sufficiency}, the requirement that the power of $\sD_{t,i}$ be no larger than $\sigma_i$ is necessary. Assume that 
\begin{equation}
p(z)=\sD^{\sigma_1+1}_{t,i}(A).
\end{equation}
Then the reduction of the contact sub-connection fails to be ESFL. Indeed, applying the truncated Euler operator and using Corollary \ref{kernel cor}, we find that
\begin{equation}
\bar{\sE}_{\sigma_i}(\bar{p})=\bar{\sD}^{\sigma_i+1}_t\left(\partiald{\bar{A}}{z_0}\right).
\end{equation}
This violates Theorem \ref{necessity} of the previous section, which says that if $\bar{\mathcal{H}}_G$ is ESFL then $\bar{\sE}_{\sigma_1}(\bar{p})$ must be nonzero and have dependence only on $t$ and $z_0$.

Recall Example \ref{so(5)} from Chapter 1:
\begin{equation}
\begin{aligned}
\dot{x}_1=\frac{1}{2}(x_2+2x_3x_5),& \phantom{==}\dot{x}_2=2(x_3+x_1x_5),\\
\dot{x}_3=\frac{2(u_1-x_1u_2)}{1+x_1},&\phantom{==}\dot{x}_4=x_5,\\
\dot{x}_5=\frac{2(u_1+u_2)}{1+x_1}.&\,
\end{aligned}
\end{equation}

This system is cascade feedback linearizable via a 1-dimensional symmetry group. This control system first appears in \cite{Cascade2}, where it was shown to be CFL, but not SFL or even EDFL by partial prolongation. The control system has a 1-dimensional control admissible symmetry group whose infinitesimal action is generated by the vector field
\begin{equation}
X_\Gamma=x_1\partial_{x_1}+x_2\partial_{x_2}+x_3\partial_{x_3}+u_1\partial_{u_1}+\frac{x_1u_2-u_1}{1+x_1}\partial_{u_2}.
\end{equation}
The associated quotient system is SFL, and thus we can use Theorem \ref{control reconstruction} to transform the system into a contact sub-connection:
\begin{equation}
\gamma^G=\left\langle\theta^i_0,\theta^i_1,d\epsilon-\frac{2(z^2_1)^2-z^2_1z^1_1-z^1_0}{z^2_1z^1_0-2}\,dt\right\rangle,\text{ for }i=1,2.
\end{equation}
Notice that in $\gamma^G$ the $p(z)$ term is given by
\begin{equation}
\begin{aligned}
p(z)&=\frac{2(z^2_1)^2-z^2_1z^1_1-z^1_0}{z^2_1z^1_0-2}\\
\,&=\frac{2(z^2_1)^2-z^1_0}{z^2_1z^1_0-2}-\sD_{t,1}\ln(z^2_1z^1_0-2).
\end{aligned}
\end{equation}
Looking at the hypotheses of Theorem \ref{sufficiency} we see that $p(z)=A_0+\sD_{t,1}(A_1)$, where
\begin{equation}
A_0=\frac{2(z^2_1)^2-z^1_0}{z^2_1z^1_0-2}\\
\end{equation}
and
\begin{equation}
A_1=-\ln(z^2_1z^1_0-2).
\end{equation}
The $A_0$ and $A_1$ terms both have dependence on $z^1_0$, but not on any higher order terms, i.e. $z^1_1$ or $z^1_2$. Thus if we reduce by codimension 1 partial contact curves that annihilate $\theta_0^2$ and $\theta_1^2$, then by Theorem \ref{sufficiency} we should expect the resulting control system to be ESFL. This reduction can be described by setting $z^2_0=g(t), z^2_1=\dot{g}(t),\text{ and }z^2_2=\ddot{g}(t)$. We then find that
\begin{equation}
\bar{\gamma}^G=\left\langle\theta^1_0,\theta^1_1,d\epsilon-\frac{2(\dot{g})^2-\dot{g}z^1_1-z^1_0}{\dot{g}z^1_0-2}\,dt\right\rangle.
\end{equation}
Additionally, if we compute $\bar{p}(z)$, then we can determine $Q_{\sigma_1+1}=Q_{3}$. Using Lemma \ref{reduced}, we obtain
\begin{equation}
\begin{aligned}
\bar{p}(z)&=\frac{2(\ddot{g})^2-z^1_0}{\dot{g}z^1_0-2}-\bar{\sD}_t\ln(\dot{g}z^1_0-2)+\frac{\ddot{g}z^1_0}{\dot{g}z^1_0-2}\\
\,&=\frac{2(\ddot{g})^2+(\ddot{g}-1)z^1_0}{\dot{g}z^1_0-2}-\bar{\sD}_t\ln(\dot{g}z^1_0-2),\\
\end{aligned}
\end{equation}
and so by Lemma \ref{EOQ} we can conclude that
\begin{equation}
\begin{aligned}
Q_{3}&=-\bar{\sE}_2(\bar{p})\\
\,&=\frac{2\dot{g}\ddot{g}^2+2(\ddot{g}-1)}{(\dot{g}z^1_0-2)^2}.
\end{aligned}
\end{equation}
Hence the choice of partial contact curve that permits $\bar{\gamma}^G$ to be ESFL is generic. To see this, notice that $g(t)$ cannot be a solution of $\dot{g}\ddot{g}^2+\ddot{g}-1=0$. Indeed, if $g(t)$ is a solution to $\dot{g}\ddot{g}^2+\ddot{g}-1=0$, then $Q_{\sigma_1+1}=0$. But this contradicts that $\bar{\mathcal{H}}_G$ is bracket generating because of (\ref{Qn+1 not 0}) in Proposition \ref{PDE form}.
\subsection{Final Thoughts and Future Work}
We conclude this thesis by highlighting the importance of the results in Chapter 4 in relation to the theories of CFL and EDFL control systems. Recall that Theorem \ref{CFL is EI} says that CFL systems are explicitly integrable; however, Proposition \ref{EI iff EDFL} from Chapter 1 says that a control system is explicitly integrable if and only if the control system is EDFL. Therefore it is sufficient to prove that a control system is CFL to determine that it is EDFL. In fact, given a CFL control system, \cite{Cascade2} gives an explicit way to produce an EDF linearization of the associated contact sub-connection. 

In light of the relationship between CFL and EDFL control systems, we then realize the importance of Theorem \ref{sufficiency}. It produces an entire class of control systems that are \textit{known} to be EDFL. Furthermore, Theorem \ref{necessity} imposes a strong necessary condition on a contact sub-connection to have an ESFL partial contact curve reduction and gives an indication about which contact coordinates should be reduced to arrive at an ESFL reduced system. These are both important results for the goal of completely classifying all control systems with symmetry that are EDFL. Thus, with the main reults of this thesis in mind, we are suddenly presented with a number of natural questions and conjectures to investigate in order to further complete a general theory of CFL control systems. 
\begin{Q}
Can Theorems \ref{necessity} and \ref{sufficiency} be extended to necessary and sufficient conditions on the unreduced function $p(t,z)$ in the contact sub-connection? Or stated differently, can we classify \textit{all} functions $p(t,z)$ that correspond to EDFL contact sub-connections?  
\end{Q}
The author believes this can be achieved by exploring truncated versions of more identities and operators from the calculus of variations. Namely, establishing an insightful relationship between the reduced truncated Euler operator $\bar{\sE}_{\sigma_1}$ and the non-reduced truncated Euler operator $\sE_{\sigma_1}$ in a similar manner to Lemma \ref{reduced} for the truncated total derivative operator. 
\begin{Q}
At least in the case of codimension 1 contact curve redutions, can the results of this thesis be extended to similar theorems for control systems with control admissible symmetry groups of dimension greater than 1? 
\end{Q}
The author believes this is also possible. Careful examination of procedure \textbf{contact} ought to reveal any additional obstructions to ESF linearizability of the reduced system. This will also be the content of future work by the author. 
\begin{conj}
If a contact sub-connection is ESFL via a codimension $s>1$ partial contact curve reduction, then the contact sub-connection must have a codimension 1 partial contact curve reduction that is ESFL. 
\end{conj}
This conjecture is of central importance to the theory of CFL control systems. If true, the conjecture would mean that classifying all contact sub-connections that possess an ESFL partial contact curve reduction would be reduced to understanding the phenomenom of ESFL partial contact curve reductions to control systems of a single input. One example for which this conjecture holds true is the BC system. If one further reduces the reduced contact sub-connection (\ref{BC reduced}) with $z_0^2=g(t), z_1^2=\dot{g}(t)$ for an arbitrary $g(t)$ then the resulting control system with one control is ESFL.  
Additionally, the author has recently had some insight about the following: 
\begin{conj}
In the case that $\dim G=1$, if a contact sub-connection is EDFL by partial prolongation, then it must admit an ESFL codimension 1 partial contact curve reduction.  
\end{conj}
The key idea here has to do with the fundamental bundle from procedure \textbf{contact}. 

One can also investigate other aspects of CFL control systems related to a control admissible symmetry group. Interestingly, there appears to be little connection between the algebraic structure of the control admissible symmetry groups and whether or not control systems are CFL. However, the \textit{representation} of a Lie group of control admissible symmetries in the diffeomorphism group of the manifold of time, states, and controls for the control system \textit{does} appear to have an impact. This is seen in Theorem \ref{lifted solutions} in Chapter 3, where the equation of Lie type is constructed precisely from the \textit{action} of $G$ on $M$. From this perspective, the explicit form of the contact sub-connection in Theorem \ref{sufficiency} seems to hint at some kind of prolongation structure for the action of $G$ on $M$. The author hopes that research in this direction will help shed light on the following conjecture. 
\begin{conj}
If a control system with control admissible symmetry group is EDFL, then it must be CFL. 
\end{conj}

Although algebraic properties of the control admissible symmetry group seem to be mostly irrelevant, one algebraic structure that is known to play a role is solvability. This property eases finding first integrals for the group action, for example. Another interesting structure to investigate is the relationship between two CFL reductions of the same control system $\omega$ with control admissible symmetry groups $G$ and $H$ such that $H<G$. If $H$ is normal in $G$, then $\omega/G=(\omega/H)(G/H)$, which is found in \cite{AndersonFelsBacklund}. 
\begin{Q}
Let $\omega$ be a control system with two control admissible symmetry groups $G$ and $H$ such that $H<G$. If $\omega/G$ and $\omega/H$ are ESFL, then what can be said about ESFL partial contact curve reductions of $\gamma^H$ and $\gamma^G$? 
\end{Q}
The author intends to present some results concerning this question in a forthcoming work. 

One can also consider extensions of CFL theory of control systems to cases in which quotient systems are not ESFL. 
\begin{Q}
Assume that $\omega$ is a control system with 2 controls and that it possesses a control admissible symmetry group $G$. If the dimension of $G$ is such that the quotient system $\omega/G$ is a control system of 3 states and 2 controls, then the work of \cite{Wilkens3s2cEquiv} allows one to find a normal form for $\omega/G$ via a SF transformation. Can one then construct a theory analagous to that of CFL systems? That is, can one construct a ``contact" sub-connection $\delta^G$ along with some variation of ESFL contact curve reductions of $\delta^G$? 
\end{Q} 
Naturally, this leaves behind the question of explicit integrability, but in principle, if one can perform such a ``generalized" CFL process, then further avenues of control system classification present themselves.

The questions and conjectures above are by no means an exhaustive list of possible directions for future research. One could envision studying global phenomena or incorporating topological questions along the lines of $h$-principles, for example. One may even be able to build analagous versions of CFL theory for \textit{PDE} control systems or combine the theory of CFL system with mixed discrete or stochastic control processes, although these latter subjects are (currently) beyond the expertise of the author. 

Overall, this thesis presents a thorough treatment of the geometric theory of cascade feedback linearizable control systems. Furthermore, it adds to that theory by presenting new examples, new theorems, as well as brand new connections to operators that are related to the calculus of variations. The theory of CFL control systems is still far from complete but seems to offer many different avenues of interesting research. Finally, the author hopes that this theory will find itself useful in applications to concrete scientific and engineering problems.

\bibliographystyle{plain}	
\nocite{*}		
\bibliography{CtrlRefThesis}		

\end{document}